\documentclass[a4paper]{amsart}

\usepackage{geometry}
\usepackage{verbatim}
\usepackage{enumitem}

\usepackage{amssymb}
\usepackage[mathscr]{euscript}
\usepackage{MnSymbol}
\usepackage{wasysym}
\usepackage{todonotes}
\usepackage{stmaryrd}
\usepackage{tikz-cd}
\usepackage{nicefrac}
\usepackage{bbm}
\tikzcdset{arrow style=math font}

\DeclareMathOperator{\id}{id}
\DeclareMathOperator{\Map}{Map}
\DeclareMathOperator{\Hom}{Hom}
\DeclareMathOperator{\map}{map}

\DeclareMathOperator{\THH}{THH}
\DeclareMathOperator{\TC}{TC}

\DeclareMathOperator{\TR}{TR}
\DeclareMathOperator{\HH}{HH}

\DeclareMathOperator{\colim}{colim}
\DeclareMathOperator{\Fun}{Fun}
\DeclareMathOperator{\can}{can}

\DeclareMathOperator{\cofib}{cofib}

\DeclareMathOperator{\BMod}{BMod}

\DeclareMathOperator{\Fin}{Fin}
\DeclareMathOperator{\Free}{Free}
\DeclareMathOperator{\Trunc}{Trunc}
\DeclareMathOperator{\FinTrunc}{FinTrunc}
\DeclareMathOperator{\Path}{Path}
\DeclareMathOperator{\CycBMod}{CycBMod}

\DeclareMathOperator{\act}{act}
\DeclareMathOperator{\fgen}{fgen}
\DeclareMathOperator{\qfgen}{qfgen}
\DeclareMathOperator{\QFin}{QFin}
\DeclareMathOperator{\Span}{Span}

\newcommand{\op}{\mathrm{op}}
\newcommand{\prop}{\mathrm{prop}}
\newcommand{\add}{\mathrm{add}}
\newcommand{\vadd}{\mathrm{vadd}}

\DeclareMathOperator{\hex}{\varhexagon}

\newcommand{\Sp}{\mathrm{Sp}}
\newcommand{\Spaces}{\mathcal{S}}
\newcommand{\cS}{\mathcal{S}}
\newcommand{\Alg}{\mathrm{Alg}}

\newcommand{\Cat}{\cat{C}\mathrm{at}}
\newcommand{\CycSp}{\mathrm{CycSp}}

\newcommand{\PgcSp}{\mathrm{PgcSp}}

\renewcommand{\S}{\mathbb{S}}
\newcommand{\T}{\mathbb{T}}
\newcommand{\N}{\mathbb{N}_{>0}}
\newcommand{\Z}{\mathbb{Z}}

\newcommand{\Q}{\mathbb{Q}}
\newcommand{\E}{\mathbb{E}}
\newcommand{\cat}[1]{\mathscr{#1}}
\newcommand{\h}{\mathrm{h}}
\newcommand{\tate}{\mathrm{t}}
\newcommand{\B}{\mathrm{B}}
\newcommand{\ETG}{\widetilde{\EG}}
\newcommand{\EG}{\mathrm{E}\mathcal{P}}

\usepackage{hyperref}

\newtheorem{theorem}{Theorem}[section]
\newtheorem{proposition}[theorem]{Proposition}
\newtheorem{lemma}[theorem]{Lemma}
\newtheorem{corollary}[theorem]{Corollary}

\newtheorem{theoremvar}{Theorem}
\newenvironment{theorem*}[1]{\renewcommand{\thetheoremvar}{#1}\theoremvar}{\endtheoremvar}

\theoremstyle{definition}
\newtheorem{definition}[theorem]{Definition}

\newtheorem{remark}[theorem]{Remark}
\newtheorem{example}[theorem]{Example}

\newtheorem{notation}[theorem]{Notation}
\newtheorem{warning}[theorem]{Warning}
\newtheorem{construction}[theorem]{Construction}

\title{Polygonic spectra and TR with coefficients}

\author{Achim Krause}
\address{Fachbereich Mathematik und Informatik, Universität Münster, Germany}
\email{krauseac@uni-muenster.de}

\author{Jonas McCandless}
\address{Max Planck Institute for Mathematics, Bonn, Germany}
\email{mccandless@mpim-bonn.mpg.de}

\author{Thomas Nikolaus}
\address{Fachbereich Mathematik und Informatik, Universität Münster, Germany}
\email{nikolaus@uni-muenster.de}

\begin{document}

\begin{abstract}
We introduce the notion of a polygonic spectrum which is designed to axiomatize the structure on topological Hochschild homology $\THH(R,M)$ of an $\mathbb{E}_1$-ring $R$ with coefficients in an $R$-bimodule $M$. For every polygonic spectrum $X$, we define a spectrum $\TR(X)$ as the mapping spectrum from the polygonic version of the sphere spectrum $\mathbb{S}$ to $X$. In particular if applied to $X = \THH(R,M)$ this gives a conceptual definition of $\TR(R,M)$.

Every cyclotomic spectrum gives rise to a polygonic spectrum and we prove that TR agrees with the classical definition of TR in this case. We construct Frobenius and Verschiebung maps on $\TR(X)$ by exhibiting $\TR(X)$ as the $\Z$-fixedpoints of a quasifinitely genuine $\Z$-spectrum. The notion of quasifinitely genuine $\Z$-spectra is a new notion that we introduce and discuss inspired by a similar notion over $\Z$ introduced by Kaledin. Besides the usual coherences for genuine spectra, this notion additionally encodes that $\TR(X)$ admits certain infinite sums of Verschiebung maps.
\end{abstract}

\maketitle

\setcounter{tocdepth}{2}
\tableofcontents

\section{Introduction} \label{section:introductionpolygonic}
The purpose of this paper is to introduce and study the $\infty$-category of polygonic spectra as an analogue of the $\infty$-category of cyclotomic spectra in the sense of~\cite{NS18}. The point being that the $\infty$-category of polygonic spectra provides the natural framework for studying THH and TR with coefficients. To motivate our definition of a polygonic spectrum, recall that the topological Hochschild homology $\THH(R,M)$ of an $\E_1$-ring $R$ with coefficients in an $R$-bimodule $M$ is defined as the geometric realization of a simplicial spectrum which we informally depict by
\[
\begin{tikzcd}
	\cdots \arrow[yshift=0.5ex]{r} \arrow[yshift=-0.5ex]{r} \arrow[yshift=1.5ex]{r} \arrow[yshift=-1.5ex]{r} & M \otimes R \otimes R \arrow{r} \arrow[yshift=1ex]{r} \arrow[yshift=-1ex]{r} & M \otimes R \arrow[yshift=0.5ex]{r} \arrow[yshift=-0.5ex]{r} & M.
\end{tikzcd}
\]
For instance, the first pair of maps are given by the left and right action of $R$ on $M$, respectively. In the special case when $R = M$, the simplicial spectrum above admits extra structure, namely an action of the cyclic group $C_n$ on the $(n-1)$st term by cyclic permutations. More precisely, the simplicial spectrum extends to a cyclic spectrum as introduced by Connes~\cite{Con83}, which implies that the topological Hochschild homology $\THH(R) = \THH(R, R)$ admits an action of the circle group $\T$. Additionally, there is a map of spectra with $\T$-action
\[
	\varphi_p : \THH(R) \to \THH(R)^{\tate C_p}
\]
for every prime $p$, where we consider the cyclic group $C_p$ as the roots of unity and the Tate construction carries the residual $\T \simeq \T/C_p$-action. This structure exhibits $\THH(R)$ as a cyclotomic spectrum in the sense of~\cite{NS18}, and it is the cyclotomic structure which facilitates the construction of TR and TC as in~\cite{McC21,NS18}. For an arbitrary $R$-bimodule $M$, the simplicial spectrum defining $\THH(R, M)$ no longer extends to a cyclic spectrum and the main goal of this paper is to isolate the key structural properties of $\THH(R, M)$ which allows us to extract further invariants such as TR.

\subsection{Polygonic spectra and TR}
Our main insight is that while $\THH(R, M)$ does not admit an action of the circle group, it still carries a variant of the map $\varphi_p$ which takes the form
\[
	\varphi_p : \THH(R, M) \to \THH(R, M^{\otimes_R p})^{\tate C_p},
\]
where $M^{\otimes_R p}$ denotes the $R$-bimodule obtained as the $p$-fold tensor product $M \otimes_R \cdots \otimes_R M$. The fact that $\THH(R, M^{\otimes_R p})$ carries an action of the cyclic group $C_p$ induced by cyclic permutations of the tensor factors is rather surprising since these permutations are not maps of $R$-bimodules. We carefully construct these structures in this paper. We will prove that the collection of spectra $\{\THH(R, M^{\otimes_R n})\}_{n}$ carries the structure of a polygonic spectrum which we define as follows:

\begin{definition} 
A polygonic spectrum consists of the following data:
\begin{enumerate}[leftmargin=2em, topsep=5pt, itemsep=5pt]
	\item A collection of spectra $\{X_n\}_{n \geq 1}$, where each $X_n$ admits an action of $C_n$.
	\item For every prime number $p$ and $n \geq 1$, the datum of a $C_n$-equivariant map of spectra
	\[
		\varphi_{p, n} : X_n \to (X_{pn})^{\tate C_p},
	\]
	where the Tate construction carries the residual $C_{pn}/C_p \simeq C_n$-action.
\end{enumerate}
Let $\PgcSp$ denote the $\infty$-category of polygonic spectra (cf. Definition~\ref{definition:infty_cat_polygonic}).
\end{definition}

We think of a polygonic spectrum as a multiobject variant of a cyclotomic spectrum indexed by the $n$-polygons rather than the circle. Indeed, the notion of a cyclotomic spectrum is encoded by a single spectrum acted upon by the orientation preserving symmetries of the circle and with Frobenius maps corresponding to the self-coverings of the circle. Instead, we think of a polygonic spectrum as a rule which assigns to every $n$-polygon a spectrum with $C_n$-action obtained by acting with the orientation preserving symmetries on the $n$-polygon and the polygonic Frobenius maps correspond to the self-coverings of the polygon. 

\begin{example} \label{introduction:examples}
We discuss some  examples of polygonic spectra.

\begin{enumerate}[leftmargin=2em, topsep=5pt, itemsep=5pt]
	\item Every cyclotomic spectrum $X$ can be regarded as a constant polygonic spectrum by $X_n = X$, where we consider $X$ as a spectrum with $C_n$-action by restriction for each $n$. The polygonic structure maps $\varphi_{p, n}$ are obtained by regarding the cyclotomic Frobenius map $\varphi_p : X \to X^{\tate C_p}$ as a $C_n$-equivariant map for each $n$. As a special case, we may regard the sphere spectrum equipped with the trivial cyclotomic structure as a polygonic spectrum. More generally, every spectrum can be regarded as a cyclotomic spectrum with trivial $\T$-action thus also as a constant polygonic spectrum, where all the $C_n$-actions are trivial.

	\item The topological Hochschild homology of an $\E_1$-ring $R$ with coefficients in an $R$-bimodule $M$ canonically admits the structure of a polygonic spectrum with
	\[
		\THH(R, M)_n = \THH(R, M^{\otimes_R n})
	\]
	for every $n \geq 1$. This is the motivating example of a polygonic spectrum and a substantial amount of this paper is devoted to constructing $\THH(R, M)$ as a polygonic spectrum (cf.~\textsection\ref{section:polygonicTHH}). We will address our strategy for doing so in more detail below (cf.~\textsection\ref{introduction:THHwithcoef_subsec}).

	\item For every polygonic spectrum $X$, we can form the shifted polygonic spectrum $\mathrm{sh}_n X$ defined by the formula $(\mathrm{sh}_n X)_k = X_{kn}$ (cf. Example~\ref{example:polygonic_shift}). The construction $X \mapsto \mathrm{sh}_n X$ determines a functor $\mathrm{sh}_n : \PgcSp \to \PgcSp$, which will play an important role throughout this paper. For instance, we have that $\mathrm{sh}_n \THH(R, M) \simeq \THH(R, M^{\otimes_R n})$ for every $n \geq 1$.
\end{enumerate}
\end{example}

The point of introducing the $\infty$-category of polygonic spectra is that it allows for a construction of TR by a suitable corepresentability construction as in~\cite{McC21}.

\begin{definition} \label{definition:definition_of_TR}
If $X$ is a polygonic spectrum, then $\TR(X)$ is defined by
\[
	\TR(X) = \map_{\PgcSp}(\S, X),
\]
where the sphere spectrum $\S$ is considered as a constant polygonic spectrum. Equivalently, the spectrum $\TR(X)$ is equivalent to the equalizer of the following diagram of spectra
\[
\begin{tikzcd}
	\displaystyle\prod_{n \geq 1} X_n^{\h C_n} \arrow[yshift=0.7ex]{r} \arrow[yshift=-0.7ex]{r} & \displaystyle\prod_{p \in \mathbb{P}} \prod_{n \geq 1} (X_{pn}^{\tate C_p})^{\h C_n}
\end{tikzcd}
\]
where the maps are induced by the various $\varphi_{p,n}^{\h C_n}$ and the canonical maps $X_{pn}^{\h C_{pn}} \to (X_{pn}^{\tate C_p})^{\h C_n}$.
\end{definition}

This definition is comparable to the construction of TC obtained by the third author with Scholze in~\cite{NS18}, as a functor corepresented by the sphere spectrum on the $\infty$-category of cyclotomic spectra. The equivalence between the mapping spectrum and the equalizer description in Definition~\ref{definition:definition_of_TR} above follows from~\cite[Proposition II.1.5]{NS18}.

In fact, TC can not be defined on the $\infty$-category of polygonic spectra which makes TR the correct construction to consider if one is interested in studying topological Hochschild homology with coefficients. We define TR of an $\E_1$-ring $R$ with coefficients in an $R$-bimodule $M$ by
\[
	\TR(R, M) = \TR(\THH(R, M)).
\]
We remark that Lindenstrauss--McCarthy also construct a version of TR of a functor with smash product with coefficients in a bimodule using point set constructions (cf.~\cite{LM12}). In the case when both $R$ and $M$ are discrete, the abelian group $\TR_{0}(R, M)$ is isomorphic to the Witt vectors with coefficients $W(R, M)$ introduced in~\cite{DKNP20} as will be established in~\cite{DKNP22b}. 

As one of our main results, we analyze the relationship between the $\infty$-category of cyclotomic spectra and the $\infty$-category of polygonic spectra. Let $i : \CycSp \to \PgcSp$ denote the functor which regards every cyclotomic spectrum as a constant polygonic spectrum as discussed above. In~\textsection\ref{subsection:pgc_cycsp}, we obtain the following result (cf. Theorem~\ref{theorem:adjoints_pgcsp_cycsp}).

\begin{theorem*}{A} \label{introduction:theoremA}
The canonical functor $i : \CycSp \to \PgcSp$ admits both adjoints 
\[
\begin{tikzcd}[column sep = large]
	\PgcSp \arrow[yshift=1.6ex]{r}{L} \arrow[yshift=-1.6ex, swap]{r}{R} & \CycSp \arrow["i" description]{l}
\end{tikzcd}
\]
such that there are natural equivalences of cyclotomic spectra
\begin{align*}
	Li(X) &\simeq X \otimes \widetilde{\THH}(\S[t])  \\
	Ri(X) &\simeq \Omega X \,\widehat{\otimes}\, \widetilde{\THH}^{\mathrm{cont}}(\S[\![t]\!]) := \Omega \varprojlim( X \otimes \widetilde{\THH}(\S[t]/t^n))
\end{align*}
for every cyclotomic spectrum $X$. 
\end{theorem*}

We refer to Notation~\ref{notation:reducedTHH_inverselimit} for an explanation of the notation we use in Theorem~\ref{introduction:theoremA} above. Note that the limit in the description of $Ri(X)$ is a limit in the $\infty$-category of cyclotomic spectra. Since these are not formed underlying this might be a complicated object. 
The proof of Theorem~\ref{introduction:theoremA} is an easy consequence of explicit formulas for the induction and coinduction functors. 

Nonetheless, Theorem~\ref{introduction:theoremA} allows us to express TR of a cyclotomic spectrum $X$ in two ways. Firstly, using the formula for left adjoint $L : \PgcSp \to \CycSp$, we find that
\[
	\TR(X) = \map_{\PgcSp}(i(\S), i(X)) \simeq \map_{\CycSp}(Li(\S), X) \simeq \map_{\CycSp}(\widetilde{\THH}(\S[t]),X) 
\]
which recovers a result first proved by Blumberg--Mandell~\cite{BM16} and later revisited by the second author in~\cite{McC21} (cf. Proposition~\ref{proposition:comparison_mcc}). Using the formula for the right adjoint $R$, we find that 
\[
	\TR(X) = \map_{\PgcSp}(i(\S), i(X)) \simeq \map_{\CycSp}(\S, Ri(X)) \simeq  \Omega \varprojlim_n  \TC(X \otimes \widetilde{\THH}(\S[t]/t^n))
\]
which recovers the result obtained by the second author in~\cite{McC21}, which in turn generalizes work of Hesselholt~\cite{Hes96} and Betley--Schlichtkrull~\cite{BS05} (cf. Theorem~\ref{theorem:curves_arbitrary_truncationset}).
We further obtain a refinement of this result for which we need a slight variant of the $\infty$-category of polygonic spectra.

\begin{remark}
In~\textsection\ref{subsection:infty_cat_pgcsp}, we will consider a $T$-typical variant of the $\infty$-category of polygonic spectra for an arbitrary truncation set $T \subseteq \N$. More precisely, we let $\PgcSp_T$ denote the full subcategory of $\PgcSp$ spanned by those polygonic spectra $X$ which satisfy that $X_n = 0$ when $n \notin T$. Every polygonic spectrum $X$ gives rise to a $T$-typical polygonic spectrum $X_{\mid T}$ by letting the necessary $X_n$ be zero. We define a $T$-typical variant of TR by
\[
	\TR_T(X) = \map_{\PgcSp_T}(\S_{\mid T}, X)
\]
for every $T$-typical polygonic spectrum $X$.
\end{remark}

Using $T$-typical polygonic spectra and a $T$-typical refinement of Theorem \ref{introduction:theoremA}, we will prove the following result (cf. Theorem~\ref{theorem:curves_arbitrary_truncationset}).

\begin{theorem*}{B} \label{introduction:theoremB}
Let $T$ denote a truncation set. There is a natural equivalence of spectra
\[
	\TR_T(X) \simeq \Omega\TC\Big( \prod_{t \in T} X \otimes \Sigma^\infty_+(\T/C_t) \Big).
\]
for every cyclotomic spectrum $X$.
\end{theorem*}

\subsection{Frobenius and Verschiebung maps}
Currently we have constructed a spectrum $\TR(R, M)$ for every $\E_1$-ring $R$ and $R$-bimodule $M$. The close relationship between TR and the Witt vectors suggests that $\TR(R, M)$ should be equipped with additional structure encoding Verschiebung and Frobenius maps. Indeed, we construct an action of $C_n$ on $\TR(R, M^{\otimes_R n})$ and a pair of maps
\begin{align*}
	F_k &: \TR(R, M) \to \TR(R, M^{\otimes_R k}) \\
	V_k &: \TR(R, M^{\otimes_R k}) \to \TR(R, M)
\end{align*}
which are subject to certain coherences. For instance, we have that $F_n \circ F_m \simeq F_{nm}$ while $V_n \circ V_m \simeq V_{nm}$, and the maps $F_n$ and $V_n$ canonically refine to a pair of maps
\begin{align*}
	F_k &: \TR(R, M) \to \TR(R, M^{\otimes_R k})^{\h C_k} \\
	V_k &: \TR(R, M^{\otimes_R k})_{\h C_k} \to \TR(R, M)
\end{align*}
Finally, we have that $F_n \circ V_n$ is equivalent to the $C_n$-norm map of $\TR(R, M^{\otimes_R n})$, and $F_n \circ V_m \simeq V_m \circ F_n$ if $m$ and $n$ are coprime. We will formalize this additional structure on $\TR(R, M)$ using the formalism of genuine equivariant homotopy theory, more precisely the formalism of finitely genuine $\Z$-spectra\footnote{We use the name `finitely genuine $\Z$-spectra' for those spectra with $\Z$-action that admit genuine fixed points for cofinite subgroups $n\Z \subseteq \Z$. In classical terminology this would be called genuine $\Z$-spectra for the family of cofinite subgroups. These are also equivalent to genuine $\widehat\Z$-spectra in the language of~\cite[Example B]{Bar17}.}. We will provide a brief summary of this in the formalism of spectral Mackey functors~\cite{Bar17} in Definition~\ref{definition:finitely_genuine_G_spectra}.

\begin{theorem*}{C} \label{introduction:theoremB}
Let $M$ denote a bimodule over an $\E_1$-ring $R$. There is a finitely genuine $\Z$-spectrum  denoted $\underline{\TR}(R, M)$ whose fixedpoints for the subgroup $n\Z$ of $\Z$ are given by
\[
	\underline{\TR}(R, M)^{n\Z} \simeq \TR(R, M^{\otimes_R n}).
\]
\end{theorem*}

The finitely genuine $\Z$-spectrum $\underline{\TR}(R, M)$ encodes the structure described above in a coherent fashion. Namely, the Frobenius maps are given by inclusions of fixedpoints of this finitely genuine $\Z$-spectrum, and the Verschiebung maps are given by the transfer maps. In fact, we will more generally construct a finitely genuine $\Z$-spectrum $\underline{\TR}(X)$ for every polygonic spectrum $X$. 

\subsection{Completeness}\label{comp}
In addition to the existence of the maps $V_k$ and $F_k$,  the ring of Witt vectors $W(R) = \pi_0\TR(R,R)$ carries a complete topology. Our goal now is to explain the analogue of this for $\TR(R,M)$. In the $p$-typical case the topology on the $p$-typical Witt vectors is the $V$-adic topology and completeness translates into 
\[
	W_{\langle p^\infty \rangle}(R) = \varprojlim_n W_{\langle p^\infty \rangle}(R)/V^n
\]
and it holds that $W_{\langle p^\infty \rangle}(R)/V = R$. In particular the topology is entirely encoded by the algebraic operation $V$. The analogs of these properties for $p$-typical TR was obtained by Antieau and the third author in~\cite{AN20}. 
However, the correct notion of completeness on integral TR and the integral Witt vectors is more subtle in the integral situation since it necessarily involves forming certain infinite sums of Verschiebung maps.

\begin{example}
If $R$ is a commutative ring, then the underlying additive group of the ring of big Witt vectors admits the following description
\[
(W(R), +)  = (1 + t R[\![ t ]\!], \cdot),
\]
where the topology on $W(R)$ corresponds to the the $t$-adic topology on the right hand side. The $n$th Verschiebung is given by sending a power series $p(t)$ to the series $p(t^n)$. The point is that every collection of Verschiebung maps $\{V_{n_i}\}_{i \geq 1}$ indexed by a sequence of positive integers $\{n_i\}_{i \geq 1}$ with the property that $n_i \to \infty$ for $i \to \infty$ is summable, meaning that the infinite sum $\sum_{i = 1}^\infty V_{n_i}$ converges. This can be easily seen translating this into an infinite product of power series which converges $t$-adically.   
More precisely, the dotted extension $V_{\{n_i\}}$ exists in the following diagram
\[
\begin{tikzcd}[column sep = large]
	\displaystyle\bigoplus_{i \geq 1} W(R) \arrow[hook]{d} \arrow{r}{\sum_{i \geq 1} V_{n_i}} & W(R) \\
	\displaystyle\prod_{i \geq 1} W(R) \arrow[dashed, swap]{ur}{V_{\{n_i\}}} & 
\end{tikzcd}
\]
for every such sequence $\{n_i\}_{i \geq 1}$ as above. Moreover, we can recover $R$ as the quotient of $W(R)$ by the image of $\sum_{i = 2}^\infty V_{i}$.
In particular using similar sums the $t$-adic topology can be recovered from the maps  $\sum_{i = 1}^\infty V_{n_i}$. Therefore, instead of considering $W(R)$ as a topological abelain group we may equivalently consider it as being equipped with these infinite sums of Verschiebung maps, subject to some axioms (e.g. ensuring that the induced topology is complete and all the maps are continuous). 
\end{example}

We construct a map of spectra
\[
	V_{\{n_i\}} : \prod_{i \geq 1} \TR(R, M^{\otimes_R n_i}) \to \TR(R, M)
\]
for every sequence of positive integers $\{n_i\}$ with $n_i \to \infty$ for $i \to \infty$ which fits into a commutative diagram as above. In fact, we will explain below how one can recover $\THH(R, M)$ from $\TR(R, M)$ by forming a suitable cofiber of all these infinite sums of Verschiebung maps on $\TR(R, M)$. This is the integral analog of the $p$-typical feature that $\TR_{\langle p^\infty \rangle}(R)/V \simeq \THH(R)$ (cf.~\cite[Lemma 3.19]{AN20}).\\

In order to formalize the properties of these infinite sums of Verschiebung maps on TR, we will use an extension of the theory of finitely genuine $\Z$-spectra and prove a slightly stronger version of Theorem~\ref{introduction:theoremB}. More precisely, we will use the formalism of quasifinitely genuine $\Z$-spectra, that we will introduce based on the analogous notion for chain complexes introduced by Kaledin~\cite{Kal14}\footnote{In~\cite{Kal14}, Kaledin refers to the analogue of our quasifinitely genuine $\Z$-spectra as a $\Z$-Mackey profunctors. We will give an overview of this formalism in~\textsection\ref{section:polygonicTR_completeCartier} adapted to our setting.}. The idea is to replace the category of finite $\Z$-sets with the category of quasifinite $\Z$-sets in the sense of Dress--Siebeneicher~\cite{DS88}. By considering certain spectral Mackey functor on the latter, this produces an $\infty$-category $\Sp^{\Z}_{\qfgen}$ of quasifinitely genuine $\Z$-spectra which is equipped with a forgetful functor of $\infty$-categories
\[
	\Sp^{\Z}_{\qfgen} \to \Sp^{\Z}_{\fgen},
\]
where the target denotes the $\infty$-category of finitely genuine $\Z$-spectra discussed above. The main difference between $\Sp^{\Z}_{\qfgen}$ and $\Sp_{\fgen}^{\Z}$ is that the former additionally encoporates certain infinite sums of transfer maps since the infinite coproduct of orbits
\[
	\coprod_{i \geq 1} \Z/n_i\Z
\]
exists as a quasifinite $\Z$-set whenever $\{n_i\}_{i \geq 1}$ is a sequence of positive integers with $n_i \to \infty$ for $i \to \infty$ as above. The functor above simply forgets these convergent sums of transfer maps and its left adjoint formally adjoins convergent sums of transfer maps indexed by the sequences above. We will discuss this in more detail in~\textsection\ref{subsection:qfgen_G}. For now, we state our main result:

\begin{theorem*}{C'} \label{introduction:TR_quasifinitelygenuine_geometricfixedpoints}
For every pair $(R, M)$ consisting of an $\E_1$-ring $R$ and an $R$-bimodule $M$, there is a quasifinitely genuine $\Z$-spectrum denoted $\underline{\TR}(R, M)$ whose fixedpoints spectrum is given by
\[
	\underline{\TR}(R, M)^{n\Z} \simeq \TR(R, M^{\otimes_R n})
\]
for every cofinite subgroup $n\Z$ of $\Z$. Furthermore, if $R$ and $M$ are connective, then the geometric fixedpoints spectrum is given by $\underline{\TR}(R, M)^{\Phi n\Z} \simeq \THH(R, M^{\otimes_R n})$ for every cofinite subgroup $n\Z$. 
\end{theorem*} 

The construction of $\TR(R, M)$ as a quasifinitely genuine $\Z$-spectrum answers a question raised by Kaledin in the introduction of~\cite{Kal14}. 
Besides the structure already present in the underlying finitely genuine $\Z$-spectrum of Theorem~\ref{introduction:theoremB}, this structure additionally encodes the infinite sums of Verschiebung maps as well as coherences between those and with the other structure. Informally, the final part of Theorem~\ref{introduction:TR_quasifinitelygenuine_geometricfixedpoints} asserts that if both $R$ and $M$ are connective, then $\THH(R, M)$ can be recovered from $\TR(R, M)$ by killing all the proper infinite sums of Verschiebung maps. This is the integral analog of the result that $\TR_{\langle p^\infty \rangle}(R)/V \simeq \THH(R)$ established by Antieau and the third author in~\cite{AN20}. 

The strategy of our proof of Theorem~\ref{introduction:TR_quasifinitelygenuine_geometricfixedpoints} is to construct $\underline{\TR}$ as a right adjoint of the functor
\[
	\mathrm{L} : \Sp^{\Z}_{\qfgen} \to \PgcSp
\]
carrying a quasifinitely genuine $\Z$-spectrum $X$ to the polygonic spectrum with $(\mathrm{L}X)_n = X^{\Phi n\Z}$, where $X^{\Phi n\Z}$ denotes the geometric fixedpoints with its canonical $C_n$-action for the cofinite subgroup $n\Z$ of $\Z$. An advantage of the $\infty$-category of quasifinitely genuine $\Z$-spectra is the observation due to Kaledin that $\mathrm{L}$ is conservative on the uniformly bounded objects (cf.~\cite[Proposition 3.5]{Kal14} and Proposition~\ref{proposition:geometric_fixedpoints_conservative}). Therefore, the proof of Theorem~\ref{introduction:TR_quasifinitelygenuine_geometricfixedpoints} amounts to calculating the geometric fixedpoints of TR of a uniformly bounded below polygonic spectrum, which we do in~\textsection\ref{subsection:qfgenZ_TR}. In fact, we will see that the adjunction $\mathrm{L} \dashv \TR$ induces an equivalence on bounded below objects. 

\begin{remark}
In upcoming work, we introduce the $\infty$-category of (integral) complete topological Cartier modules. We will show that this is equivalent to cyclotomic spectra using the ideas of this paper and employ this to identify the heart of the $t$-structure on the $\infty$-category of cyclotomic spectra (extending the work of Antieau and the third author~\cite{AN20} to the integral case). 
\end{remark}

\subsection{Topological Hochschild homology with coefficients} \label{introduction:THHwithcoef_subsec}
We return to the construction of topological Hochchild homology with coefficients as a polygonic spectrum as briefly indicated in Example~\ref{introduction:examples}. More precisely, we prove the following result (cf. Theorem~\ref{theorem:THH_polygonic}).

\begin{theorem*}{D} \label{introduction:theoremD}
Let $R$ denote an $\E_1$-ring and let $M$ denote an $R$-bimodule. Then the topological Hochschild homology $\THH(R, M)$ admits the structure of a polygonic spectrum with
\[
	\THH(R, M)_n = \THH(R, M^{\otimes_R n})
\]
for every $n \geq 1$, where $M^{\otimes_R n}$ is the $R$-bimodule defined by $M^{\otimes_R n} = \underbrace{M \otimes_R \cdots \otimes_R M}_{\text{$n$ times}}$.
\end{theorem*}

We end by explaining our strategy for proving Theorem~\ref{introduction:theoremD}. Inspired by the work of Kaledin~\cite{Kal15}, we proceed by studying topological Hochschild homology as a trace theory on a suitable $\infty$-category. Concretely, for each $n \in \N$, we introduce an $\infty$-category $\CycBMod_n$ whose objects are cyclic graphs whose vertices are labelled with $\E_1$-rings $R_0, \ldots, R_{n-1}$ and whose edges are labelled by $R_i$-$R_{i+1}$-bimodules $M_i$ for every $i$ counted modulo $n$. We picture an object of $\CycBMod_3$ by:
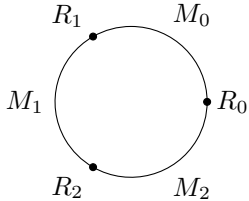
\[
\begin{tikzpicture}
	\filldraw (1,0) circle (1.3pt);
	\filldraw (-0.5, 0.86602) circle (1.3pt);
	\filldraw (-0.5, -0.86602) circle (1.3pt);

	\node at (1, 0) [right] {$R_0$};
	\node at (-0.5, 0.86602) [above left] {$R_1$};
	\node at (-0.5, -0.86602) [below left] {$R_2$};

	\node at (0.8, 1.15) {$M_0$};
	\node at (-1.4, 0) {$M_1$};
	\node at (0.8, -1.15) {$M_2$};

	\draw (1,0) arc[start angle = 0, end angle = 120, x radius = 1, y radius = 1];
	\draw (-0.5, 0.86602) arc[start angle = 120, end angle = 240, x radius = 1, y radius = 1];
	\draw (-0.5, -0.86602) arc[start angle = 240, end angle = 360, x radius = 1, y radius = 1];
\end{tikzpicture}
\]
The $\infty$-category $\CycBMod_n$ is constructed as the algebras of a suitable $\infty$-operad underlying an explicitly defined coloured operad which we discuss in~\textsection\ref{subsection:cyclicbimdouleoperad}. As we let $n$ vary\footnote{More precisely, the $\infty$-category $\Lambda^{\CycBMod}$ is the total space of a cocartesian fibration $\Lambda^{\CycBMod} \to \Lambda^{\op}$, where $\Lambda$ denotes the cyclic category. This will allow us to rotate the cyclic graphs which will be crucial to formulate the cyclic invariance of THH.}, these $\infty$-categories are neatly organized into another $\infty$-category $\Lambda^{\CycBMod}$ whose objects are pairs $(n, F)$ consisting of a natural number $n \in \N$ and an object $F$ of $\CycBMod_n$. This is a variant of the $\infty$-category $\Lambda^{\mathrm{st}}$ introduced by the third author in~\cite{HS19}, which parameterizes cyclic graphs labelled by stable $\infty$-categories and profunctors. For every $n$, we construct a functor
\[
	\THH : \CycBMod_n \to \Sp
\]
which is defined as the geometric realization of the cyclic bar construction associated to an object of $\CycBMod_n$. For instance, the functor carries the object of $\CycBMod_2$ depicted as follows
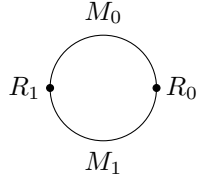
\[
\begin{tikzpicture}
	\filldraw (0.7,0) circle (1.3pt);
	\filldraw (-0.7,0) circle (1.3pt);

	\node at (0.7, 0) [right] {$R_0$};
	\node at (-0.7, 0) [left] {$R_1$};

	\node at (0, 1) {$M_0$};
	\node at (0, -1) {$M_1$};

	\draw (0.7,0) arc[start angle = 0, end angle = 360, x radius = 0.7, y radius = 0.7];
\end{tikzpicture}
\]
to the geometric realization of the simplicial spectrum informally depicted as follows
\[
\begin{tikzcd}
	\cdots \arrow[yshift=0.5ex]{r} \arrow[yshift=-0.5ex]{r} \arrow[yshift=1.5ex]{r} \arrow[yshift=-1.5ex]{r} & M_0 \otimes M_1 \otimes R_0 \otimes R_1 \otimes R_0 \otimes R_1 \arrow{r} \arrow[yshift=1ex]{r} \arrow[yshift=-1ex]{r} & M_0 \otimes M_1 \otimes R_0 \otimes R_1 \arrow[yshift=0.5ex]{r} \arrow[yshift=-0.5ex]{r} & M_0 \otimes M_1
\end{tikzcd}
\]
where for instance the first pair of the maps are given by $m_0 \otimes m_1 \otimes r_0 \otimes r_1 \mapsto m_0 r_1 \otimes m_1 r_0$ and $m_0 \otimes m_1 \otimes r_0 \otimes r_1 \mapsto r_0 m_0 \otimes r_1 m_1$, respectively. With this in mind, we prove the following result which appears as Theorem~\ref{theorem:HH_trace}. 

\begin{theorem*}{E} \label{introduction:theoremE}
The construction $(n, F) \mapsto \THH(F)$ refines to a functor of $\infty$-categories
\[
	\THH : \Lambda^{\CycBMod} \to \Sp
\]
which satisfies the trace property, which means that it carries cocartesian morphisms in $\Lambda^{\CycBMod}$ to equivalences of spectra.
\end{theorem*}

As a consequence of Theorem~\ref{introduction:theoremE}, the following span of morphisms in $\Lambda^{\CycBMod}$
\[
	(R_0, M_0 \otimes_{R_1} M_1) \leftarrow (R_0, M_0; R_1, M_1) \rightarrow (R_1, M_1 \otimes_{R_0} M_0)
\]
is carried to an equivalence of spectra
\[
	\THH(R_0, M_0 \otimes_{R_1} M_1) \simeq \THH(R_1, M_1 \otimes_{R_0} M_0)
\]
which expresses the cyclic invariance of topological Hochschild homology with coefficients. Furthermore, we endow $\THH(R, M^{\otimes_R n})$ with a $C_n$-action by rotating $(R, M; \ldots; R, M)$ as an object of $\Lambda^{\CycBMod}$ and using the trace property to transport the $C_n$-action to
\[
	\THH(R, M^{\otimes_R n}) \simeq \THH(R, M; \ldots; R, M).
\]
Finally, we construct the polygonic structure maps on $\THH(R, M)$ using the functoriality of the Tate diagonal on finite free $C_p$-sets as discussed in~\cite[\textsection III.3]{NS18}. Alternatively, one can employ the work of Lawson~\cite{Law21} to construct the polygonic structure on $\THH(R, M)$. We also mention that the $\infty$-categories $\CycBMod_n$ are closely related to the bypass categories appearing in the work of Berman~\cite{Ber22}. 

\subsubsection*{Acknowledgements}
The authors are grateful to Dmitry Kaledin for suggesting that the formalism of quasifinitely genuine $\Z$-spectra encodes the correct notion of completeness on TR. In fact, this paper is an attempt at understanding the foundational work of Kaledin on this formalism and the formalism of trace theories in the setting of higher algebra. 

The authors were funded by the Deutsche Forschungsgemeinschaft (DFG, German Research Foundation) -- Project-ID 427320536 -- SFB 1442, as well as under Germany’s Excellence Strategy EXC 2044 390685587, Mathematics M\"unster: Dynamics--Geometry--Structure. 
The second author is grateful for the hospitality and financial support offered by the Max Planck Institute for Mathematics in Bonn, where part of this work was carried out. 

\section{Polygonic spectra} \label{section:pgcsp}
The main goal of this section is to define the notion of a polygonic spectrum and the resulting $\infty$-category of polygonic spectra. This is a formalism designed to capture the additional structure on topological Hochschild homology with coefficients in a bimodule which will be addressed in~\textsection\ref{section:polygonicTHH}. The main result of this section is an explicit formula for the left and right adjoints of the functor which regards a cyclotomic spectrum as a constant polygonic spectrum (cf. Theorem~\ref{theorem:adjoints_pgcsp_cycsp}). We will use this to study TR on the $\infty$-category of polygonic spectra in~\textsection\ref{section:polygonicTR}. 

\subsection{Truncation sets} \label{definition:truncationset}
We will begin by reviewing the notion of a truncation set and give various examples which will play an important role throughout this exposition.

\begin{definition} \label{definition:truncationset}
A truncation set is a subset of $\N$ such that if $xy \in T$, then $x \in T$ and $y \in T$.
\end{definition}

Let $\Trunc$ denote the category of truncation sets whose morphisms are inclusions of truncation sets. Let $\FinTrunc$ denote the full subcategory of $\Trunc$ spanned by those truncation sets which are finite. We record some examples of truncation sets which will be important throughout.

\begin{example}
The following are examples of truncation sets.
\begin{enumerate}[leftmargin=2em, topsep=5pt, itemsep=5pt]
	\item For every $n \in \N$, the set $[n] = \{1, 2, \ldots, n\}$ is a truncation set.

	\item For every $n \in \N$, the set $\langle n \rangle = \{k \in \N \mid k|n\}$ is a truncation set. For instance, if $p$ is a prime, then $\langle p^n \rangle = \{1, p, \ldots, p^n\}$. Let $\langle p^\infty \rangle$ denote the truncation set given by $\langle p^\infty \rangle = \{1, p, p^2, \ldots\}$. 

	\item If $T$ is a truncation set, then $T/n = \{t \in T \mid nt \in T\}$ is again a truncation set for every $n$. 
\end{enumerate}
\end{example}

Finally, to establish notation, we introduce the following terminology.

\begin{notation}
Let $(\N, \mathrm{div})$ denote the set of positive integers regarded as a category via the divisibility relation. This means that there is a morphism $n \to m$ precisely if $n$ divides $m$. 
\end{notation}

\subsection{The $\infty$-category of polygonic spectra} \label{subsection:infty_cat_pgcsp}
In the following, we will let $T$ denote a truncation set as in Definition~\ref{definition:truncationset}. We define the notion of a $T$-typical polygonic spectrum and discuss the salient features of the accompanying $\infty$-category of $T$-typical polygonic spectra. 

\begin{definition} \label{definition:polygonic_spectrum}
A $T$-typical polygonic spectrum consists of the following data:
\begin{enumerate}[leftmargin=2em, topsep=5pt, itemsep=5pt]
	\item A $T$-indexed collection of spectra $\{X_t\}_{t \in T}$, where each $X_t$ is an object of $\Sp^{\B C_t}$.
	\item For every prime number $p$ and $t \in T/p$, the datum of a $C_t$-equivariant map of spectra
	\[
		\varphi_{p, t} : X_t \to (X_{pt})^{\tate C_p},
	\]
	where the Tate construction carries the residual $C_{pt}/C_p \simeq C_t$-action\footnote{We identify the residual $C_{pt}/C_p$-action with a $C_t$-action using the $p$th root isomorphism $C_t \to C_{pt}/C_p$.}.
\end{enumerate}
\end{definition}

The formalism of polygonic spectra is designed to capture the additional structure present on topological Hochschild homology with coefficients as discussed in~\textsection\ref{section:introductionpolygonic}. In~\textsection\ref{section:polygonicTHH}, we will prove that if $R$ is an $\E_1$-ring and $M$ is an $R$-bimodule, then the topological Hochschild homology $\THH(R, M)$ of $R$ with coefficients in $M$ canonically refines to a polygonic spectrum with
\[
	\THH(R, M)_n = \THH(R, M^{\otimes_R n})
\]
for every $n \in \N$. We invite the reader to keep this example in mind throughout the proceeding discussion which accordingly will be of a formal nature. We proceed by defining the accompanying $\infty$-category of $T$-typical polygonic spectra using the formalism of lax equalizers (cf.~\cite[\textsection II.1]{NS18}).

\begin{construction} \label{construction:F_p_functor}
For every prime number $p$, there is a functor of $\infty$-categories
\[
	\prod_{t \in T} \Sp^{\B C_t} \xrightarrow{F_p} \prod_{t \in T/p} \Sp^{\B C_t}
\]
which, for every $t \in T/p$, is induced by the following functor of $\infty$-categories
\[
	\prod_{s \in T} \Sp^{\B C_s} \xrightarrow{\mathrm{pr}} \Sp^{\B C_{pt}} \xrightarrow{(-)^{\tate C_p}} \Sp^{\B C_t}.
\]
The functor $F_p$ is determined by the construction $(X_s)_{s \in T} \mapsto (X_{pt})^{\tate C_p}$ whenever $t \in T/p$, and the Tate construction carries the residual $C_{pt}/C_p \simeq C_t$-action. 
\end{construction}

\begin{definition} \label{definition:infty_cat_polygonic}
The $\infty$-category of $T$-typical polygonic spectra is defined as the lax equalizer
\[
	\PgcSp_T = \mathrm{LEq} \Big(
	\begin{tikzcd}
		\prod_{t \in T} \Sp^{\B C_t} \arrow[yshift=0.7ex]{r} \arrow[yshift=-0.7ex]{r} & \prod_{p} \prod_{t \in T/p} \Sp^{\B C_t}
	\end{tikzcd}
	\Big),
\]
where the functors have $p$th components given by the identity and the functor $F_p$, respectively. 
\end{definition}

Unwinding Definition~\ref{definition:infty_cat_polygonic}, we see that an object of $\PgcSp_T$ is given by a $T$-typical polygonic spectrum in the sense of Definition~\ref{definition:polygonic_spectrum}. We consider some examples of interest.

\begin{example} \label{example:trivialtruncationset}
Consider the truncation set given by $S = \langle 1 \rangle$. In this case, the canonical functor
\[
	\PgcSp_{\langle 1 \rangle} \to \Sp
\]
is an equivalence of $\infty$-categories. 
\end{example}

\begin{example} \label{example:relationC2spectra}
For $T = \N$, we write $\PgcSp$ for the $\infty$-category of $\N$-typical polygonic spectra. According to Definition~\ref{definition:infty_cat_polygonic}, a polygonic spectrum consists of a spectrum $X_n$ with $C_n$-action for every $n \in \N$, together with a $C_n$-equivariant map of spectra
\[
	\varphi_{p, n} : X_n \to (X_{pn})^{\tate C_p}
\]
for every prime $p$ and $n$, where the Tate construction carries the residual $C_{pn}/C_p \simeq C_n$-action.
\end{example}

\begin{example}
Let $p$ be a prime number and consider the $\infty$-category $\PgcSp_{\langle p \rangle}$. An object of $\PgcSp_{\langle p \rangle}$ is a pair consisting of a spectrum $X$ and a spectrum with $C_p$-action $Y$, together with a map of spectra $X \to Y^{\tate C_p}$, so the $\infty$-category $\PgcSp_{\langle p \rangle}$ is equivalent to the following pullback
\[
\begin{tikzcd}
	\PgcSp_{\langle p \rangle} \arrow{r} \arrow{d} & \Sp^{\Delta^1} \arrow{d}{\mathrm{ev}_1} \\
	\Sp^{\B C_p} \arrow{r}{(-)^{\tate C_p}} & \Sp
\end{tikzcd}
\]
Consequently, there is an equivalence $\PgcSp_{\langle p \rangle} \simeq \Sp^{C_p}$, where $\Sp^{C_p}$ denotes the $\infty$-category of genuine $C_p$-spectra (cf.~\cite[Theorem 6.24]{MNN17}).
\end{example}

We summarize the salient features of the $\infty$-category of polygonic spectra and give an equalizer formula for the mapping spaces in $\PgcSp$. This is an immediate consequence of the formalism of lax equalizers developed in~\cite{NS18}. 

\begin{proposition} \label{proposition:mapping_pgcsp}
Let $T$ denote a truncation set. 

\begin{enumerate}[leftmargin=2em, topsep=5pt, itemsep=5pt]
	\item The $\infty$-category $\PgcSp_T$ is presentable and stable, and the canonical functor
	\[
		\PgcSp_T \to \Sp^{\B C_t}
	\]
	is exact and preserves small colimits for every $t \in T$. 

	\item For every pair of objects $X$ and $Y$ of $\PgcSp_T$, there is a natural equivalence of spaces
	\[
	\Map_{\PgcSp_T}(X, Y) \simeq \mathrm{Eq} \big(
	\begin{tikzcd}[column sep = 1.6em]
		\prod_{t \in T} \Map_{\Sp^{\B C_t}}(X_t, Y_t) \arrow[yshift=0.7ex]{r} \arrow[yshift=-0.7ex]{r} & \prod_{p}\prod_{t \in T/p} \Map_{\Sp^{\B C_t}}(X_t, Y_{pt}^{\tate C_p})
	\end{tikzcd}
	\big).
	\]
	If $p$ is a prime and $t \in T/p$, then the top map is induced by the following composite 
	\[
		\prod_{t \in T} \Map_{\Sp^{\B C_t}}(X_t, Y_t) \xrightarrow{\mathrm{pr}} \Map_{\Sp^{\B C_t}}(X_t, Y_t) \xrightarrow{(\varphi_{p, t})_\ast} \Map_{\Sp^{\B C_t}}(X_t, Y_{pt}^{\tate C_p}),
	\]
	and the bottom map is induced by the following composite
	\[
		\prod_{t \in T} \Map_{\Sp^{\B C_t}}(X_t, Y_t) \xrightarrow{(-)^{\tate C_p} \circ \mathrm{pr}} \Map_{\Sp^{\B C_t}}(X_{pt}^{\tate C_p}, Y_{pt}^{\tate C_p}) \xrightarrow{(\varphi_{p, t})^\ast} \Map_{\Sp^{\B C_t}}(X_t, Y_{pt}^{\tate C_p}).
	\]	
  \end{enumerate}
\end{proposition}

\begin{proof}
This follows from~\cite[Proposition II.1.5]{NS18}. 
\end{proof}

For every truncation set $T$, the $\infty$-category $\PgcSp_T$ is equivalent to the full subcategory of $\PgcSp$ spanned by those polygonic spectra $X = \{X_n\}_{n \in \N}$ which satisfy that $X_n = 0$ whenever $n$ is not contained in $T$. By construction, there is an adjunction of $\infty$-categories
\[
\begin{tikzcd}
	\PgcSp \arrow[yshift=0.8ex]{r}{\mathrm{Res}_{T}} & \PgcSp_T \arrow[yshift=-0.8ex, hook]{l}
\end{tikzcd}
\]
where the right adjoint is fully faithful, so the $\infty$-category $\PgcSp_T$ of $T$-typical polygonic spectra is obtained as a localization of the $\infty$-category $\PgcSp$ of polygonic spectra.
We use this to show that the $\infty$-category $\PgcSp_T$ is functorial in the truncation set variable and deduce that the construction $T \mapsto \PgcSp_T$ is right Kan extended from finite truncation sets.

\begin{construction} \label{construction:coherent_restriction}
If $T \subseteq T'$ is an inclusion of truncation sets, then $\PgcSp_T$ is a full subcategory of $\PgcSp_{T'}$, and there is an adjunction of $\infty$-categories
\[
\begin{tikzcd}
	\PgcSp_{T'} \arrow[yshift=0.8ex]{r}{\mathrm{Res}^{T'}_{T}} & \PgcSp_T \arrow[yshift=-0.8ex, hook]{l}
\end{tikzcd}
\]
The functor $\mathrm{Res}^{T'}_{T} : \PgcSp_{T'} \to \PgcSp_T$ is characterized by $(\mathrm{Res}^{T'}_{T} X)_t = X_t$ for every $X \in \PgcSp_{T'}$ and $t \in T$. The construction $T \mapsto \PgcSp_T$ determines a functor $\Trunc \to \mathrm{Pr}^{\mathrm{R}}$, where the functors are given by the right adjoint inclusions above. Functoriality follows since we can realize all those categories $\PgcSp_T$ as full subcategories of $\PgcSp$ and thus we simply need to write down a map of posets into the poset of subsets of isomorphism classes of objects in $\PgcSp$.
So we obtain a functor of $\infty$-categories
\[
	\PgcSp_{(-)} : \Trunc^{\op} \to \mathrm{Pr}^{\mathrm{L}}
\]
by passing to left adjoint functors, which are given by the restrictions $\mathrm{Res}^{T'}_{T} : \PgcSp_{T'} \to \PgcSp_T$. In particular, the functor $(\N, \mathrm{div}) \to \Trunc$ given by the assignment $n \mapsto \langle n \rangle$ induces a functor
\[
	\PgcSp_{\langle - \rangle} : (\N, \mathrm{div})^{\op} \to \mathrm{Pr}^{\mathrm{L}}.
\]
\end{construction}

We will need the following fundamental result:

\begin{lemma} \label{lemma:PgcSp_rightKanExtended}
The functor $\PgcSp_{(-)} : \Trunc^{\op} \to \Cat_{\infty}$ is right Kan extended from its restriction to $\FinTrunc^{\op}$. Concretely, for every truncation set $T$, the canonical functor is an equivalence
\[
	\PgcSp_T \to \mathrm{lim}_{T \to S} \, \PgcSp_S,
\]
where the limit is indexed over the category of finite truncation sets contained in $T$.
\end{lemma}

\begin{proof}
First recall that the $\infty$-category $\PgcSp_T$ is defined as the lax equalizer
\[
	\PgcSp_T = \mathrm{LEq} \Big(
	\begin{tikzcd}
		\prod_{t \in T} \Sp^{\B C_t} \arrow[yshift=0.7ex]{r} \arrow[yshift=-0.7ex]{r} & \prod_{p} \prod_{t \in T/p} \Sp^{\B C_t}
	\end{tikzcd}
	\Big),
\]
where the functors have $p$th components given by the identity and the functor $F_p$, respectively. The products indexed by $T$ and $T/p$ in the definition of $\PgcSp$ are both right Kan extended from their restriction to $\FinTrunc^{\op}$, since the canonical functor induced by the projections
\[
	\prod_{t \in T} \Sp^{\B C_t} \to \mathrm{lim}_{T \to S} \prod_{s \in S} \Sp^{\B C_s}
\]
is an equivalence of $\infty$-categories, and similarly for the product indexed by $T/p$. The projection functors are compatible with the parallel functors in the lax equalizer defining $\PgcSp$, so the canonical functor induced by the projections
\[
	\PgcSp \to \mathrm{LEq} \Big(
	\begin{tikzcd}
		\mathrm{lim}_{T \to S}\prod_{s \in S} \Sp^{\B C_s} \arrow[yshift=0.7ex]{r} \arrow[yshift=-0.7ex]{r} & \mathrm{lim}_{T \to S}\prod_{p} \prod_{s \in S/p} \Sp^{\B C_s}
	\end{tikzcd}
	\Big) \simeq \mathrm{lim}_{T \to S} \,\PgcSp_S
\]
is an equivalenc since the lax equalizer commutes with limits by definition.
\end{proof}

\begin{remark} \label{remark:limit_divisibility_kanextended}
As a consequence of Lemma~\ref{lemma:PgcSp_rightKanExtended}, the canonical functor of $\infty$-categories
\[
	\PgcSp \to \mathrm{lim}_{n \in (\N, \mathrm{div})^{\op}} \PgcSp_{\langle n \rangle}
\]
is an equivalence, since the inclusion $(\N, \mathrm{div})^{\op} \to \FinTrunc^{\op}$ is limit cofinal. This description of the $\infty$-category of polygonic spectra will play a central role in~\textsection\ref{subsection:qfgenZ_TR}. 
\end{remark}

We will leverage Lemma~\ref{lemma:PgcSp_rightKanExtended} to extend various constructions with polygonic spectra indexed by finite truncation sets to arbitrary truncations sets (see for instance Definition~\ref{definition:R_T_infinite_truncation}). We end this section by discussing symmetric monoidal structures on the $\infty$-category of polygonic spectra.

\begin{construction}
The $\infty$-category $\PgcSp$ of polygonic spectra canonically admits the structure of a symmetric monoidal $\infty$-category using that the Tate construction is lax symmetric monoidal (cf.~\cite[Construction IV.2.1]{NS18}). If $X = \{X_n\}$ and $Y = \{Y_n\}$ are polygonic spectra, then
\[
	X \otimes Y = \{X_n \otimes Y_n\},
\]
and the polygonic Frobenius maps are defined by the following $C_n$-equivariant composite
\[
	X_n \otimes Y_n \xrightarrow{\varphi_{p, n} \otimes \id} (X_{pn})^{\tate C_p} \otimes Y_n \xrightarrow{\id \otimes \varphi_{p,n}} (X_{pn})^{\tate C_p} \otimes (Y_{pn})^{\tate C_p} \to (X_{pn} \otimes Y_{pn})^{\tate C_p}
\]
for every $n$, using that the Tate construction is lax symmetric monoidal (\cite[Theorem I.3.1]{NS18}). The unit is given by the constant polygonic spectrum $\S$ with $\S_n = \S$ for each $n \geq 1$. For every truncation set $T$, the full subcategory $\PgcSp_T$ inherits the symmetric monoidal structure of $\PgcSp$, and the unit is given by the $T$-typical polygonic spectrum $\S_T$ defined by
\[
	(\S_T)_n = 
	\begin{cases}
		\S & \text{if $n \in T$} \\
		0 & \text{if $n \notin T$}
	\end{cases}
\]
In this terminology, we have that $\S_T \otimes X \simeq \mathrm{Res}_T X$ for every polygonic spectrum $X$ and the localization $\PgcSp \to \PgcSp_T$ is smashing.
\end{construction}

\subsection{Cyclotomic spectra and polygonic spectra} \label{subsection:pgc_cycsp}
We discuss the relationship between polygonic spectra and cyclotomic spectra. Informally speaking, the notion of a polygonic spectrum is a multiobject variant of a cyclotomic spectrum indexed by the $n$-polygons rather than the circle. This is precisely the structure needed to define TR (cf.~\textsection\ref{section:polygonicTR}). We will use the following notation:

\begin{notation} \label{notation:ind_res_coind}
For every integer $n \geq 1$, there is a pair of adjunctions of $\infty$-categories
\[
\begin{tikzcd}[column sep = large]
	\Sp^{\B C_n} \arrow[bend left = 20]{r}{\mathrm{ind}_n} \arrow[bend right = 20, swap]{r}{\mathrm{coind}_n} & \Sp^{\B\T}, \arrow["\mathrm{res}_n" description]{l}
\end{tikzcd}
\]
where $\mathrm{res}_n$ is the functor obtained by restriction along $\B C_n \to \B\T$. The functor $\mathrm{ind}_n$ is obtained by left Kan extension while $\mathrm{coind}_n$ is obtained by right Kan extension. Explicitly, we have that
\begin{align*}
	\mathrm{ind}_n(X) &\simeq X \otimes_{C_n} \T \\
	\mathrm{coind}_n(X) &\simeq \map_{\Sp^{\B C_n}}(\T, X)
\end{align*}
for every spectrum $X$ with $C_n$-action, where $\mathrm{ind}_n(X) \simeq X \otimes_{C_n} \T$ carries the right $\T$-action. Note that if $X$ is a spectrum with $\T$-action considered as a spectrum with $C_n$-action by restriction, then $\mathrm{ind}_n X \simeq X \otimes \Sigma^{\infty}_+(\T/C_n)$, where $X \otimes \Sigma^{\infty}_+(\T/C_n)$ carries the diagonal $\T$-action. Indeed, since the functors involved preserve colimits, we may assume that $X$ is a discrete space with $\T$-action, in which case the shearing map $X \times_{C_n} \T \to X \times \T/C_n$ given by $[x, \lambda] \mapsto (\lambda x, [\lambda])$ is a $\T$-equivariant equivalence with respect to the right $\T$-action on $X \times_{C_n} \T$ and the diagonal $\T$-action on $X \times \T/C_n$. 
\end{notation} 

We observe that every cyclotomic spectrum can be regarded as a constant polygonic spectrum. 

\begin{construction}
For every cyclotomic spectrum $X$, let $i(X)$ denote the polygonic spectrum defined by $i(X)_n = X$, where we regard $X$ as a spectrum with $C_n$-action by restriction. Similarly, we regard the $\T$-equivariant cyclotomic Frobenius map $\varphi_p : X \to X^{\tate C_p}$ as a $C_n$-equivariant map by restriction. We show that the construction $X \mapsto i(X)$ determines a functor of $\infty$-categories
\[
	i : \CycSp \to \PgcSp.
\]
By construction of $\CycSp$ and $\PgcSp$, we need to construct a functor of $\infty$-categories
\[
	\mathrm{LEq} \Big(
	\begin{tikzcd}
		\Sp^{\B\T} \arrow[yshift=0.7ex]{r}{\id} \arrow[yshift=-0.7ex, swap]{r}{(-)^{\tate C_p}} & \prod_{p} \Sp^{\B\T}
	\end{tikzcd}
	\Big) \to 
	\mathrm{LEq} \Big(
	\begin{tikzcd}
		\prod_{n \geq 1} \Sp^{\B C_n} \arrow[yshift=0.7ex]{r}{\id} \arrow[yshift=-0.7ex, swap]{r}{F_p} & \prod_{p} \prod_{n \geq 1} \Sp^{\B C_n}
	\end{tikzcd}
	\Big).
\]
For every prime $p$, the restriction functors $\Sp^{\B\T} \to \Sp^{\B C_n}$ induce a functor of $\infty$-categories
\[
	\Sp^{\B\T} \to \prod_{n \geq 1} \Sp^{\B C_n}
\]
which commutes with the top map in the lax equalizer of $\PgcSp$. For the bottom functor, there is a natural equivalence between the following composite of functors
\[
	\Sp^{\B\T} \xrightarrow{(-)^{\tate C_p}} \Sp^{\B\T} \xrightarrow{\mathrm{res}} \prod_{n \geq 1} \Sp^{\B C_n}
\]
and the following composite
\[
	\Sp^{\B\T} \xrightarrow{\mathrm{res}} \prod_{n \geq 1} \Sp^{\B C_n} \xrightarrow{F_p} \prod_{n \geq 1} \Sp^{\B C_n}
\]
induced by the natural equivalence $\mathrm{res}_n \circ (-)^{\tate C_p} \to (-)^{\tate C_p} \circ \mathrm{res}_{pn}$ of functors $\Sp^{\B\T} \to \Sp^{\B C_n}$. The equalizer is a full subcategory of the lax equalizer, we obtain the desired functor $\CycSp \to \PgcSp$. For every truncation set $T$, let $i_T$ denote the functor of $\infty$-categories defined by
\[
	\CycSp \xrightarrow{i} \PgcSp \xrightarrow{\mathrm{Res}_T} \PgcSp_T.
\]
We will refer to $i_T : \CycSp \to \PgcSp_T$ as the canonical functor.
\end{construction}

We prove that the canonical functor $i : \CycSp \to \PgcSp$ admits both adjoints and that these are described by explicit formulas (cf. Theorem~\ref{theorem:adjoints_pgcsp_cycsp}). We will begin with the following lemma:

\begin{lemma} \label{lemma:canonical_map_coinduction}
Let $X$ denote a spectrum with $C_{ps}$-action for some prime number $p$ and $s \geq 1$. There is a natural equivalence of spectra with $\T$-action
\[
	\mathrm{coind}_{ps}(X)^{\tate C_p} \to \mathrm{coind}_s(X^{\tate C_p}),
\]
where the source carries the residual $\T/C_p$-action and $X^{\tate C_p}$ carries the residual $C_{ps}/C_p$-action.
\end{lemma}

\begin{proof}
The sequence $\Z \twoheadrightarrow C_{ps} \hookrightarrow \T$, where we have identified $\T \simeq \B (ps\Z)$, induces an adjunction
\[
\begin{tikzcd}
	\Sp^{\B \T} \arrow[yshift=0.7ex]{r}{\mathrm{res}_{ps}} & \Sp^{\B C_{ps}} \arrow[yshift=-0.7ex]{l}{(-)^{\h\Z}}
\end{tikzcd}
\]
exhibiting $(-)^{\h\Z}$ as a right adjoint of restriction. Consequently, we conclude that there is a natural equivalence of functors $(-)^{\h\Z} \simeq \mathrm{coind}_{ps} : \Sp^{\B C_{ps}} \to \Sp^{\B\T}$. The desired map in the assertion of the lemma is then induced by the canonical natural transformation of functors
\[
	(-)^{\tate C_p} \circ (-)^{\h\Z} \to (-)^{\h\Z} \circ (-)^{\tate C_p} : \Sp^{\B C_{ps}} \to \Sp^{\B\T},
\]
which is an equivalence since $(-)^{\h\Z}$ is a finite limit.
\end{proof}

\begin{remark}
Similarly, we obtain a natural equivalence of spectra with $\T$-action
\[
	\mathrm{ind}_{ps}(X)^{\tate C_p} \to \mathrm{ind}_s(X)^{\tate C_p}
\]
as in Lemma~\ref{lemma:canonical_map_coinduction}. Alternatively, we obtain the map above by means of $\T$-equivariant Atiyah duality which provides a natural equivalence of spectra with $\T$-action $\mathrm{coind}_n \simeq \Omega\mathrm{ind}_n$ for $n \geq 1$. 
\end{remark}

We will begin by analyzing the right adjoint of the canonical functor $\CycSp \to \PgcSp$ which regards a cyclotomic spectrum as a constant polygonic spectrum. As an instrumental prelude, we provide an explicit construction of the right adjoint of $i_S : \CycSp \to \PgcSp_S$ under the assumption that $S$ is a finite truncation set. The general case is then obtained by leveraging Lemma~\ref{lemma:PgcSp_rightKanExtended}. 

\begin{construction} \label{construction:right_adjoint_to_cycsp_pgcsp}
Let $S$ denote a finite tru{}ncation set and let $R_S^f : \PgcSp_S \to \Sp^{\B\T}$ denote the functor of $\infty$-categories defined by the formula
\[
	R^f_S(X) = \bigoplus_{s \in S} \mathrm{coind}_s(X_s) \simeq \bigoplus_{s \in S} \map_{\Sp^{\B C_s}}(\T, X_s)
\]
for every $S$-typical polygonic spectrum $X$. We construct a cyclotomic structure on the spectrum with $\T$-action given by $R^f_S(X)$. If $p$ is a prime number and $s \in S/p$, then the cyclotomic structure on $R^f_S(X)$ is induced by the following map of spectra with $\T$-action
\[
	\mathrm{coind}_s(X_s) \xrightarrow{\mathrm{coind}_s(\varphi_{p, s})} \mathrm{coind}_{s}(X_{ps}^{\tate C_p}) \xrightarrow{\simeq} \mathrm{coind}_{ps}(X_{ps})^{\tate C_p},
\]
where the final equivalence is induced by Lemma~\ref{lemma:canonical_map_coinduction}. If $s \notin S/p$, then the cyclotomic structure on $R^f_S(X)$ is induced by $\mathrm{coind}_s(X_s) \to 0$. Consequently, the construction $X \mapsto R^f_S(X)$ canonically refines to a functor of $\infty$-categories $R^f_S : \PgcSp_S \to \CycSp$. We do not need to specify coherences by construction of the $\infty$-category of cyclotomic spectra.
\end{construction}

By similar methods, we provide an explicit construction of the left adjoint of $\CycSp \to \PgcSp_T$. In this case, we do not need to restrict to finite truncation sets.

\begin{construction} \label{construction:left_adjoint_to_cycsp_pgcsp}
Let $T$ denote a truncation set and let $L_T : \PgcSp_T \to \Sp^{\B\T}$ denote the functor of $\infty$-categories defined by the following formula
\[
	L_T(X) = \bigoplus_{t \in T} \mathrm{ind}_t(X_t)
\]
for every $T$-typical polygonic spectrum $X$. We construct a cyclotomic structure on the spectrum with $\T$-action given by $L_T(X)$. If $p$ is a prime number and $t \in T/p$, then the cyclotomic structure on $L_T(X)$ is induced by the following map of spectra with $\T$-action
\[
	\mathrm{ind}_t(X_t) \xrightarrow{\mathrm{ind}_t(\varphi_{p,t})} \mathrm{ind}_t(X_{pt}^{\tate C_p}) \xrightarrow{\simeq} \mathrm{ind}_{pt}(X_{pt})^{\tate C_p}
\]
and by $\mathrm{ind}_t(X_t) \to 0$ if $t \notin T/p$. As before, the construction $X \mapsto L_T(X)$ canonically refines to a functor of $\infty$-categories $L_T : \PgcSp_T \to \CycSp$.  
\end{construction}

As expected, we have the following result:

\begin{proposition} \label{proposition:adjoints_cycsp_pgc_finite}
Let $T$ denote an arbitrary truncation set.
\begin{enumerate}[leftmargin=2em, topsep=5pt, itemsep=5pt]
	\item If $T = S$ is finite, then $R^f_S : \PgcSp_S \to \CycSp$ is a right adjoint of $i_S : \CycSp \to \PgcSp_S$.
	\item The functor $L_T : \PgcSp_T \to \CycSp$ is a left adjoint of $i_T : \CycSp \to \PgcSp_T$.
\end{enumerate}
\end{proposition}

\begin{proof}
We begin by proving $(1)$. We show that there is a natural equivalence of spaces
\[
	\Map_{\PgcSp_S}(i_S(X), Y) \simeq \Map_{\CycSp}(X, R^f_S(Y))
\]
for every cyclotomic spectrum $X$ and every $S$-typical polygonic spectrum $Y$. Using~\cite[Proposition II.1.5]{NS18}, there is a natural equivalence of spaces
\[
	\Map_{\CycSp}(X, R^f_S(Y)) \simeq \mathrm{Eq} \big(
	\begin{tikzcd}[column sep = 1.6em]
		\Map_{\Sp^{\B\T}}(X, R^f_S(Y)) \arrow[yshift=0.7ex]{r} \arrow[yshift=-0.7ex]{r} & \prod_{p} \Map_{\Sp^{\B\T}}(X, R^f_S(Y)^{\tate C_p})
	\end{tikzcd}
	\big),
\]
where the top map is induced by the cyclotomic structure of $X$, and the bottom map is induced by the cyclotomic structure of $R^f_S(Y)$ described in Construction~\ref{construction:right_adjoint_to_cycsp_pgcsp}. Using that $S$ is finite, this equalizer is in turn naturally equivalent to the following equalizer
\begin{align*}
	&\mathrm{Eq}\Big(
	\begin{tikzcd}[ampersand replacement = \&]
		\prod_{s \in S} \Map_{\Sp^{\B\T}}(X, \mathrm{coind}_s(Y_s)) \arrow[yshift=0.7ex]{r} \arrow[yshift=-0.7ex]{r} \& \prod_{p} \prod_{s \in S/p} \Map_{\Sp^{\B\T}}(X, \mathrm{coind}_{ps}(Y_{ps})^{\tate C_p})
	\end{tikzcd}
	\Big) \\
	\simeq \, & \mathrm{Eq}\Big(
	\begin{tikzcd}[ampersand replacement = \&]
		\prod_{s \in S} \Map_{\Sp^{\B\T}}(X, \mathrm{coind}_s(Y_s)) \arrow[yshift=0.7ex]{r} \arrow[yshift=-0.7ex]{r} \& \prod_{p} \prod_{s \in S/p} \Map_{\Sp^{\B\T}}(X, \mathrm{coind}_{s}((Y_{ps})^{\tate C_p}))
	\end{tikzcd}
	\Big),
\end{align*}
since the cyclotomic structure of $R^f_S(Y)$ is induced by the zero map for $s \notin S/p$, and since the map $\mathrm{coind}_s((Y_{ps})^{\tate C_p}) \simeq \mathrm{coind}_{ps}(Y_{ps})^{\tate C_p}$ is an equivalence of spectra with $\T$-action. The top map is induced by sending $f : X \to R^f_S(Y)$ to the following map of spectra with $\T$-action
\[
	X \xrightarrow{f_s} \mathrm{coind}_s(Y_s) \xrightarrow{\mathrm{coind}_s(\varphi_{p, s})} \mathrm{coind}_s((Y_{ps})^{\tate C_p}),
\]
and the bottom map is induced by sending $f$ to the following map of spectra with $\T$-action
\[
	X \xrightarrow{\varphi_{p, s}} X^{\tate C_p} \xrightarrow{(f_{ps})^{\tate C_p}} \mathrm{coind}_{ps}(Y_{ps})^{\tate C_p} \simeq \mathrm{coind}_s((Y_{ps})^{\tate C_p}).
\]
We conclude that there is a natural equivalence of spaces
\[
	\Map_{\CycSp}(X, R^f_S(Y) \simeq \mathrm{Eq} \big(
	\begin{tikzcd}[column sep = 1.6em]
		\prod_{s \in S} \Map_{\Sp^{\B C_s}}(X, Y_s) \arrow[yshift=0.7ex]{r} \arrow[yshift=-0.7ex]{r} & \prod_{p} \prod_{s \in S} \Map_{\Sp^{\B C_s}}(X, (Y_{ps})^{\tate C_p})
	\end{tikzcd}
	\big),
\]
using that coinduction is right adjoint of the restriction functor, where the maps are as described in Proposition~\ref{proposition:mapping_pgcsp}. The proof of $(2)$ is entirely analogous using that the induction functor is left adjoint of the restriction functor. 
\end{proof}

Before we proceed, we recall some terminology which we will use throughout. 

\begin{notation} \label{notation:reducedTHH_inverselimit}
Let $\S[t]$ denote the $\E_{\infty}$-ring defined by the formula
\[
	\S[t] = \Sigma^{\infty}_+ \mathbb{N}.
\]
As a warning, we note that $\S[t]$ is not the free $\E_\infty$-ring on one generator. However, the underlying $\E_1$-ring is the free $\E_1$-ring on one generator since $\mathbb{N}$ is the free $\E_1$-monoid on a single generator in the $\infty$-category of spaces. We consider the following constructions:

\begin{enumerate}[leftmargin=2em, topsep=5pt, itemsep=5pt]
	\item There is a map of $\E_\infty$-rings $\S[t] \to \S$ induced by the unique map of $\E_\infty$-monoids $\mathbb{N} \to \ast$. We obtain a map of cyclotomic spectra $\THH(\S[t]) \to \S$ where the specturm $\S$ is equipped with the trivial cyclotomic structure. Let $\widetilde{\THH}(\S[t])$ denote the cyclotomic spectrum defined by
	\[
		\widetilde{\THH}(\S[t]) = \mathrm{fib}(\THH(\S[t]) \to \S).
	\]
	In fact, there is an equivalence of cyclotomic spectra
	\[
		\widetilde{\THH}(\S[t]) = \bigoplus_{n \geq 1} \Sigma^\infty_+(\T/C_n),
	\]
	where $\T/C_n$ denotes the quotient space whose $\T$-action is given by $[x] \mapsto [\lambda x]$ for every $\lambda \in \T$, and the cyclotomic structure of the right hand side of the equivalence above is induced by the following $\T$-equivariant map of spaces
	\[
		\T/C_n \xrightarrow{x \mapsto \sqrt[p]{x}} (\T/C_{pn})^{\h C_p} 
	\]
	for every prime number $p$. We refer the reader to~\cite[\textsection 2]{McC21} for a discussion of this.

	\item For every $n \geq 1$, let $\S[t]/t^n$ denote the $\E_\infty$-ring defined by the pushout of $\E_\infty$-rings
	\[
	\begin{tikzcd}
		\S[t] \arrow{r}{t \mapsto t^n} \arrow{d}{t \mapsto 0} & \S[t] \arrow{d} \\
		\S \arrow{r} & \S[t]/t^n
	\end{tikzcd}
	\]
	Let $X \,\widehat{\otimes}\, \widetilde{\THH}^{\mathrm{cont}}(\S[\![t]\!])$ denote the cyclotomic spectrum defined by
	\[
		X \,\widehat{\otimes}\, \widetilde{\THH}^{\mathrm{cont}}(\S[\![t]\!]) = \varprojlim(X \otimes \widetilde{\THH}(\S[t]/t^n))
	\]
	for every cyclotomic spectrum $X$. The right hand side denotes the limit in the $\infty$-category of cyclotomic spectra formed over the maps induced by $\S[t]/t^{n+1} \to \S[t]/t^n$ (cf.~\cite[\textsection 4]{McC21}).
\end{enumerate}
\end{notation}

In general, we define $R_T$ for an arbitrary truncation set by a Kan extension procedure. 

\begin{definition} \label{definition:R_T_infinite_truncation}
Let $R_T : \PgcSp_T \to \CycSp$ denote the functor of $\infty$-categories defined by
\[
	R_T(X) = \mathrm{lim}\big( \Trunc_{T/}^{\op} \times_{\Trunc^{\op}} \FinTrunc^{\op} \to \FinTrunc^{\op} \to \CycSp \big)
\]
for every truncation set $T$. The last functor is determined by the construction $S \mapsto R^f_S(\mathrm{Res}^T_S X)$. Alternatively, the functor $R_T$ is defined as the right Kan extension of the functor $S \mapsto R^f_S(\mathrm{Res}^T_S X)$ along the inclusion $i : \FinTrunc \hookrightarrow \Trunc$. We will typically write
\[
	R_T(X) = \mathrm{lim}_{S \subseteq T} R^f_S(\mathrm{Res}^T_S X)
\]
to shorten the notation. 
\end{definition}

We prove the main result of this section, which summarizes the various functors relating the $\infty$-category of polygonic spectra and the $\infty$-category of cyclotomic spectra. These descriptions will play an important role in~\textsection\ref{subsection:def_polygonicTR}. 

\begin{theorem} \label{theorem:adjoints_pgcsp_cycsp}
Let $T$ be a truncation set. The functor $i_T$ admits both adjoints
\[
\begin{tikzcd}[column sep = large]
	\PgcSp_T \arrow[yshift=1.6ex]{r}{L_T} \arrow[yshift=-1.6ex, swap]{r}{R_T} & \CycSp. \arrow["i_T" description]{l}
\end{tikzcd}
\]
If $X$ is a cyclotomic spectrum and $T = \N$, then there are equivalences of cyclotomic spectra
\begin{align*}
	L i(X) &\simeq X \otimes \widetilde{\THH}(\S[t]) \\
	R i(X) &\simeq X \,\widehat{\otimes}\,\,\Omega\widetilde{\THH}^{\mathrm{cont}}(\S[\![t]\!]),
\end{align*}
where both $\widetilde{\THH}(\S[t])$ and $X \,\widehat{\otimes}\, \widetilde{\THH}^{\mathrm{cont}}(\S[\![t]\!])$ are defined in Notation~\ref{notation:reducedTHH_inverselimit}.
\end{theorem}

\begin{proof}
The functor $L_T : \PgcSp_T \to \CycSp$ is a left adjoint of the functor $i_T : \CycSp \to \PgcSp_T$ by part $(2)$ of Proposition~\ref{proposition:adjoints_cycsp_pgc_finite}. If $T = \N$, then we have equivalences of cyclotomic spectra
\[
	Li(X) \simeq \bigoplus_{n \geq 1} \mathrm{ind}_n(\mathrm{res}_n X) \simeq \bigoplus_{n \geq 1} X \otimes \Sigma^{\infty}_+(\T/C_n) \simeq X \otimes \widetilde{\THH}(\S[t]),
\]
using that $\mathrm{ind}_n(Y) \simeq Y \otimes \Sigma^{\infty}_+(\T/C_n)$ provided that $Y$ is a spectrum with $C_n$-action obtained as the restriction of a spectrum with $\T$-action. To obtain the desired right adjoint, note that the functor $i_T$ preserves small colimits since these are computed underlying. It follows that $i_T$ admits a right adjoint since both $\CycSp$ and $\PgcSp_T$ are presentable, and this right adjoint is given by $R^f_S$ if $S$ is finite (cf. Proposition~\ref{proposition:adjoints_cycsp_pgc_finite}). We proceed by exhibiting a natural equivalence
\[
	\Map_{\PgcSp_T}(i_T(X), Y) \simeq \Map_{\CycSp}(X, R_T(Y))
\]
for every cyclotomic spectrum $X$ and $T$-typical polygonic spectrum $Y$. We have that
\begin{align*}
	\Map_{\CycSp}(X, R_T(Y)) &\simeq \mathrm{lim}_{S \subseteq T} \Map_{\CycSp}(X, R^f_S(\mathrm{Res}^T_S Y)) \\
	&\simeq \mathrm{lim}_{S \subseteq T} \Map_{\PgcSp_S}(i_S(X), \mathrm{Res}^T_S Y) \\
	&\simeq \mathrm{lim}_{S \subseteq T} \Map_{\PgcSp_S}(\mathrm{Res}^T_S \, i_T(X), \mathrm{Res}^T_S Y) \\ 
	&\simeq \Map_{\PgcSp_T}(i_T(X), Y),
\end{align*}
where the last equivalence is an instance of Lemma~\ref{lemma:PgcSp_rightKanExtended}. This proves that the right adjoint of $i_T$ indeed is given by $R_T$. If $T = \N$, then there are equivalence of cyclotomic spectra
\[
	R i(X) \simeq \varprojlim_n \bigoplus_{i = 1}^n (X \otimes \Sigma^{\infty-1}_+(\T/C_i)) \simeq \prod_{n \geq 1} X \otimes \Sigma^{\infty-1}_+(\T/C_n) \simeq \varprojlim_n \Omega (X \otimes \widetilde{\THH}(\S[t]/t^n)),
\]
where the first limit is formed over the projection maps, which are maps of cyclotomic spectra. We have additionally used $\T$-equivariant Atiyah duality to identify $\mathrm{coind}_i(X) \simeq \Omega\mathrm{ind}_i(X)$, and the final equivalence is the content of~\cite[Proposition 4.1.14]{McC21}. This proves the desired statement. 
\end{proof}

\begin{remark} \label{remark:general_truncationset_formula}
In the setting of Theorem~\ref{theorem:adjoints_pgcsp_cycsp}, we also obtain expressions for $R_Ti_T(X)$ for an arbitrary truncation set $T$ in place of $T = \N$. There is an equivalence of cyclotomic spectra
\[
	L_T i_T(X) \simeq X \otimes \widetilde{\THH}_T(\S[t]),
\]
where $\widetilde{\THH}_T(\S[t]) = \bigoplus_{t \in T} \Sigma^{\infty}_+(\T/C_t)$. Furthermore, there is an equivalence of cyclotomic spectra
\[
	R_T i_T(X) \simeq \prod_{t \in T} X \otimes \Sigma^{\infty-1}_+(\T/C_t),
\]
where the right hand side is the product formed in the $\infty$-category of cyclotomic spectra.
\end{remark}

As a warning, we note that the forgetful functor of $\infty$-categories
\[
	\CycSp \to \Sp^{\B\T}
\]
does not preserve limits in general. The limits appearing in the identification of $R i(X)$ in the proof of Theorem~\ref{theorem:adjoints_pgcsp_cycsp} and Remark~\ref{remark:general_truncationset_formula} above are formed in the $\infty$-category $\CycSp$, and there is no a priori reason to expect that they are preserved by the functor $\CycSp \to \Sp^{\B\T}$. Nonetheless, we briefly discuss certain situations where these limits are formed underlying.

\begin{example} \label{example:cyclotomic_product}
Let $T$ be a truncation set and assume that $X$ is a $T$-typical polygonic spectrum which is uniformly bounded below. This means that there exists an integer $N$ such that each $X_t$ is $N$-connective. In this case, the canonical map of spectra with $\T$-action
\[
	\big( \prod_{t \in T} \mathrm{coind}_t(X_t) \big)^{\tate C_p} \to \prod_{t \in T} \mathrm{coind}_t(X_t)^{\tate C_p}
\]
is an equivalence by~\cite[Lemma 2.11]{AN20}, so the product
\[
	\prod_{t \in T} \mathrm{coind}_t(X_t)
\]
formed in $\Sp^{\B\T}$ carries a cyclotomic structure induced by the $\T$-equivariant map of spectra
\[
	\mathrm{coind}_t(X_t) \to \mathrm{coind}_{pt}(X_{pt})^{\tate C_p}
\]
obtained in Lemma~\ref{lemma:canonical_map_coinduction} for $t \in T/p$, and by $\mathrm{coind}_t(X_t) \to 0$ for $t \notin T/p$.
\end{example}

\begin{corollary}
Let $T$ denote an arbitrary truncation set. If $X$ is a $T$-typical polygonic spectrum which is uniformly bounded below, then there is a natural equivalence of cyclotomic spectra
\[
	\prod_{t \in T} \mathrm{coind}_t(X_t) \to R_T(X),
\]
where the cyclotomic structure on the product is defined in Example~\ref{example:cyclotomic_product}.
\end{corollary}

\begin{proof}
If $S \subseteq T$ is an inclusion of a finite truncation set, then the projection
\[
	\prod_{t \in T} \mathrm{coind}_t(X_t) \to R^f_S(\mathrm{Res}^T_S X)
\]
determines a map of cyclotomic spectra, where the left hand side is equipped with the cyclotomic structure of Example~\ref{example:cyclotomic_product}. We obtain a canonical map of cyclotomic spectra
\[
	\Phi : \prod_{t \in T} \mathrm{coind}_t(X_t) \to R_T(X)
\]
which we prove is an equivalence. It suffices to show that $\Phi$ is an equivalence of spectra with $\T$-action since the canonical functor $\CycSp \to \Sp^{\B\T}$ preserves these limits since $X$ is assumed to be uniformly bounded below. In this case, the inclusion
\[
	R^f_S(X) \hookrightarrow \prod_{t \in T} \mathrm{coind}_t(X_t)
\]
determines a map of spectra with $\T$-action and these assemble into a map $R_T(X) \to \prod_{t \in T} \mathrm{coind}_t(X_t)$ which provides an inverse to $\Phi$ as a map of spectra with $\T$-action.
\end{proof}

We will end this section by recording an example which will be important in~\textsection\ref{subsection:qfgenZ_TR}.

\begin{example} \label{example:polygonic_shift}
Let $\mathrm{sh}_k X$ denote the polygonic spectrum defined by
\[
	(\mathrm{sh}_k X)_n = \mathrm{res}^{kn}_{n}(X_{kn}),
\]
for every polygonic spectrum $X$, where $\mathrm{res}^{kn}_{n} : \Sp^{\B C_{kn}} \to \Sp^{\B C_n}$ denotes the restriction functor. The construction $X \mapsto \mathrm{sh}_k X$ refines to a functor of $\infty$-categories $\mathrm{sh}_k : \PgcSp \to \PgcSp$. Let $\ell X$ denote the polygonic spectrum with $(\ell_k X)_{kn} = \mathrm{ind}_n^{kn}(X_n)$ and $(\ell_k X)_n = 0$ if $k$ does not divide $n$. The resulting functor $\ell_k : \PgcSp \to \PgcSp$ is a left adjoint of $\mathrm{sh}_k$. Indeed, we find that 
\[
	\Map_{\PgcSp}(\ell_k X, Y) \simeq \Map_{\PgcSp}(X, \mathrm{sh}_k Y)
\]
by Proposition~\ref{proposition:mapping_pgcsp} combined with the fact that $\mathrm{ind}_n^{kn}$ is a left adjoint of $\mathrm{res}^{kn}_n$ for all $k$ and $n$. Note that if $M$ is a bimodule over an $\E_1$-ring $R$, then there is an equivalence of polygonic spectra
\[
	\mathrm{sh}_k \THH(R, M) \simeq \THH(R, M^{\otimes_R k}).
\]
The shift functor $\mathrm{sh}_k(-)$ also admits a right adjoint obtained by replacing the induction functors above by coinduction.
\end{example}

\begin{remark}
In Construction~\ref{construction:Naction_qfgen_pgcsp}, we construct an action of $\N$ on $\PgcSp$ by the construction $n \mapsto \mathrm{sh}_n$. The $\infty$-category $\PgcSp^{\h\N}$ is equivalent to the so-called $\infty$-category of $\Q/\Z$-cyclotomic spectra. 
\end{remark}


\section{Polygonic TR} \label{section:polygonicTR}
Using the formalism developed in~\textsection\ref{section:pgcsp}, we define TR as a functor on the $\infty$-category of polygonic spectra by a corepresentability construction as in~\cite{BM16,McC21}. In~\textsection\ref{subsection:polygonic_t_structure}, we introduce a $t$-structure on the $\infty$-category of polygonic spectra as an analogue of the cyclotomic $t$-structure on the $\infty$-category of cyclotomic spectra studied by Antieau and the third author in~\cite{AN20}. 

\subsection{The definition of polygonic TR} \label{subsection:def_polygonicTR}
We define $T$-typical polygonic TR as follows:

\begin{definition} \label{definition:polygonicTR}
Let $T$ be a truncation set. If $X$ is a $T$-typical polygonic spectrum, then
\[
	\TR_T(X) = \map_{\PgcSp_T}(\S_T, X).
\]
The construction $X \mapsto \TR_T(X)$ determines a functor of $\infty$-categories $\TR_T : \PgcSp_T \to \Sp$. If $T = \N$ then we write $\TR(X) := \TR_T(X)$.
\end{definition}

\begin{remark}
According to Definition~\ref{definition:polygonicTR}, polygonic TR is corepresented by the sphere spectrum on the $\infty$-category $\PgcSp$ of polygonic spectra. This is similar to the result which asserts that TC is corepresentable by the sphere spectrum on the $\infty$-category of cyclotomic spectra (cf.~\cite{BM16,NS18}). However, topological cyclic homology can not be defined on the $\infty$-category of polygonic spectra, thus making TR the `correct' construction to consider in our context. 
\end{remark}

\begin{remark}
In~\cite{LM12}, Lindenstrauss and McCarthy introduce a version of TR with coefficients. Concretely, they define TR of a functor $F$ with smash product with coefficients in an $F$-bimodule. Their construction also makes use of the polygonic structure on topological Hochschild homology with coefficients though formulated in a different formalism. We do not compare our construction of TR with the construction obtained by Lindenstrauss--McCarthy in~\cite{LM12}. We certainly expect that the two constructions agree as indicated by comparing isotropy separation sequences and using the Tate-orbit lemma. The main technical obstacle is providing a framework which at the same time encodes the point-set construction of~\cite{LM12} using functors with smash products and our construction on an equal footing.
\end{remark}

Before we continue, we explore a few examples to get a feeling for the definition.

\begin{example} \label{example:polygonic_concentrated_degree_1}
If $S = \langle 1 \rangle$, then $\PgcSp_{\langle 1 \rangle} \simeq \Sp$ as in Example~\ref{example:trivialtruncationset}, so $\TR_{\langle 1 \rangle}(X) \simeq X$.
\end{example}

\begin{example}
Let $X$ be a cyclotomic spectrum whose underlying spectrum is bounded below considered as a constant $\langle p \rangle$-typical polygonic spectrum for an arbitrary prime number $p$. Using the equalizer formula obtained in Proposition~\ref{proposition:mapping_pgcsp}, we find that
\[
	\TR_{\langle p \rangle}(X) \simeq \mathrm{Eq} \big(
	\begin{tikzcd}[column sep = 1.6em]
		X \times X^{\h C_p} \arrow[yshift=0.7ex]{r} \arrow[yshift=-0.7ex]{r} & X^{\tate C_p}
	\end{tikzcd}
	\big)
\]
where the maps are determined by $X \times X^{\h C_p} \to X \xrightarrow{\varphi} X^{\tate C_p}$ and $X \times X^{\h C_p} \to X^{\h C_p} \xrightarrow{\can} X^{\tate C_p}$, respectively. It follows that $\TR_{\langle p \rangle}(X)$ can be expressed as the following pullback of spectra
\[
\begin{tikzcd}
	\TR_{\langle p \rangle}(X) \arrow{r} \arrow{d} & X^{\h C_p} \arrow{d}{\can} \\
	X \arrow{r}{\varphi} & X^{\tate C_p}
\end{tikzcd}
\]
which implies that $\TR_{\langle p \rangle}(X) \simeq X^{C_p}$ by~\cite{NS18}. 
\end{example}

\begin{remark}
For every truncation set $T$, the functor $\TR_T : \PgcSp_T \to \Sp$ canonically refines to a lax symmetric monoidal functor since it is corepresented by the unit.
\end{remark}

More generally, we obtain an explicit equalizer formula for $T$-typical polygonic TR.

\begin{proposition}
Let $T$ be a truncation set. If $X$ is a $T$-typical polygonic spectrum, then
\[
	\TR_T(X) \simeq \mathrm{Eq} \big(
	\begin{tikzcd}[column sep = 1.6em]
		\displaystyle\prod_{t \in T} X_t^{\h C_t} \arrow[yshift=0.7ex]{r} \arrow[yshift=-0.7ex]{r} & \displaystyle\prod_{p}\prod_{t \in T/p} (X_{pt}^{\tate C_p})^{\h C_t}
	\end{tikzcd}
	\big)
\]
where the maps are $X_{pt}^{\h C_{pt}} \simeq (X_{pt}^{\h C_p})^{\h C_t} \xrightarrow{\can^{\h C_t}} (X_{pt}^{\tate C_p})^{\h C_t}$ and $X_t^{\h C_t} \xrightarrow{\varphi_{p,t}^{\h C_t}} (X_{pt}^{\tate C_p})^{\h C_t}$ for $pt \in T$.
\end{proposition}

\begin{proof}
This follows from Proposition~\ref{proposition:mapping_pgcsp} using that $\map_{\Sp^{\B C_t}}(\S, Y) \simeq Y^{\h C_t}$ for every $Y \in \Sp^{\B C_t}$.
\end{proof}

We compare our construction of TR with the corepresentability construction obtained in~\cite{McC21}. It follows from~\cite[Theorem 3.3.12]{McC21} that our construction of TR agrees with the classical construction of TR of a cyclotomic spectrum whose underlying spectrum is bounded below. The full strength of the comparison obtained in~\cite{McC21} is that it is a comparison of spectra with Frobenius lifts. 

\begin{proposition} \label{proposition:comparison_mcc}
If $X$ is a cyclotomic spectrum, then there is a natural equivalence of spectra
\[
	\TR(X) \simeq \map_{\CycSp}(\widetilde{\THH}(\S[t]), X),
\]
where the right hand side of the equivalence above is the construction of $\TR(X)$ from~\cite{McC21}. 
\end{proposition}

\begin{proof}
There is an equivalence of cyclotomic spectra $L\S \simeq \widetilde{\THH}(\S[t])$ by virtue of Theorem~\ref{theorem:adjoints_pgcsp_cycsp}, which immediately implies that $\TR(X) \simeq \map_{\PgcSp}(\S, i(X)) \simeq \map_{\CycSp}(\widetilde{\THH}(\S[t]), X)$.
\end{proof}

We construct the polygonic TR tower on the $\infty$-category of polygonic spectra.

\begin{construction} \label{construction:polygonicTRtower}
Recall that the construction $T \mapsto \PgcSp_T$ refines to a functor of $\infty$-categories
\[
	\PgcSp_{(-)} : \Trunc^{\op} \to \mathrm{Pr}^{\mathrm{L}}
\]
where an inclusion $T \subseteq T'$ is carried to the functor $\mathrm{Res}^{T'}_{T} : \PgcSp_{T'} \to \PgcSp_T$ (cf. Construction~\ref{construction:coherent_restriction}). As a consequence, we obtain a functor of $\infty$-categories
\[
	\Trunc^{\op} \to \Fun(\PgcSp, \Sp)
\]
determined by $T \mapsto \map_{\PgcSp}(\S, \mathrm{Res}_T(-))$. For a polygonic spectrum $X$, we refer to the functor 
\[
	\TR_{(-)}(X) : \Trunc^{\op} \to \Sp
\]
given by the construction $T \mapsto \map_{\PgcSp}(\S, \mathrm{Res}_T(X)) \simeq \TR_T(X)$ as the polygonic TR tower. It will be convenient to work with the following version of the polygonic TR tower
\[
	\TR_{\langle \star \rangle}(X) : (\N, \mathrm{div})^{\op} \to \Sp
\]
obtained by precomposition with the functor $(\N, \mathrm{div}) \to \Trunc$ given by $n \mapsto \langle n \rangle$. In this case, if $(m, n)$ is a pair of positive integers with $n$ dividing $m$, then we refer to the map of spectra
\[
	R_{\nicefrac{m}{n}}^{\hex} : \TR_{\langle m \rangle}(X) \to \TR_{\langle n \rangle}(X)
\]
as the polygonic restriction map.
\end{construction}

We prove that polygonic TR is the limit of the polygonic TR tower.

\begin{lemma}
If $X$ is a polygonic spectrum, then there is an equivalence of spectra
\[
	\TR(X) \simeq \mathrm{lim}_{n \in (\N, \mathrm{div})^{\op}} \TR_{\langle n \rangle}(X).
\]
\end{lemma}

\begin{proof}
By Remark~\ref{remark:limit_divisibility_kanextended}, there is an equivalence of $\infty$-categories
\[
	\PgcSp \simeq \mathrm{lim}_{n \in (\N, \mathrm{div})^{\op}} \PgcSp_{\langle n \rangle}
\]
which yields that $\TR(X) \simeq \mathrm{lim}_{n \in (\N, \mathrm{div})^{\op}} \map_{\PgcSp_{\langle n \rangle}}(\S, X) \simeq \mathrm{lim}_{n \in (\N, \mathrm{div})^{\op}} \TR_{\langle n \rangle}(X)$. 
\end{proof}

\begin{remark}
More generally, there is an equivalence of spectra
\[
	\TR(X) \simeq \mathrm{lim}_{S \in \FinTrunc^{\op}} \TR_S(X).
\]
for every polygonic spectrum $X$ by using Lemma~\ref{lemma:PgcSp_rightKanExtended} in place of Remark~\ref{remark:limit_divisibility_kanextended}.  
\end{remark}

Using the explicit formula for the right adjoint of the canonical functor $i_T : \CycSp \to \PgcSp_T$, we obtain the following refinement of~\cite[Theorem 4.2.3]{McC21} for an arbitrary truncation set. As before, we stress that the product in the result below is formed in the $\infty$-category of cyclotomic spectra, and that $\TC = \map_{\CycSp}(\S, -)$ by~\cite{NS18}. 

\begin{theorem} \label{theorem:curves_arbitrary_truncationset}
Let $T$ denote a truncation set. There is a natural equivalence of spectra
\[
	\TR_T(X) \simeq \Omega\TC\Big( \prod_{t \in T} X \otimes \Sigma^\infty_+(\T/C_t) \Big).
\]
for every cyclotomic spectrum $X$. If $T = \N$, then there is a natural equivalence of spectra
\[
	\TR(X) \simeq \varprojlim \Omega \TC(X \otimes \widetilde{\THH}(\S[t]/t^n)).
\]
\end{theorem}

\begin{proof}
First note that there is an equivalence of cyclotomic spectra
\[
	R_T i_T(X) \simeq \prod_{t \in T} X \otimes \Sigma^{\infty-1}_+(\T/C_t)
\]
by the second part of Remark~\ref{remark:general_truncationset_formula}, so we conclude that
\begin{align*}
	\TR_T(X) &\simeq \map_{\PgcSp_T}(i_T(\S), i_T(X)) \\
	&\simeq \map_{\CycSp}(\S, R_T i_T(X)) \\
	&\simeq \Omega \TC \Big(\prod_{t \in T} X \otimes \Sigma^\infty_+(\T/C_t)\Big)
\end{align*}
using Theorem~\ref{theorem:adjoints_pgcsp_cycsp} and that $\TC$ is corepresented by the sphere spectrum $\S$ on the $\infty$-category of cyclotomic spectra. The final assertion is a consequence of the last part of Theorem~\ref{theorem:adjoints_pgcsp_cycsp}.
\end{proof}

We calculate TR of a polygonic spectrum which is concentrated in a single degree and bounded below. This calculation will play a crucial role in~\textsection\ref{subsection:qfgenZ_TR}. As an input for the calculation, we record the following profinite variant of~\cite[Lemma II.4.1 \& Lemma II.4.2]{NS18}.

\begin{lemma} \label{lemma:version_tate_orbit_lemma}
Let $X$ denote a spectrum with $C_n$-action for some integer $n \geq 2$. If $X$ is bounded below, then the canonical map of spectra
\[
	X^{\tate C_n} \to \prod_{p | n} (X^{\tate C_p})^{\h C_{\nicefrac{n}{p}}}
\]
is an equivalence, where the product is indexed over the set of prime numbers dividing $n$.
\end{lemma}

\begin{proof}
First we claim that for every prime $p$ dividing $n$, the canonical map
\[
	X^{\tate C_n} \to (X^{\tate C_p})^{\h C_{\nicefrac{n}{p}}}
\]
is a $p$-adic equivalence and the target is $p$-complete. The latter follows by~\cite[Lemma I.2.9]{NS18}, where we additionally have used that limits of $p$-complete spectra are $p$-complete. To see that the map is a $p$-adic equivalence we may replace $X$ by $X^\wedge_p$, i.e. we may assume that $X$ is $p$-complete. Then $X^{\tate C_n} = (X^{\tate C_{p^k}})^{\h C_{\nicefrac{n}{p^k}}}$ where $k$ is the $p$-adic valuation of $n$. 
This follows since the Tate construction is the cofiber of the map
\[
X_{\h C_n} \simeq (X_{\h C_{p^k}})_{\h C_{\nicefrac{n}{p^k}}} \simeq (X_{\h C_{p^k}})^{\h C_{\nicefrac{n}{p^k}}} \to X^{\h C_n} \simeq (X^{\h C_{p^k}})^{\h C_{\nicefrac{n}{p^k}}} 
\]
where we have used that $\nicefrac{n}{p^k}$ is invertible on $X$. Finally the claim then follows from \cite[Lemma II.4.1]{NS18}.
Consequently, the map in question
\[
	X^{\tate C_n} \to \prod_{p | n} (X^{\tate C_p})^{\h C_{\nicefrac{n}{p}}}
\]
is a profinite equivalence and the target is profinite complete. Therefore, it suffices to prove that $X^{\tate C_n}$ is profinite complete. Using~\cite[Lemma I.2.6]{NS18} and the fact that limits of profinite complete spectra are profinite complete, we assume that $X$ is bounded. We may further reduce to the case when $X = \mathrm{H}M[i]$ is an Eilenberg--Mac Lane spectrum concentrated in degree $i$. In this case, the Tate cohomology $\widehat{\mathrm{H}}^{\star}(C_n, M)$ has $n$-torsion, so $X^{\tate C_n}$ has $n$-torsion, which implies that $X^{\tate C_n}$ is profinitely complete. This proves the desired statement.
\end{proof}

\begin{lemma} \label{lemma:TR_polygonic_in_degree_n}
Let $X$ be a polygonic spectrum which is concentrated in degree $n$ for some $n \geq 1$. If $X$ is bounded below, then there is a natural equivalence of spectra $\TR(X) \simeq X_{\h C_n}$. 
\end{lemma}

\begin{proof}
If $n = 1$, then $\TR(X) \simeq X$ by virtue of Example~\ref{example:polygonic_concentrated_degree_1}. If $n \geq 2$, then we conclude that
\[
	\TR(X) \simeq \mathrm{Eq}\Big( 
	\begin{tikzcd}
		X^{\h C_n} \arrow[yshift=0.7ex]{r}{\can} \arrow[yshift=-0.7ex, swap]{r}{0} & \displaystyle\prod_{p | n} (X^{\tate C_p})^{\h C_{\frac{n}{p}}}
	\end{tikzcd}
	\Big),
\]
since $X$ is concentrated in degree $n$, where the top map is given by the composite
\[
	X^{\h C_n} \simeq (X^{\h C_p})^{\h C_{\frac{n}{p}}} \xrightarrow{\can} (X^{\tate C_p})^{\h C_{\frac{n}{p}}}
\]
if $p$ divides $n$. Using Lemma~\ref{lemma:version_tate_orbit_lemma}, we conclude that $\TR(X) \simeq \mathrm{fib}(X^{\h C_n} \xrightarrow{\can} X^{\tate C_n}) \simeq X_{\h C_n}$.
\end{proof}

\subsection{The polygonic $t$-structure} \label{subsection:polygonic_t_structure}
We construct a $t$-structure on the $\infty$-category of polygonic spectra entirely analogous to the construction of the $t$-structure on the $\infty$-category of cyclotomic spectrum in~\cite[\textsection 2.1]{AN20}. Our presentation follows the presentation in~\cite{AN20} verbatim. 

\begin{definition}
Let $\PgcSp_{\geq 0}$ denote the full subcategory of $\PgcSp$ spanned by those polygonic spectra $\{X_n\}_{n \geq 1}$ which satisfy that the underlying spectrum of $X_n$ is connective for every $n \geq 1$.
\end{definition}

If $X \in \PgcSp_{\geq 0}$ is a connective polygonic spectrum, then the polygonic Frobenius map
\[
	\varphi_{p,n} : X_n \to (X_{pn})^{\tate C_p}
\]
canonically factors through the connective cover $\tau_{\geq 0} (X_{pn})^{\tate C_p}$ for every $n \geq 1$ and every prime $p$. In particular, the functor $F_p$ from Construction~\ref{construction:F_p_functor} canonically refines to a functor
\[
	\tau_{\geq 0} F_p : \displaystyle\prod_{n \geq 1} (\Sp^{\B C_n})_{\geq 0} \to \displaystyle\prod_{n \geq 1} (\Sp^{\B C_n})_{\geq 0}.
\]
Using this, we establish the analog of~\cite[Theorem 2.1]{AN20} for the $\infty$-category $\PgcSp$ of polygonic spectra. The proof is a direct adaptation of the argument given by Antieau and the third author.

\begin{proposition} \label{proposition:polygonic_t_structure}
The $\infty$-category $\PgcSp_{\geq 0}$ forms the connective part of an accessible $t$-structure on $\PgcSp$ which is left complete and compatible with the symmetric monoidal structure on $\PgcSp$.
\end{proposition}

\begin{proof}
By~\cite[Proposition 1.4.4.11]{Lur17} it suffices to prove that $\PgcSp_{\geq 0}$ is a presentable $\infty$-category which is closed under extensions and small colimits. Note that $\PgcSp_{\geq 0}$ is closed under extensions by construction. As in~\cite[Lemma 2.10]{AN20}, the $\infty$-category $\PgcSp_{\geq 0}$ is equivalent to the lax equalizer
\[
	\PgcSp_{\geq 0} \simeq \mathrm{LEq}\Big(
	\begin{tikzcd}
		\displaystyle\prod_{n \geq 1} (\Sp^{\B C_n})_{\geq 0} \arrow[yshift=0.7ex]{r} \arrow[yshift=-0.7ex]{r} & \displaystyle\prod_{p \in \mathbb{P}}\prod_{n \geq 1} (\Sp^{\B C_n})_{\geq 0}
	\end{tikzcd}
	\Big),
\]
where the top map is the identity and the bottom map is $\{\tau_{\geq 0} F_p\}$, where the desired map goes from right to left and is induced by the natural transformation $\tau_{\geq 0} F_p \to F_p$. This implies that $\PgcSp_{\geq 0}$ is presentable and closed under small colimits as explained in~\cite[Lemma 2.10]{AN20}. Similarly, the $t$-structure is left complete by a direct adaptation of the argument in the proof of~\cite[Theorem 2.1]{AN20}. The $t$-structure is compatible with the symmetric monoidal structure on $\PgcSp$ since the unit $\S$ is connective and the tensor product of connective polygonic spectra is connective.
\end{proof}

We refer to this $t$-structure as the polygonic $t$-structure on the $\infty$-category $\PgcSp$ of polygonic spectra. The coconnective part of the polygonic $t$-structure is determined by its connective part in the sense that a polygonic spectrum $Y$ is contained in $\PgcSp_{\leq 0}$ precisely if
\[
	\Map_{\PgcSp}(X, Y[-1]) \simeq 0
\]
for every connective polygonic spectrum $X$. For instance, if $X \in \PgcSp_{\leq 0}$, then $\pi_i\TR(X) \simeq 0$ for every $i > 0$ since $\TR(X) = \map_{\PgcSp}(\S, X)$ is corepresented by $\S$ which is connective in the polygonic $t$-structure by construction. 

\begin{remark}
The construction of the polygonic $t$-structure on the $\infty$-category of polygonic spectra is formal. As in~\cite{AN20}, the difficult part is identifying the heart of the polygonic $t$-structure. In fact, the heart is equivalent to the abelian category of so-called complete polygonic Cartier modules. We discuss this further in upcoming work, where we additionally use the ideas of this paper to identify the heart of the $t$-structure on the $\infty$-category of cyclotomic spectra in terms of (integral) complete Cartier modules extending the work of Antieau and the third author in~\cite{AN20} to the integral case.
\end{remark}

\section{Quasifinitely genuine $G$-spectra after Kaledin} \label{section:polygonicTR_completeCartier}
In this section, we introduce the $\infty$-category of quasifinitely genuine $G$-spectra for an arbitrary group $G$ as first introduced in its $\Z$-linear version by Kaledin in~\cite{Kal14} under the name of $G$-Mackey profunctors. In~\textsection\ref{subsection:qfgen_G} and~\textsection\ref{subsection:monoidal}, we will discuss structural features of the $\infty$-category of quasifinitely genuine $G$-spectra including monoidal structures and geometric fixedpoints. In~\textsection\ref{subsection:qfgen_Z}, we focus on the particular case when $G = \Z$, where we establish some structural results about the geometric fixedpoints functors which are special to the case of $G = \Z$. This section provides the necessary background for our discussion of TR of a polygonic spectrum as a quasifinitely genuine $\Z$-spectrum in~\textsection\ref{subsection:qfgenZ_TR}. 

\subsection{Definition of quasifinitely genuine $G$-spectra} \label{subsection:qfgen_G}
In this section, our goal is to introduce the $\infty$-category of quasifinitely genuine $G$-spectra for an arbitrary group $G$ following Kaledin~\cite{Kal14}. The formalism of quasifinitely genuine spectra provides a reasonable theory of genuine equivariant homotopy theory for infinite discrete groups such as the group of integers which will be our main example of interest. Recall that if $G$ is a finite group, then the $\infty$-category of genuine $G$-spectra is defined as the $\infty$-category of spectral Mackey functors on the category of finite $G$-sets (cf.~\cite{Bar17,GM11}). This definition is meaningful for an arbitrary group but it will in general not produce a formalism which possesses the desired features of an $\infty$-category of genuine $G$-spectra. To avoid confusion about which variant of equivariant homotopy theory we are considering, we recall the following:

\begin{definition}
For an arbitrary group $G$ let $\Fin_G$ denote the category of finite $G$-sets, that is finite sets with a $G$-action. 
\end{definition}

Every object $S$ of $\Fin_G$ can be written as a finite coproduct of orbits of the form $G/H$, where $H$ is a cofinite subgroup of $G$. In fact, there is a canonical inclusion $\mathrm{Orb}_G \hookrightarrow \Fin_G$ from the cofinite orbit category of $G$, which exhibits the latter as the finite coproduct completion of $\mathrm{Orb}_G$. Note that $\Fin_G$ admits pullbacks, allowing us to form the span $\infty$-category $\mathrm{Span}(\Fin_G)$ which admits direct sums given by coproducts of finite $G$-sets (cf.~\cite[Proposition 4.2]{Bar17}).

\begin{definition} \label{definition:finitely_genuine_G_spectra}
Let $G$ be a group. The $\infty$-category of finitely genuine $G$-spectra is defined as
\[
	\Sp^G_{\fgen} = \Fun^{\mathrm{add}}(\mathrm{Span}(\Fin_G), \Sp),
\]
the full subcategory spanned by those functors $\mathrm{Span}(\Fin_G) \to \Sp$ which preserve finite products (equivalently finite coproducts). 
\end{definition}

We warn the reader that typically in equivariant homotopy theory the subindex refers to the family of subgroups that can arise as stabilisers. From this perspective it would be more logical to denote this $\infty$-category by $\Sp^G_{\mathrm{cofin}}$ to indicate that for finite $G$-sets the stabilizers are cofinite.  

If $G$ is finite, then $\Sp^{G}_{\fgen}$ is equivalent to the $\infty$-category of genuine $G$-spectra as defined via orthogonal spectra (cf.~\cite{MM02}). We remark that Definition~\ref{definition:finitely_genuine_G_spectra} also produces a reasonable formalism when $G$ is a profinite group (cf.~\cite{Bar17,Nar16}). The main idea of Kaledin~\cite{Kal14} is to replace the category of finite $G$-sets above with the category of quasifinite $G$-sets in the sense of Dress--Siebeneicher~\cite{DS88}:

\begin{definition} \label{definition:quasifinite_G_set}
Let $G$ be a group. A quasifinite $G$-set is a $G$-set $S$ which satisfies the following:
\begin{enumerate}[leftmargin=2em, topsep=5pt, itemsep=5pt]
	\item For every element $s \in S$, the stabilizer $G_s$ is a cofinite subgroup of $G$.
	\item For every cofinite subgroup $H$ of $G$, the fixedpoints set $S^H$ is finite.
\end{enumerate}
Let $\QFin_G$ denote the full subcategory of $\mathrm{Set}_G$ spanned by those $G$-sets which are quasifinite. 
\end{definition}

Note that every finite $G$-set gives an example of a quasifinite $G$-set, and there is a fully faithful functor $i : \Fin_G \hookrightarrow \QFin_G$. We think of the category $\QFin_G$ as providing an intermediate between $\Fin_G$ and the full coproduct completion of $\mathrm{Orb}_G$, in the sense that we may form certain infinite coproducts of $G$-orbits: If $S \in \QFin_G$ with quotient map $\pi : S \to S/G$, then the canonical map
\[
	\coprod_{\bar{x} \in S/G} \pi^{-1}(\bar{x}) \to S
\]
is an isomorphism of quasifinite $G$-sets, and a choice of lift $s$ of $\bar{s} \in S/G$ determines an isomorphism $\pi^{-1}(\bar{s}) \simeq G/G_s$. Said more informally, if we write our $G$-set $S$ as a coproduct of orbits
\[
	S = \coprod_{s \in S/G} G/H_s,
\]
then $S$ is quasifinite if each $H_s$ is a cofinite subgroup of $G$ and each cofinite subgroup $H$ of $G$ is contained in at most finitely many of the $H_s$.
We have the following instructive example:

\begin{example}
We claim that $S = \coprod_{n \geq 1} \Z/n\Z$ is a quasifinite $\Z$-set. The first condition is clear since the stabilizers are given by $n\Z$ as $n$ ranges over $\N$. Furthermore, if $k \geq 1$, then
\[
	S^{k\Z} = \coprod_{n \in \langle k \rangle} \Z/n\Z
\]
which is finite, showing that the second condition is satisfied. In general, every quasifinite $\Z$-set which is not contained in the essential image of the functor $i : \Fin_{\Z} \hookrightarrow \QFin_{\Z}$ is of the form
\[
	T = \coprod_{i \in I} \Z/m_i\Z
\]
for a sequence $\{m_i\}_{i \in I}$ of integers $m_i \geq 1$ for every $i \in I$, which satisfies that $m_i \to \infty$ for $i \to \infty$. If $T'$ is an additional object of $\QFin_{\Z}$ specified by a sequence $\{n_j\}_{j \in J}$ as above, then a morphism $T \to T'$ in $\QFin_{\Z}$ is determined by a function $f : I \to J$ such that $n_{f(i)}$ divides $m_i$ for every $i \in I$. We will use these observation repeatedly in~\textsection\ref{subsection:qfgen_Z} and~\textsection\ref{subsection:qfgenZ_TR}. 
\end{example}

We claim that the category $\QFin_G$ of quasifinite $G$-sets admits pullbacks. Indeed, pullbacks are formed in the category of $G$-sets and one verifies that these are again quasifinite: the stabilizer of elements is the intersection of the stabilizers, hence again cofinite. The fixedpoints of a pullback is the pullback of fixedpoints, hence again finite. Now since $\QFin_G$ has pullbacks we may form the $\infty$-category $\mathrm{Span}(\QFin_G)$. Note that this category is not the $(\infty, 2)$-category of spans but rather the $\infty$-category (aka $(\infty,1)$-category). That is a 2-morphism between spans consists of an equivalence between the involved morphism objects. 

The following definition is due to Kaledin~\cite[Definition 3.2]{Kal14}. 

\begin{definition} \label{definition:quasifinitely_genuine_G_spectra}
For a group $G$, the $\infty$-category of quasifinitely genuine $G$-spectra is defined by
\[
	\Sp^{G}_{\qfgen} = \Fun^{\mathrm{vadd}}(\mathrm{Span}(\QFin_G), \Sp),
\]
where $\Fun^{\mathrm{vadd}}(\mathrm{Span}(\QFin_G), \Sp)$ denotes the full subcategory of $\Fun(\mathrm{Span}(\QFin_G), \Sp)$ spanned by those functors $X$ which are very additive in the sense that the canonical map of spectra
\[
	X(S) \to \prod_{\bar{s} \in S/G} X(\pi^{-1}(\bar{x}))
\]
is an equivalence for every quasifinite $G$-set $S$ with quotient map $\pi : S \to S/G$. 
\end{definition}

The fully faithful functor $i : \Fin_G \hookrightarrow \QFin_G$ induces a restriction functor of $\infty$-categories
\[
	i^{\star} : \Sp^{G}_{\qfgen} \to \Sp^{G}_{\fgen}
\]
allowing us to regard every quasifinitely genuine $G$-spectrum $X$ as a finitely genuine $G$-spectrum. In particular, we find that $X$ supports the usual restriction and transfer functoriality: If $H$ is a cofinite subgroup of $G$, then the terminal map $G/H \to G/G$ induces a restriction map
\[
	F_H : X^G \to X^H
\]
using the contravariant functoriality, which is equivariant with respect to the action of the Weyl group $W_G(H) = N_G(H)/H$ on the target. The covariant functoriality, induces a transfer map
\[
	V_H : X^H \to X^G
\]
which is equivariant for the action of the Weyl group $W_G(H)$ on $X^{H}$. The additional structure encoded by the quasifinitely genuine $G$-structure on $X$ which is not presently available on the underlying finitely genuine $G$-spectrum $i^{\star}X$ is that certain infinite sums of transfer maps exist. If $S = \coprod_{i \in I} G/H_i$ is a quasifinite $G$-set, then the terminal map $S \to G/G$ induces a map
\[
	V_{\{H_i\}} : \prod_{i \in I} X^{H_i} \to X^{G}
\]
such that the composite map of spectra
\[
	X^{H_{i_0}} \hookrightarrow \displaystyle\bigoplus_{i \in I} X^{H_i} \to \prod_{i \in I} X^{H_i} \xrightarrow{V_{\{H_i\}}} X^{G},
\]
is canonically equivalent to the transfer map $V_{H_{i_0}} : X^{H_{i_0}} \to X^G$ for each $i_0 \in I$. Informally, this can be summarized by saying that the map $V_{\{H_i\}}$ is given by the infinite sum
\[
	V_{\{H_i\}} = \sum_{i \in I} V_{H_i}.
\]
Even more informally, we think of a quasifinitely genuine $G$-spectrum as a spectrum equipped with restriction and transfer functoriality for the cofinite subgroups of $G$ together with a specification of which infinite sums of transfer maps converge. This is the beautiful insight of Kaledin~\cite{Kal14}. We begin by establishing some basic facts about $\Sp^G_{\qfgen}$. We stress that $G$ is an arbitrary group.

\begin{lemma} \label{lemma:spGqfgen_closed_limits}
The $\infty$-category $\Sp^{G}_{\qfgen}$ is presentable and the  inclusion 
\[
\Sp^{G}_{\qfgen} \hookrightarrow \Fun(\mathrm{Span}(\QFin_G), \Sp)
\]
 admits a left adjoint $(-)_{\mathrm{vadd}}$.
\end{lemma}

\begin{proof}
The $\infty$-category $\Fun(\mathrm{Span}(\QFin_G), \Sp)$ is presentable and we claim that $\Sp^{G}_{\qfgen}$ are the local objects for a presentable Bousfield localization. Indeed, it suffices to check that the very additive functors are exactly the local objects with respect to a set of maps (i.e. the Bousfield localization is generated by a set of maps). To see this, we take the set of all shifts of the maps
\[
\bigoplus_{i \in I} \Sigma^\infty_+ \underline{S_i} \to \Sigma^\infty_+\Big( \underline{\coprod_{i \in I} S_i } \Big)
\]
for (a representative of) every object $S =  \coprod S_i$ in $\QFin_G$, which is a set since $\QFin_G$ is essentially small. Here the underline denotes the corepresentable functor where we consider them as objects of the span category. The map is induced from the inclusions $S_i \to S$ (considered as  backward morphisms in the span category).  
\end{proof}

\begin{warning} \label{warning:colimits_spgqfgen} Limits in  $\Sp^{G}_{\qfgen}$ are formed as limits of the underlying diagram in the $\infty$-category $\Fun(\mathrm{Span}(\QFin_G), \Sp)$, but 
colimits are not formed underlying. Instead colimits are obtained by applying $(-)_{\mathrm{vadd}}$ to the underlying colimit formed in $\Fun(\mathrm{Span}(\QFin_G), \Sp)$. 
\end{warning}

\begin{remark}
Note that $i^{\star} : \Sp^G_{\qfgen} \to \Sp^G_{\fgen}$ preserves limits and $\kappa$-filtered colimits for $\kappa$ large enough\footnote{Precisely we need that the products appearing in the definition of quasifinitely genuine $G$-spectra commute with $\kappa$-filtered colimits. For this we need to choose $\kappa$ so that the cofinality is larger than the number of conjugacy classes of cofinite subgroups of $G$}, so it admits a left adjoint $(-)_{\qfgen}$. 
\end{remark}

There is a functor of $\infty$-categories $\Omega^\infty : \Sp^{G}_{\qfgen} \to \Spaces^{G}$ given by the following composite
\[
	\Sp^{G}_{\qfgen} \to \Sp^{G}_{\fgen} \xrightarrow{\Omega^\infty} \Fun^{\times}(\mathrm{Span}(\Fin_G), \Spaces) \to \Fun^{\times}(\Fin_G^{\op}, \Spaces) \simeq \Fun(\mathrm{Orb}_G^{\op}, \Spaces),
\]
which preserves limits and is accessible by Lemma~\ref{lemma:spGqfgen_closed_limits}, so it admits a left adjoint $\Sigma^\infty_+ : \Spaces^{G} \to \Sp^{G}_{\qfgen}$. The left adjoint $\Sigma^\infty_+ $ can also be described by its value on orbits (which freely generated  $\Spaces^{G}$). This is the content of the next statement:

\begin{lemma} \label{lemma:fixedpoints_corepresented}
Let $G$ be a group. For every cofinite subgroup $H$ of $G$, the functor of $\infty$-categories
\[
	(-)^H : \Sp^{G}_{\qfgen} \to \Sp^{BW_G(H)}
	\]
is corepresented by $\Sigma^\infty_+(G/H)$ equipped with its right $W_G(H)$-action.
\end{lemma}

\begin{proof}
For every quasifinitely genuine $G$-spectrum $X$, we find that
\[
	\Map_{\Sp^G_{\qfgen}}(\Sigma^\infty_+(G/H), X) \simeq \Map_{\Spaces^G}(G/H, \Omega^\infty X) \simeq \Omega^\infty X(G/H)
\]
by the Yoneda lemma since $G/H$ is the presheaf represented by $G/H$ as a $G$-space. This proves the desired statement since both $X \mapsto X^H$ and $X \mapsto \map_{\Sp^G_{\qfgen}}(\Sigma^\infty_+(G/H), X)$ are exact functors which agree after applying $\Omega^\infty$. 
\end{proof}

\begin{warning} \label{warning:generators_spGqfgen}
Lemma~\ref{lemma:fixedpoints_corepresented} implies that the set $\{\Sigma^\infty_+(G/H)\}_H$ indexed by the cofinite subgroups of $G$ forms a set of generators of $\Sp^{G}_{\qfgen}$, but these generators are not compact since colimits are not formed underlying. This marks one of the main differences between $\Sp^{G}_{\qfgen}$ and $\Sp^G_{\fgen}$. 
\end{warning}

\begin{example}
If $G$ is a finite group, then the functor $i^\star : \Sp^G_{\qfgen} \to \Sp^G_{\fgen}$ is an equivalence of $\infty$-categories since the inclusion functor $i : \Fin_G \hookrightarrow \QFin_G$ is an equivalence.
\end{example}

We construct a $t$-structure on the $\infty$-category of quasifinitely genuine $G$-spectra. The important feature is that this $t$-structure is pointwise in contrast to the $t$-structure on the $\infty$-category of polygonic spectra constructed in~\textsection\ref{subsection:polygonic_t_structure}. 

\begin{proposition} \label{proposition:tstructure_Spqfgen}
The following pair of subcategories
\begin{align*}
	(\Sp^G_{\qfgen})_{\geq 0} &= \Fun^{\vadd}(\mathrm{Span}(\QFin_G), \Sp_{\geq 0}) \\
	(\Sp^G_{\qfgen})_{\leq 0} &= \Fun^{\vadd}(\mathrm{Span}(\QFin_G), \Sp_{\leq 0})
\end{align*}
determines a $t$-structure on the $\infty$-category $\Sp^{G}_{\qfgen}$ with $\Sp^{G, \heartsuit}_{\qfgen} \simeq \Fun^{\vadd}(\mathrm{hSpan}(\QFin_G), \mathrm{Ab})$.
\end{proposition}

In other words, a quasifinitely genuine $G$-spectrum $X$ is connective provided that the spectrum $X^S$ is connective for every quasifinite $G$-set $S$ and $X$ is coconnective if $X^S$ is coconnective for every $S$. We remark that the proof of Proposition~\ref{proposition:tstructure_Spqfgen} proceeds as the proof that the $\infty$-category of spectral Mackey functors on an orbital $\infty$-category admits a $t$-structure (cf.~\cite[Prop. 6.1]{BGS19}).

\begin{proof}
Consider the functor $F$ of $\infty$-categories defined by the following composite
\[
	\Fun^{\vadd}(\mathrm{Span}(\QFin_G), \Sp) \hookrightarrow \Fun(\mathrm{Span}(\QFin_G), \Sp) \xrightarrow{\tau_{\leq -1}(-)_{\vadd}} \Fun^{\vadd}(\mathrm{Span}(\QFin_G), \Sp)
\]
and note $F$ is a localization since both $\tau_{\leq -1}$ and $(-)_{\vadd}$ are localizations. Using that $\Sp_{\leq -1}$ is closed under arbitrary products, we conclude that the essential image of $F$ is equivalent to
\[
	\Fun^{\vadd}(\mathrm{Span}(\QFin_G), \Sp_{\leq -1})
\]
which is closed under extensions since $\Sp_{\leq -1}$ is and products in $\Sp_{\leq -1}$ are calculated in $\Sp$ where they are exact. The desired assertion now follows from~\cite[Proposition 1.2.1.16]{Lur17}.
\end{proof}

We will use the following terminology (see also Remark~\ref{remark:geometric_uniformly_boundedbelow}). 

\begin{definition}
A quasifinitely genuine $G$-spectrum $X$ is uniformly bounded below provided that the collection of spectra $\{X^S\}_{S \in \QFin_G}$ is uniformly bounded below.
\end{definition}

In~\cite[Proposition 3.2.11]{McC21}, the second author established a variant of the classical tom Dieck splitting for spectral Mackey functors on the finite coproduct completion of an orbital $\infty$-category. More precisely, the tom Dieck splitting can be regarded as providing a formula for the free spectral Mackey functor of a spectral presheaf on an orbital $\infty$-category. It would be desirable to obtain a similar formula for quasifinitely genuine $G$-spectra. Currently, we do not know how to do this (cf. Remark~\ref{remark:tomdieck_infinite_product}). However, we will require a weaker version of such a splitting for our discussion of geometric fixedpoints in~\textsection\ref{subsection:geometricfixedpoints}. This will be the content of the rest of this section. The reader might want to skip this on a first reading and proceed to~\textsection\ref{subsection:monoidal}. 

\begin{remark} \label{remark:tomdieck_infinite_product}
One could conjecture that for any $G$-space $X$, the associated quasifinitely genuine $G$-spectrum $\Sigma^\infty_+ X$ satisfies the formula
\[
(\Sigma^\infty_+ X)^G \simeq \prod_{(K)} (X^K)_{hW_K}
\]
where the product is indexed over the conjugacy classes of cofinite subgroups $K$ of $G$. We have been unable to prove this. The main obstruction is that we could not decide whether the group completion of $\E_\infty$-monoids commutes with infinite products. 
\end{remark}

We will however obtain an analog of the tom Dieck splitting in the slightly larger $\infty$-category of additive functors $\Fun^{\add}(\mathrm{Span}(\QFin_G), \Sp)$. 
This statement will be crucial for a proof in our discussion of the geometric fixedpoints functor on $\Sp^{G}_{\qfgen}$ in~\textsection\ref{subsection:geometricfixedpoints}.
First consider the functor
\[
	i_! : \Fun(\QFin_G^{\op}, \Spaces) \to \Fun(\mathrm{Span}(\QFin_G), \Spaces) 
\]
obtained by forming the left Kan extension along the functor $i: \QFin_G^{\op} \to \mathrm{Span}(\QFin_G)$.

\begin{lemma} \label{Kan}
If $E \in  \Fun(\QFin_G^{\op}, \Spaces)$, then the left Kan extension $i_!(E)$ is given by the functor
$
	\mathrm{Span}(\QFin_G) \to \Spaces
$
determined by the construction 
\[
S \mapsto \left((\QFin_G)_{/E \times \underline{S}}\right)^\simeq  \subseteq \mathrm{Map}_{\mathrm{Span}(\Fun(\QFin_G^{\op}, \Spaces))}(E, \underline{S}). 
\]
The term $\left((\QFin_G)_{/E \times \underline{S}}\right)^\simeq$ is the comma category of the Yoneda inclusion $\QFin_G \to \Fun(\QFin_G^{\op}, \Spaces)$ over $E \times \underline{S}$. It is considered as the full subspace of the mapping space $\mathrm{Map}_{\mathrm{Span}(\Fun(\QFin_G^{\op}, \Spaces))}(E, S)$ on those spans whose middle term is in $\QFin_G$. The latter description also makes the functoriality clear.  
\end{lemma}

\begin{proof}
The assignment that takes $E$ to the functor $\mathrm{Span}(\QFin_G) \to \Spaces$ described above is canonically a covariant functor in $E$ (e.g. by the desciption through spans or by a straightening construction for slices). Moreover since the objects of $\QFin_G$ are absolutely compact in $\Fun(\QFin_G^{\op}, \Spaces)$ it follows that this assignment preserves colimits\footnote{An object $c \in \cat{C}$ is called absolutely compact if $\Map_{\cat{C}}(c, -): \cat{C} \to \Spaces$ preserves all colimits. Then  the assignment 
\[
\cat{C} \to \Spaces \qquad d \mapsto \left((\cat{C}^{\mathrm{ac}})_{/d}\right)^\simeq
\]
preserves colimits, where $\cat{C}^{\mathrm{ac}} \subseteq \cat{C}$ is the full subcategory on the absolutely compact objects. Indeed, note that
\[
\left((\cat{C}^{\mathrm{ac}})_{/d}\right)^\simeq \simeq \colim_{c \in (\cat{C}^{\mathrm{ac}})^\simeq} \Map_{\cat{C}}(c,d)
\]
which commutes with colimits in $d$.}, so its uniquely determined by the composition
\[
\QFin_G \to \Fun(\QFin_G^{\op}, \Spaces) \to  \Fun(\mathrm{Span}(\QFin_G), \Spaces).
\]
This composition clearly agrees with the corresponding composition for the left Kan extension, since the Kan extension sends representables to representables (here corepresentables since we work with copresheaves on $\mathrm{Span}(\QFin_G))$.
\end{proof}

Next, we wish to describe the left adjoint to the functor 
\[
	\Fun^{\add}(\mathrm{Span}(\QFin_G), \Sp) \to \Fun^{\times}(\QFin_G^{\op}, \Spaces),
\]
which we denote by $\widetilde{\Sigma}^\infty_+$. Here the right hand side denotes the full subcategory of $\Fun(\QFin_G^{\op}, \Spaces)$ spanned by those functors which preserve finite products. We establish a formula for $\widetilde{\Sigma}^\infty_+$.

\begin{lemma} \label{lemma:tomDieck_splitting}
If $E \in \Fun^{\times}(\QFin^{\op}_G, \Spaces)$, then $i_! E$ preserves finite products. Thus, the functor $i_! E$ canonically refines to a functor of $\infty$-categories
\[
	\mathrm{Span}(\QFin_G) \to \mathrm{Mon}_{\E_\infty}(\Spaces)
\]
determined by $S \mapsto \left((\QFin_G)_{/E \times \underline{S}}\right)^\simeq  \subseteq \mathrm{Map}_{\mathrm{Span}(\Fun(\QFin_G^{\op}, \Spaces))}(E, \underline{S})$ and $\widetilde{\Sigma}^{\infty}_+ E$ is given by the pointwise group completion of this additive functor.
\end{lemma}

\begin{proof}
We wish to prove that if $E$ preserves products, then the left Kan extension $i_! E$ given by the formula in Lemma~\ref{Kan} preserves finite products. 
To see this, we note that finite products in $\mathrm{Span}(\QFin_G)$ are given by disjoint unions of quasifinite $G$-sets. To check whether a functor $E': \mathrm{Span}(\QFin_G) \to \Spaces$ preserves products, we may restrict along
\[
\QFin_G^\op \to \mathrm{Span}(\QFin_G) \to \Spaces
\]
and check if this preserves finite products. In the case of interest this restriction is given by
\[
S \mapsto \left((\QFin_G)_{/E \times \underline{S}}\right)^\simeq 
\]
with the `pullback' functoriality. For two quasifinite $G$-sets $S$ and $T$, we then need to verify that 
\[
\left((\QFin_G)_{/E \times \underline{S \sqcup T}}\right)^\simeq \to \left((\QFin_G)_{/E \times \underline{S}}\right)^\simeq  \times \left((\QFin_G)_{/E \times \underline{T}}\right)^\simeq 
\]
is an equivalence. For this to be true it would be sufficient to have that $\underline{S \sqcup T} = \underline{S} \sqcup \underline{T}$ which is unfortunately false in 
$\Fun(\QFin^{\op}_G, \Spaces)$. Instead, we can work in the full subcategory
\[
\Fun^{\times}(\QFin^{\op}_G, \Spaces) \subseteq \Fun(\QFin^{\op}_G, \Spaces)
\]
which contains the Yoneda image of $\QFin_G$ and the functor $E$ by virtue of the assumption that $E$ preserves finite products. Now we interpret 
\[
\left((\QFin_G)_{/E \times \underline{S}}\right)^\simeq
\]
as the comma category over the functor 
\[
\QFin_G \to \Fun^\times(\QFin_G^{\op}, \Spaces)
\]
and note that this inclusion does preserves coproducts. In other words, we have that 
$\underline{S \sqcup T} = \underline{S} \sqcup \underline{T}$ in the $\infty$-category $\Fun^\times(\QFin_G^{\op}, \Spaces)$. Therefore we find that 
\[
\left((\QFin_G)_{/E \times \underline{S\sqcup T} }\right)^\simeq \simeq \left((\QFin_G)_{/ (E\times \underline{S}) \sqcup (E \times \underline{T})} \right)^\simeq
\]
and since $\mathcal{X} = \Fun^\times(\QFin_G^{\op}, \Spaces)$ is an $\infty$-topos we have that the canonical map $\mathcal{X}_{/ x \sqcup y} \xrightarrow{\simeq} \mathcal{X}_{/x} \times \mathcal{X}_{/y}$ is an equivalence with inverse given by taking the coproduct. In our situation this equivalence restricts to an equivalence on the full subcategories spanned by those objects whose total space lies in $\QFin_G \subseteq \mathcal{X}$. This finishes the argument that $i_!(E)$ preserves binary products. The fact that it preserves the terminal object works similar, but easier and is left to the reader.  

Consequently, the left adjoint of the restriction functor
\[
\Fun^{\add}(\mathrm{Span}(\QFin_G), \Spaces) \to \Fun^{\times}(\QFin_G^{\op}, \Spaces)
\]
is described by the formula for $i_!$ and the forgetful functor $\mathrm{Mon}_{\E_\infty}(\Spaces) \to \Spaces$ induces an equivalence 
\[
	\Fun^{\add}(\mathrm{Span}(\QFin_G), \mathrm{Mon}_{\E_\infty}(\Spaces)) \xrightarrow{\simeq} \Fun^{\add}(\mathrm{Span}(\QFin_G), \Spaces)
\]
since $\mathrm{Span}(\QFin_G)$ is semiadditive (cf.~\cite[Corollary 2.5]{GGN16}). We note that the forgetful functor
\[
	\Fun^{\add}(\mathrm{Span}(\QFin_G), \Sp) \to \Fun^{\add}(\mathrm{Span}(\QFin_G), \mathrm{Mon}_{\E_\infty}(\Spaces))
\]
admits a left adjoint given by forming the pointwise group completion (which also is additive since group completion commutes with finite products). Combining these assertions proves the desired statement.
\end{proof}

\subsection{Monoidal properties}\label{subsection:monoidal}
In this section, we will provide a symmetric monoidal structure on the $\infty$-category $\Sp^G_{\qfgen}$ of quasifinitely genuine $G$-spectra with the property that the functor
\[
\Sigma^\infty_+: \Spaces^G \to \Sp^G_{\qfgen}
\] 
is symmetric monoidal and such that the object $\Sigma^\infty_+(S)$ is self-dual for every $S \in \QFin_G$.

\begin{remark}
We also provide a symmetric monoidal structure on the larger $\infty$-category 
\[
\Fun^{\mathrm{add}}(\mathrm{Span}(\QFin_G), \Sp)
\]
such that the left adjoint functors of $\infty$-categories
\[
\Spaces^G \to \Fun^{\mathrm{add}}(\mathrm{Span}(\QFin_G), \Sp) \to \Sp^G_{\qfgen} 
\]
are symmetric monoidal. We remark that it is not clear what the intrinsic/geometric meaning of the $\infty$-category $\Fun^{\mathrm{add}}(\mathrm{Span}(\QFin_G), \Sp)$ is, but it will be technically convenient to work there for some arguments since we do not have to worry about infinite products (cf. Remark~\ref{remark:tomdieck_infinite_product}).
\end{remark}

The category $\QFin_G$ of quasifinite $G$-sets has cartesian products since it has pullbacks and the trivial $G$-set given by the one-point space is quasifinite. Therefore, we may regard the category $\QFin_G$ with the cartesian symmetric monoidal structure. Below, we construct an induced symmetric monoidal structure on $\mathrm{Span}(\QFin_G)$ given on objects by forming the cartesian product in $\QFin_G$. This was also carried out in~\cite{BGS19}.

\begin{construction}
The construction $\mathrm{Span}(-)$ defines a functor $\Cat_\infty^{\mathrm{LEx}} \to \Cat_\infty$ where the source denotes the $\infty$-category of $\infty$-categories with finite limits and finite limit preserving functors\footnote{In fact, pullbacks would be enough but we do not need this generality here.}. The functor $\mathrm{Span}$ preserves products meaning that the canonical morphism 
\[
	\mathrm{Span}(\cat{C} \times \cat{D}) \to \mathrm{Span}(\cat{C}) \times \mathrm{Span}(\cat{D})
\]
is an equivalence for each pair of $\infty$-categories with finite limits. 

An $\infty$-category $\cat{C}$ with finite products gives rise to a cartesian symmetric monoidal $\infty$-category which straightens to an $\E_\infty$-monoid in $\Cat_\infty$, i.e. a functor of $\infty$-categories
\[
\mathrm{Fin}_\ast \to \Cat_\infty
\]
satisfying the Segal condition. If $\cat{C}$ has all finite limits then this functor factors trough $\Cat_\infty^{\mathrm{LEx}} \to \Cat_\infty$, since products commute with finite limits. We postcompose this functor with the functor $\mathrm{Span}$ to obtain the desired $\E_\infty$-monoid structure on $\mathrm{Span}(\cat{C})$ in $\Cat_\infty$.

The symmetric monoidal $\infty$-category $\mathrm{Span}(\cat{C})$ has the property that every object $c$ of $\cat{C}$ is canonically self-dual when considered as an object of $\mathrm{Span}(\cat{C})$. To see this one simply exhibits the evaluation and coevaluation of the duality given by the span 
\[
	c \times c \xleftarrow{\Delta} c \to \mathrm{pt}
\]
once read from the left to the right and once in the reverse direction. The zigzag identities are then readily checked since the composition
\[
\begin{tikzcd}
&& c \ar[ld,"\Delta"'] \ar[rd, "\Delta"]\\
{} & c \times c \ar[rd, "\id \times \Delta"]  \ar[ld, "\mathrm{pr}_1"'] & & c \times c \ar[ld, "\Delta \times \id"'] \ar[rd, "\mathrm{pr}_2"] \\
c & & c \times c \times c& & c
\end{tikzcd}
\]
is the identity. 
\end{construction}

The symmetric monoidal structure on $\mathrm{Span}(\QFin_G)$ induces the Day convolution symmetric monoidal structure on the $\infty$-category $\Fun(\mathrm{Span}(\QFin_G), \Sp)$ with the following properties:

\begin{enumerate}[leftmargin=2em, topsep=5pt, itemsep=5pt]
	\item The symmetric monoidal structure on $\Fun(\mathrm{Span}(\QFin_G), \Sp)$ is closed and the internal hom can be described by the following formula
	\[
		\underline{\Hom}(X, Y)(S) \simeq \int_{T \in \QFin_G} \map_{\Sp}(X(T), Y(T \times S))
	\]
	for every pair of functors $X, Y : \mathrm{Span}(\QFin_G) \to \Sp$ and every quasifinite $G$-set $S$. 

	\item The contravariant Yoneda functor of $\infty$-categories
	\[
		\mathrm{Span}(\QFin_G)^{\mathrm{op}} \to \Fun(\mathrm{Span}(\QFin_G), \Sp)
	\]
	given by the assignment $S \mapsto \underline{S}$ with $\underline{S}(T) = \Sigma^\infty_+ \Map_{\mathrm{Span}(\QFin_G)}(S, T)$ canonically refines to a symmetric monoidal functor.
\end{enumerate}

In particular, the second property above implies that for every quasifinite $G$-set $S$, the functor $\underline{S}$ is canonically self-dual as an object of the symmetric monoidal $\infty$-category $\Fun(\mathrm{Span}(\QFin_G), \Sp)$. Furthermore, the following composite functor
\[
	\QFin_G \to \mathrm{Span}(\QFin_G)^{\op} \to \Fun(\mathrm{Span}(\QFin_G), \Sp)
\]
refines to a symmetric monoidal functor, thus also its colimit preserving extension
\[
	\Fun(\QFin_G^{\op}, \Spaces) \to \Fun(\mathrm{Span}(\QFin_G), \Sp),
\]
where the source carries the cartesian symmetric monoidal structure (which agrees with Day convolution symmetric monoidal structure with respect to the Cartesian product on $\QFin_G$). 

\begin{proposition} \label{proposition:qfgen_symmetric_monoidal_localization}
The $\infty$-categories $\Sp^{G}_{\qfgen}$ and $\Fun^{\mathrm{add}}(\mathrm{Span}(\QFin_G), \Sp)$ are both symmetric monoidal localizations of $\Fun(\mathrm{Span}(\QFin_G), \Sp)$.
\end{proposition}

\begin{proof}
We prove that the $\infty$-category of quasifinitely genuine $G$-spectra is a symmetric monoidal localization of $\Fun(\mathrm{Span}(\QFin_G), \Sp)$. Let $X$ and $Y$ denote a pair of functors $\mathrm{Span}(\QFin_G) \to \Sp$ and assume that $Y$ is very additive. We verify that the internal hom $\underline{\Hom}(X, Y)$ is very additive. Indeed, for every quasifinite $G$-set $S = \coprod_{i \in I} S_i$, we have that
\[
	Y\Big( T \times \coprod_{i \in I} S_i\Big) \simeq Y\big( \coprod_{i \in I} T \times S_i\big) \simeq \prod_{i \in I} Y(T \times S_i),
\] 
using that products distribute over arbitrary coproducts in $\QFin_G$ and our assumption that $Y$ is very additive. Combining this with the formula for the internal hom above, we find that
\begin{align*}
	\underline{\Hom}(X, Y)(S) \simeq \int_{T \in \QFin_G} \map_{\Sp}\Big(X(T), \prod_{i \in I} Y(T \times S_i) \Big) \simeq \prod_{i \in I} \underline{\Hom}(X, Y)(S_i)
\end{align*}
which proves the desired statement. Finally, the analogues assertion for $\Fun^{\mathrm{add}}(\mathrm{Span}(\QFin_G), \Sp)$ is obtained by replacing every occurence of very additive by additive in the argument above.
\end{proof}

It follows from Proposition~\ref{proposition:qfgen_symmetric_monoidal_localization} that the following functors
\[
	\Fun(\QFin_G^{\op}, \Spaces) \to \Fun(\mathrm{Span}(\QFin_G), \Sp) \to \Fun^{\mathrm{add}}(\mathrm{Span}(\QFin_G), \Sp) \to \Sp^{G}_{\qfgen}
\]
are symmetric monoidal. The composite above factors over the symmetric monoidal localization
\[
	\Fun(\QFin_G^{\op}, \Spaces) \to \Fun^{\mathrm{vadd}}(\QFin_G^{\op}, \Spaces) \simeq \Fun(\mathrm{Orb}_G^{\op}, \Spaces)
\]
and the induced functor is given by the functor $\Sigma^\infty_+: \Fun(\mathrm{Orb}_G^{\op}, \mathcal{S}) \to  \Sp^{G}_{\qfgen}$ which follows by comparing adjoints. In particular, we conclude that the functor $\Sigma^\infty_+$ has a canonical symmetric monoidal structure. Similarly, we have that the composite 
\[
\Fun(\QFin_G^{\op}, \mathcal{S})  \to  \Fun^{\mathrm{add}}(\mathrm{Span}(\QFin_G), \Sp)
\]
still factors to a symmetric monoidal functor $\Fun^{\mathrm{add}}(\QFin_G^{\op}, \mathcal{S})   \to  \Fun^{\mathrm{add}}(\mathrm{Span}(\QFin_G), \Sp)$.

To summarize, we have the following diagram of symmetric monoidal functors:
\[
  \begin{tikzcd}
    \QFin_G \rar\dar& \Span(\QFin_G)^\op\dar\\
    \Fun(\QFin_G^\op, \cS)\rar\dar & \Fun(\Span(\QFin_G),\Sp)\dar\\
    \Fun^\add(\QFin_G^\op, \cS)\rar\dar & \Fun^\add(\Span(\QFin_G),\Sp)\dar\\
    \Fun^\vadd(\QFin_G^\op, \cS)\rar \dar{\simeq} & \Fun^\vadd(\Span(\QFin_G),\Sp) \dar{\simeq} \\
    \cS^G \rar{\Sigma^\infty_+} & \Sp^G_{\qfgen}
  \end{tikzcd}
\]

\subsection{Geometric fixedpoints} \label{subsection:geometricfixedpoints}

The $\infty$-category of quasifinitely genuine $G$-spectra supports geometric fixedpoint functors as we will now discuss.

\begin{construction} \label{construction:geometric_fixedpoints_functor}
Let $G$ denote an arbitrary group. For every cofinite subgroup $H$ of $G$, the fixedpoints set $S^H$ is finite by definition, so the construction $S \mapsto S^H$ determines a functor
\[
	(-)^H : \QFin_G \to \Fin_{W_G(H)}.
\]
The functor $(-)^H$ preserves pullbacks and the restriction functor along $\mathrm{Span}(-)^H$ carries additive functors to very additive functors, hence canonically refines to a functor of $\infty$-categories
\[
\mathrm{infl}_H : \Sp^{W_G(H)}_{\fgen} \to \Sp^G_{\qfgen}
\]
referred to as the inflation functor. This functor preserves limits and $\kappa$-filtered colimits for large enough $\kappa$, since these are formed underlying in the source and the target, so the inflation functor admits a left adjoint functor
\[
	(-)^{\Phi H} : \Sp^{G}_{\qfgen} \to \Sp^{W_G(H)}_{\fgen}
\]
referred to as the geometric fixedpoints functor for $H$. If $H \subseteq H' \subseteq G$ is a pair of cofinite normal subgroups of $G$, then there is a natural equivalence of finitely genuine $G/H'$-spectra
\[
	(X^{\Phi H})^{\Phi H'/H} \simeq X^{\Phi H'}
\]
for every quasifinitely genuine $G$-spectrum $X$. Indeed, the following composite
\[
	\QFin_G \xrightarrow{(-)^H} \Fin_{G/H} \xrightarrow{(-)^{H'/H}} \Fin_{G/H'}
\]
is isomorphic to the functor $(-)^{H'} : \QFin_G \to \Fin_{G/H'}$, inducing a natural equivalence of functors
\[
	\mathrm{infl}_{H'/H} \circ \mathrm{infl}_{H} \simeq \mathrm{infl}_{H'}
\]
which gives the desired equivalence $(-)^{\Phi H'/H} \circ (-)^{\Phi H} \simeq (-)^{\Phi H'}$ by passing to left adjoints.
\end{construction}

As the main result of this section, we obtain an explicit formula for the underlying spectrum of the geometric fixedpoints $X^{\Phi H}$ for every cofinite subgroup $H$ of $G$ (this makes sense since the Weyl group is finite). We will abusively also denote this spectrum as $X^{\Phi H}$. The formula is inspired by the usual formula for geometric fixedpoints in equivariant stable homotopy theory for finite groups (cf. Remark~\ref{remark:usual_formula_geometric_fixedpoints_finitegroup}) and by Kaledin's result (cf.~\cite[Proposition 3.5]{Kal14}).

\begin{notation}
We say that a morphism of quasifinite $G$-sets $f : S_0 \to S_1$ is \emph{proper}, if for each of the components $G/H\subseteq S_1$, we have that $f^{-1}(G/H)^H = \emptyset$. Equivalently, this means that for each $x\in S_0$, the isotropy group of $f(x)$ is strictly larger than that of $x$. 
For a quasifinite $G$-set $S$, let $(\QFin_G)^{\mathrm{prop}}_{/S}$ denote the proper slice over $S$, i.e. the category whose objects are quasifinite $G$-sets $S'$ equipped with a proper morphism $S' \to S$. If $S = G/G$, then we denote this category by $\QFin_G^{\mathrm{prop}}$.
\end{notation}

For every quasifinitely genuine $G$-spectrum $X$, there is a canonical functor
\[
	(\QFin_G)^{\mathrm{prop}}_{/(G/H)} \to  \QFin_G \to \mathrm{Span}(\QFin_G) \xrightarrow{X} \Sp
\]
which carries $S$ to $X^S$ and a morphism $S \to S'$ to the transfer map $X^S \to X^{S'}$ (cf.~\cite{Bar17}). The goal of the present section is to prove the following result:

\begin{theorem} \label{theorem:cofiber_sequence_geometricfixedpoints}
Let $G$ be a group and let $H$ be a cofinite subgroup of $G$. 
For every quasifinitely genuine $G$-spectrum $X$, there is a cofiber sequence of spectra
\[
	\underset{S \in (\QFin_G)^{\mathrm{prop}}_{/(G/H)}}{\colim} X^S \to X^H \to X^{\Phi H},
\]
where the first map is induced by the transfer $X^S \to X^H$ for $S \in (\QFin_G)^{\mathrm{prop}}_{/(G/H)}$. 
\end{theorem}

We will give the proof of this statement as the main result of the section, but let us first draw some conclusions. 

\begin{remark} \label{remark:usual_formula_geometric_fixedpoints_finitegroup}
Let $G$ be a finite group. If $X$ is a finitely genuine $G$-spectrum, then the geomtric fixedpoints spectrum $X^{\Phi G}$ is calculated by means of the following cofiber sequence
\[
	\underset{G/H \in \mathrm{Orb}_G^{\mathrm{prop}}}{\colim} X^H \to X^{G} \to X^{\Phi G},
\]
where the colimit is indexed over the full subcategory of $\mathrm{Orb}_G$ spanned by those orbits $G/H$ for which $H$ is a proper subgroup of $G$. The inclusion $\mathrm{Orb}_G \hookrightarrow \Fin_G$ exhibits the category of finite $G$-sets as the finite coproduct completion of $\mathrm{Orb}_G$. The restriction of a genuine $G$-spectrum $X$ on  $\Fin_G$ is the left Kan extension of its restriction to $\mathrm{Orb}_G$, thus 
the colimit appearing above is equivalent to the colimit indexed over $\Fin_G^{\mathrm{prop}}$. From this perspective, Theorem~\ref{theorem:cofiber_sequence_geometricfixedpoints} asserts that the geometric fixedpoints of a quasifinitely genuine $G$-spectrum can be calculated by means of the usual cofiber sequence simply with $\Fin_G$ replaced by $\QFin_G$. 
\end{remark}

\begin{remark} \label{remark:geometric_uniformly_boundedbelow}
If $X$ is a connective quasifinitely genuine $G$-spectrum, meaning that all the fixed points are connective, then the spectrum $X^{\Phi H}$ is connective for every cofinite subgroup $H$ of $G$ by virtue of Theorem~\ref{theorem:cofiber_sequence_geometricfixedpoints}. If $X$ is uniformly bounded below, then the collection $\{X^{\Phi H}\}$ is uniformly bounded below, where $H$ ranges through the cofinite subgroups of $G$.
\end{remark}

If $G$ is a finite group, then the geometric fixedpoints functors are jointly conservative as one proves by induction on the order of $G$ using the isotropy separation sequence (cf.~\cite[Theorem 7.2]{Sch20}). This technique is not available in the case of infinite groups $G$ (e.g. profinite) and the geometric fixedpoint functors are in fact not jointly conservative on the $\infty$-category of finitely genuine $G$-spectra. It is an observation due to Kaledin~\cite[Proposition 3.5]{Kal14} that the passage to the $\infty$-category of quasifinitely genuine $G$-spectra fixes this, meaning that the geometric fixedpoints functors are jointly conservative under suitable boundedness assumptions. We stress that the proof of Proposition~\ref{proposition:geometric_fixedpoints_conservative} below is entirely due to Kaledin, and we reproduce his argument in our setting for completeness.

\begin{proposition} \label{proposition:geometric_fixedpoints_conservative}
If $X$ is a quasifinitely genuine $G$-spectrum which is uniformly bounded below, then $X \simeq 0$ precisely if $X^{\Phi H} \simeq 0$ as spectra for every cofinite subgroup $H$ of $G$. 
\end{proposition}

\begin{proof}
Assume that $X^{\Phi H}\simeq 0$ for each cofinite subgroup, but $X\not\simeq 0$. By shifting, we may assume that $X$ is connective and that $\pi_0(X^{S_0})\neq 0$ for some $S_0$. The maps
\[
  \colim_{S\in (\QFin_G)_{/(G/H)}^{\prop}} \pi_0(X^S)\to \pi_0(\colim_{S\in (\QFin_G)_{/(G/H)}^{\prop}} X^S) \to \pi_0 (X^{G/H})
\]
are surjective, where the left hand colimit is taken in the category of sets. In fact, the first map is even an isomorphism, which follows from the fact that all $X^S$ are connective, $(\QFin_G)_{/(G/H)}^{\prop}$ is sifted by virtue of having finite coproducts, and $\pi_0$ of connective spectra commutes with sifted colimits. The second map is surjective by Theorem~\ref{theorem:cofiber_sequence_geometricfixedpoints}. Consequently, for any $x\in \pi_0(X^{G/H})$ we find a proper map $f: S'\to G/H$ and $y\in X^{S'}$ with $f_*(y)=x$. More generally, this implies that for any quasifinite $G$-set $S$ and any $x\in \pi_0(X^S)$, we find a proper map $f: S'\to S$ and $y\in \pi_0(X^S)$ with $f_*(y)=x$. For $x_0\neq 0 \in \pi_0(X^{S_0})$, we now inductively find proper maps
\[
   \ldots \to S_2 \to S_1\to S_0
\]
and $x_i\in \pi_0(X^{S_i})$ with $x_i\mapsto x_{i-1}$ under the transfer maps. Properness of the maps ensures that
\[
	S = \coprod_{i \geq 1} S_i
\]
is again a quasifinite $G$-set. Let $p : S \amalg S_0 \to S_0$ denote the canonical map assembling the maps $S_i\to S_0$, and let $i : S \hookrightarrow S \amalg S_0$ denote the canonical inclusion. Let $f : S \to S \amalg S_0$ denote the map obtained by forming the coproduct of the maps $f_i : S_i \to S_{i-1}$ for $i \geq 1$, and note that $p \circ f = p \circ i$. Finally, consider the element $x = (x_i)_{i \geq 0}$ of $\pi_0(X^{S \amalg S_0}) \cong \prod_{i \geq 0} \pi_0(X^{S_i})$ and note that
\[
  x_0 = p_\ast((x_0,0,0,\ldots)) =  p_{\ast} f_{\ast}((x_i)_{i \geq 1}) - p_{\ast} i_{\ast}((x_i)_{i \geq 1}) = (p \circ f)_{\ast}((x_i)_{i \geq 1}) - (p \circ i)_{\ast}((x_i)_{i \geq 1}) = 0,
\]
where we have used that $p \circ f = p \circ i$ as observed above. This contradicts the assumption that $x_0\neq 0$ and finishes the proof.
\end{proof}

We proceed to give the proof of Theorem~\ref{theorem:cofiber_sequence_geometricfixedpoints}. 
Let $\EG$ denote the object of the $\infty$-category $\Fun(\QFin_G^{\op}, \Spaces)$ determined by
\[
	(\EG)^S \simeq 
	\begin{cases}
		\ast & \text{if $S^G = \varnothing$} \\
		\varnothing & \text{if $S^G \neq \varnothing$}
	\end{cases}
\]
Similarly, one can define $\mathrm{E}\mathcal{F}$ for any family $\mathcal{F}$ of subgroups of $G$ which is closed under subconjugation. The notation above illustrates that we are considering the family of proper subgroups.

\begin{construction} \label{construction:EpropG}
Let $\widetilde{\EG} \in \Fun(\QFin_G^{\op}, \Spaces_{\ast})$ denote the cofiber of the map
\[
	\EG_+ \to S^0 = \underline{G/G}_+
\]
defined by applying the functor $(-)_+$ to the map $\EG \to \ast$.
\end{construction}

Note that there is no reason that the functor $\widetilde{\EG}$ should be very additive since its defined as a cofiber in the $\infty$-category $\Fun(\QFin_G^{\op}, \Spaces_\ast)$. We will repeatedly make use of the following:

\begin{lemma} \label{lemma:Eprop_tilde_smash_proper_vanish}
The functor $\EG$ is equivalent to the colimit of the following diagram
\[
	\QFin_G^{\mathrm{prop}} \hookrightarrow \QFin_G \to \Fun(\QFin_G^{\op}, \Spaces),
\]
where the latter functor denotes the Yoneda embedding $S \mapsto \underline{S}$. Furthermore, if $S$ is a quasifinite $G$-set with proper isotropy, then $\ETG \wedge S_+ \simeq \ast$.
\end{lemma}

\begin{proof}
The first assertion follows since every presheaf can be written as a colimit of representable presheaves. Indeed, there is a canonical equivalence of presheaves
\[
	\EG \simeq \underset{S \in \QFin_G, \underline{S} \to \EG}{\colim} \underline{S} \simeq \underset{S \in \QFin_G^{\mathrm{prop}}}{\colim} \underline{S},
\]
where the last equivalence follows since the indexing categories of the colimits are equivalent which follows from the definition of $\EG$. For the final assertion, note that by the cofiber sequence $\EG_+ \to S^0 \to \widetilde{\EG}$, we have to prove that for $S$ with only proper isotropy the morphism
\[
\EG \times \underline{S} \to \underline{S}
\]
is an equivalence in $\Fun(\QFin_G^{\op}, \Spaces)$. This can be checked pointwise, that is after evaluation at some $T \in \QFin_G$. Indeed, if $T$ has only proper isotropy we get $\EG^T = \ast$ and otherwise both sides are empty. 
\end{proof}

Now, we consider the functor of $\infty$-categories
\[
	- \otimes \Sigma^\infty\ETG : \Fun(\mathrm{Span}(\QFin_G), \Sp) \to \Fun(\mathrm{Span}(\QFin_G), \Sp)
\]
obtained by forming the tensor product with the image of $\ETG$ in $\Fun(\mathrm{Span}(\QFin_G),\Sp)$. This tensor product is the Day convolution product on  $\Fun(\mathrm{Span}(\QFin_G), \Sp)$ and not the one of any of the subcategories of additive or very additive functors!
This functor is equivalently given by the construction $X \mapsto \mathrm{cofib}(X \otimes \Sigma^\infty_+ \EG \to X)$.

\begin{lemma} \label{lemma:Eprop_tilde_localization}
The functor of $\infty$-categories
\[
	- \otimes \Sigma^\infty \ETG : \Fun(\mathrm{Span}(\QFin_G), \Sp) \to \Fun(\mathrm{Span}(\QFin_G), \Sp)
\]
is the localization onto the full subcategory of $\Fun(\mathrm{Span}(\QFin_G), \Sp)$ spanned by those functors $X$ which satisfy that $X^S \simeq 0$ for every quasifinite $G$-set $S$ with proper isotropy.
\end{lemma}

\begin{proof}
We first prove that $X \otimes \Sigma^\infty\ETG$ lies in the desired full subcategory for every object $X$ of the $\infty$-category $\Fun(\mathrm{Span}(\QFin_G), \Sp)$. If $S$ is a quasifinite $G$-set with proper isotropy, then
\begin{align*}
  (X \otimes \Sigma^\infty\ETG)^S &\simeq \map(\Sigma^\infty_+ \underline{S}, X \otimes \Sigma^\infty\ETG) \\
  &\simeq \map(\mathbbm{1}, X \otimes \Sigma^\infty\ETG \otimes \Sigma^\infty_+ \underline{S}) \\
  &\simeq \map(\mathbbm{1}, X \otimes \Sigma^\infty(\ETG \wedge \underline{S}_+)) \simeq 0,
\end{align*}
  using that $\Sigma^\infty_+ \underline{S}$ is canonically self-dual (cf.~\textsection\ref{subsection:monoidal}) and Lemma~\ref{lemma:Eprop_tilde_smash_proper_vanish}. There is a natural transformation of functors $\id \to (- \otimes \Sigma^\infty\widetilde{\EG})$ induced by the defining map $S^0 \to \widetilde{\EG}$. If $X$ lies in our desired full subcategory, then we claim that the induced map
\[
	X \to X \otimes\Sigma^\infty \widetilde{\EG}
\]
is an equivalence which by construction is equivalent to showing that $X \otimes \Sigma^\infty_+ \EG \simeq 0$. To this end, we may write $\EG$ as a colimit by Lemma~\ref{lemma:Eprop_tilde_smash_proper_vanish}, which in turn yields that
\[
  X \otimes \Sigma^\infty_+ \EG \simeq \underset{S \in \QFin_G^{\mathrm{prop}}}{\colim} X \otimes \Sigma^\infty_+ \underline{S},
\]
  so it suffices to prove that $X \otimes \Sigma^\infty_+ \underline{S} \simeq 0$ for every $S \in \QFin_G^{\mathrm{prop}}$. Again, this follows since
\begin{align*}
  (X \otimes \Sigma^\infty_+ \underline{S})^T &\simeq \map(\Sigma^\infty_+ T, X \otimes \Sigma^\infty_+ \underline{S}) \\
  &\simeq \map(\Sigma^\infty_+ T \otimes \Sigma^\infty_+ \underline{S}, X) \\
	&\simeq X^{T \times S} \simeq 0
\end{align*}
where we have used that $T \times S$ has proper isotropy and our assumption that $X$ lies in our desired full subcategory. Finally in order to conclude, we need to argue that both natural transformations $\Sigma^\infty \widetilde{\EG} \to \Sigma^\infty \widetilde{\EG} \otimes \Sigma^\infty \widetilde{\EG}$ are equivalences. This follows since they are both (up to a symmetry) the natural transformation of the previous discussion for $X = \Sigma^\infty \widetilde{\EG}$. This proves the desired statement by virtue of~\cite[Proposition 5.2.7.4]{Lur09}. 
\end{proof}

We have the following concrete formula for the localization functor $- \otimes \Sigma^\infty\widetilde{\EG}$:

\begin{lemma} \label{lemma:formula_tensor_Eprop_tilde}
Let $X$ be an object of $\Fun(\mathrm{Span}(\QFin_G), \Sp)$. There is an equivalence of spectra
\[
  (X \otimes \Sigma^\infty \widetilde{\EG})^T \simeq \mathrm{cofib} \Big(
	\underset{S \in \QFin_G^{\mathrm{prop}}}{\colim} X^{S \times T} \to X^T
	\Big)
\]
for every quasifinite $G$-set $T$. 
\end{lemma}

\begin{proof}
By writing $\EG$ as a colimit and using that $(-)^T$ preserves colimits, we find that
\begin{align*}
	(X \otimes \Sigma^\infty\widetilde{\EG})^T &\simeq \mathrm{cofib} \Big(
  \underset{S \in \QFin_G^{\mathrm{prop}}}{\colim} (X \otimes \Sigma^\infty_+ \underline{S})^T \to X^T
	\Big) \\
	&\simeq \mathrm{cofib} \Big(
	\underset{S \in \QFin_G^{\mathrm{prop}}}{\colim} X^{S \times T} \to X^T
	\Big),
\end{align*}
  where we once more have used that $(X \otimes \Sigma^\infty_+ \underline{S})^T \simeq X^{S \times T}$ as in the proof of Lemma~\ref{lemma:Eprop_tilde_localization}. 
\end{proof}

\begin{corollary} \label{corollary:Etilde_additive_is_vadd}
If $X : \mathrm{Span}(\QFin_G) \to \Sp$ is an additive functor, then the functor $X \otimes \Sigma^\infty \widetilde{\EG}$ is very additive. 
\end{corollary}

\begin{proof}
Note that $X \otimes \Sigma^\infty\widetilde{\EG}$ is additive by Lemma~\ref{lemma:formula_tensor_Eprop_tilde} by the assumption that $X$ is additive and that $T \times -$ preserves coproducts of quasifinite $G$-sets. Moreover, Lemma~\ref{lemma:Eprop_tilde_localization} ensures that $X \otimes \Sigma^\infty\widetilde{\EG}$ vanishes on quasifinite $G$-sets with proper isotropy. For an arbitrary quasifinite $G$-set $S = \coprod_i S_i$ all but finitely many $S_i$ satisfy that $S_i^G = \varnothing$, so it follows that $X \otimes \widetilde{\EG}$ is very additive since it is already additive.
\end{proof}

In particular, this shows that $\Sigma^\infty \ETG = (\Sigma^\infty\ETG)_\add = (\Sigma^\infty \ETG)_\vadd$, and that the localisation on $\Sp^G_{\qfgen}$ onto its full subcategory of very additive functors with vanishing fixed points on proper subgroups, which a priori is given by $X\mapsto (X\otimes\Sigma^\infty\ETG)_{\vadd}$, is still just given by $-\otimes\Sigma^\infty\ETG$.

\begin{lemma} \label{lemma:infl_equiv_on_fullsubcat}
The inflation functor 
\[
	\mathrm{infl}_G : \Sp\to \Sp^{G}_{\qfgen}
\]
is fully faithful with essential image those quasifinitely genuine $G$-spectra satisfying that $X^S \simeq 0$ for all quasifinite $G$-sets $S$ with proper stabilizers. The inverse is given by taking $G$-fixed points. 
\end{lemma}

\begin{proof}
First observe that the inflation of any spectrum has vanishing fixedpoints for all proper orbits, so the functor $\mathrm{infl}_G$ lands in the full subcategory $\cat{C} \subseteq \Sp^G_{\qfgen}$ spanned by those quasifinitely genuine $G$-spectra whose fixedpoints vanish for proper subgroups. This is equivalent to the full subcategory of those quasifinitely genuine $G$-spectra spanned by those $X$ with $X^S \simeq 0$ for every quasifinite $G$-set $S$ with proper isotropy. We wish to prove that the induced functor
\[
	\mathrm{infl}_G : \Sp \to \cat{C}
\]
is equivalence of $\infty$-categories.
As a first step, we claim that this functor preserves colimits. This follows from that the fact that colimits in $\cat{C}$ are formed pointwise, which in turn follows from the fact that the pointwise colimit of objects in $\cat{C}$ vanishes on quasifinite $G$-sets with proper isotropy. Since it is also additive (as a colimit of additive functors in a stable $\infty$-category) it is in fact already very additive as in the proof of Corollary~\ref{corollary:Etilde_additive_is_vadd}.

Next, consider the quasifinitely genuine $G$-spectrum given by the construction $\mathrm{infl}_G(\S)$. Note that $\mathrm{infl}_G(\S)^G \simeq \S$, so we in particular obtain a canonical map of quasifinitely genuine $G$-spectra
\[
	\mathbb{S} = \Sigma^\infty_+ \underline{G/G} \to \mathrm{infl}_G(\mathbb{S})
\]
which by adjunction provides a map
\[
\Sigma^\infty \ETG \to  \mathrm{infl}_G(\mathbb{S})
\]
which we claim is an equivalence. As both the source and the target have vanishing fixedpoints for all proper subgroups, it suffices to prove the claim after taking $G$-fixedpoints which in turn amounts to proving that the map
\[
	\mathrm{cofib} ((\Sigma^\infty_+\EG)^G \to \mathbb{S}^G) \to \mathbb{S}
\]
is an equivalence. In order to calculate the source we first claim that the cofiber of $\Sigma^\infty_+\EG \to \mathbb{S}$ can be calculated in the $\infty$-category $\Fun^{\add}(\mathrm{Span}(\QFin_G), \Sp)$ rather than in $\Fun^{\vadd}(\mathrm{Span}(\QFin_G), \Sp)$, where we replace the suspension spectra with their additive versions $\widetilde{\Sigma}^\infty_+ \EG$ and $\widetilde{\Sigma}^\infty_+ \underline{G/G}$, where $\widetilde{\Sigma}^{\infty}_+$ denotes the functor introduced in the discussion preceeding Lemma~\ref{lemma:tomDieck_splitting}. This follows since the cofiber formed in additive functors is automatically very additive since it has vanishing fixedpoints on proper subgroups as in Corollary~\ref{corollary:Etilde_additive_is_vadd}. In particular  the $G$-fixed points of $\Sigma^\infty \ETG$ are given by the cofiber of the map of spectra
\[
  (\widetilde{\Sigma}^\infty_+ \EG)^G \to (\widetilde{\Sigma}^\infty_+ \underline{G/G})^G .
\]
To calculate this cofiber, we utilize the version of the tom Dieck splitting established in Lemma~\ref{lemma:tomDieck_splitting}. According to this, we have that
\[
	(\widetilde{\Sigma}^\infty_+ \EG)^G \simeq \left((\QFin_G)_{/\EG}\right)^{\simeq, \mathrm{grp}} \simeq (\QFin^{\mathrm{prop}}_G)^{\simeq, \mathrm{grp}}
\]
while $(\widetilde{\Sigma}^\infty_+ \underline{G/G})^G \simeq (\QFin_G)^{\simeq, \mathrm{grp}}$. Next, note that there is an equivalence of $\E_\infty$-monoids
\[
	\QFin_G^{\simeq} \simeq (\QFin_G^{\mathrm{prop}})^{\simeq} \times \Fin^{\simeq},
\]
so we need to calculate the cofiber of the map
\[
	(\QFin^{\mathrm{prop}}_G)^{\simeq, \mathrm{grp}} \to (\QFin_G^{\mathrm{prop}})^{\simeq, \mathrm{grp}} \times \Fin^{\simeq, \mathrm{grp}},
\]
induced by the inclusion $\QFin^{\mathrm{prop}}_G \hookrightarrow \QFin_G$ and this cofiber is given by $\Fin^{\simeq, \mathrm{grp}} \simeq \S$ by virtue of the Barratt--Priddy--Quillen theorem. By naturality of the tom Dieck splitting, this proves that the desired map $\mathrm{cofib} ((\Sigma^\infty_+\EG)^G \to \mathbb{S}^G) \to \mathbb{S}$ is an equivalence, which proves the claim.

Continuing, since the inflation functor $\mathrm{infl}_G : \Sp\to \cat{C}$ preserves colimits, its adjoint is given by mapping out of $\mathrm{infl}_G(\S) \simeq \Sigma^\infty \ETG$. Since $-\otimes\Sigma^\infty \ETG$ is a localisation onto $\cat{C}$, we have that
\[
   (-)^G \simeq \map(\S, -) \simeq \map(\Sigma^\infty\ETG, -)
\]
on $\cat{C}$. Therefore the right adjoint to $\mathrm{infl}_G: \Sp\to \cat{C}$ is just given by $G$-fixedpoints.
This functor is conservative on $\cat{C}$, so it suffices to show that for any spectrum $X$, the counit of the adjunction
\[
X \to (\mathrm{infl}_GX)^G
\]
is an equivalence, which is obvious by definition of inflation.
\end{proof}

Finally, let us summarize what we have proved. It follows from Lemma~\ref{lemma:infl_equiv_on_fullsubcat} that if $X$ is a quasifinitely genuine $G$-spectrum, then there is an equivalence of spectra
\[
	X^{\Phi G} \simeq (X \otimes \Sigma^\infty \ETG)^G
\]
which provides the desired cofiber sequence in Theorem~\ref{theorem:cofiber_sequence_geometricfixedpoints} by means of Lemma~\ref{lemma:formula_tensor_Eprop_tilde}. 

In general, if $H$ is a cofinite subgroup of $G$, then $(-)^{\Phi H} : \Sp^{G}_{\qfgen} \to \Sp$ is given by the composite
\[
	\Sp^{G}_{\qfgen} \xrightarrow{\mathrm{Res}^G_H} \Sp^{H}_{\qfgen} \xrightarrow{(-)^{\Phi H}} \Sp,
\]
where $\mathrm{Res}^G_H$ is induced by the functor $\QFin_H \to \QFin_G$ determined by the assigment $S \mapsto G \times_H S$. This can be seen by looking at the adjoints. By virtue of our discussion above, we obtain natural equivalences of spectra
\[
	X^{\Phi H} \simeq \cofib \Big( \underset{T \in \QFin_H^{\mathrm{prop}}}{\colim} (\mathrm{Res}^G_H X)^T \to (\mathrm{Res}^G_H X)^H \Big)
\]
for every quasifinitely genuine $G$-spectrum $X$. We note that the functor $\QFin_H \to (\QFin_G)_{/(G/H)}$ given by the construction $T \mapsto (G \times_H T \to G \times_H \ast = G/H)$ is an equivalence of categories, which further restricts to an equivalence $\QFin^{\mathrm{prop}}_H \to (\QFin_G)^{\mathrm{prop}}_{/(G/H)}$. Consequently, we find that
\[
	X^{\Phi H} \simeq \cofib \Big( \underset{S \in (\QFin_G)^{\mathrm{prop}}_{/(G/H)}}{\colim} X^S \to X^H \Big),
\]
which establishes the desired formula for $(-)^{\Phi H}$ in Theorem~\ref{theorem:cofiber_sequence_geometricfixedpoints}, where we additionally have used that $(\mathrm{Res}^G_H X)^T = X^{G \times_H T}$ for every quasifinite $H$-set $T$.

\begin{remark}
Let $X$ denote a quasifinitely genuine $G$-spectrum and let $H$ be a cofinite subgroup of $G$. By Construction~\ref{construction:geometric_fixedpoints_functor}, the geometric fixedpoints $X^{\Phi H}$ carries an action of the Weyl group $W_G(H)$. We remark that the equivalence of spectra
\[
	X^{\Phi H} \simeq \cofib \Big( \underset{S \in (\QFin_G)^{\mathrm{prop}}_{/(G/H)}}{\colim} X^S \to X^H \Big)
\]
obtained above can be promoted to a $W_G(H)$-equivariant equivalence. However, for what follows it will suffice to know that the geometric fixedpoints spectrum $X^{\Phi H}$ carries an action by $W_G(H)$ by Construction~\ref{construction:geometric_fixedpoints_functor} and utilize the formulas established in this section to verify that certain maps which are equivariant by functoriality of
\[
	\Sp^{G}_{\qfgen} \xrightarrow{(-)^{\Phi H}} \Sp^{W_G(H)}_{\fgen} \to \Sp^{\B W_G(H)}
\]
are indeed equivalences by checking that the underlying maps of spectra are equivalences.
\end{remark}

\subsection{Quasifinitely genuine $\Z$-spectra} \label{subsection:qfgen_Z}
In~\textsection\ref{subsection:qfgen_G}, we have introduced the $\infty$-category of quasifinitely genuine $G$-spectra for an arbitrary group $G$ following ideas of Kaledin. In this section, we discuss some further features of this construction when $G = \Z$, which will be instrumental in~\textsection\ref{subsection:qfgenZ_TR} when we discuss the relation to TR. Recall that a quasifinitely genuine $\Z$-spectrum is a functor
\[
	X : \mathrm{Span}(\QFin_{\Z}) \to \Sp
\]
which satisfies that $i^{\star} X$ is a finitely genuine $\Z$-spectrum and that the canonical map
\[
	X\big( \coprod_{i \geq 1} \Z/n_i\Z \big) \to \prod_{i \geq 1} X^{n_i \Z}
\]
is an equivalence for every sequence $\{n_i\}_{i \geq 1}$ of positive integers satisfying that $n_i \to \infty$ for $i \to \infty$. We briefly discuss the structure encoded by a quasifinitely genuine $\Z$-spectrum. In the first place, a quasifinitely genuine $\Z$-spectrum $X$ is equipped with compatible inclusion and transfer maps. Indeed, for every $n \geq 1$, the terminal map $\Z/n\Z \to \Z/\Z$ in $\QFin_{\Z}$ induces a pair of maps
\begin{align*}
	X^{\Z} \to X^{n\Z} \\
	X^{n\Z} \to X^{\Z}
\end{align*}
referred to as the inclusion and the transfer map, respectively. The inclusion map is equivariant for the trivial $C_n$-action on the source and the residual $C_n$-action on the target while the transfer is equivariant for the residual $C_n$-action on the source and the trivial $C_n$-action on the target. In particular, these map canonically refine to maps $X^{\Z} \to (X^{n\Z})^{\h C_n}$ and $(X^{n\Z})_{\h C_n} \to X^{\Z}$. The inclusion maps and the transfer maps are compatible in a coherent fashion. These structures are already encoded by the underlying finitely genuine $\Z$-spectrum of $X$, and the additional structure encoded by the quasifinitely genuine structure is completeness. Indeed, for every sequence $\{n_i\}_{i \geq 1}$ of positive integers with $n_i \to \infty$ for $i \to \infty$ as above, the infinite coproduct $\coprod_{i \geq 1} \Z/n_i\Z$ is an example of a quasifinite $\Z$-set and the terminal map induces a map
\[
	V_{\{n_i\}} : \prod_{i \geq 1} X^{n_i\Z} \to X^{\Z}
\]
such that $V_{\{n_i\}} \circ \mathrm{incl}_k : X^{n_k \Z} \to X^{\Z}$ is canonically equivalent to the transfer map $V_{n_k} : X^{n_k\Z} \to X^{\Z}$ for each $k \geq 1$. This can be summarized by saying that the map $V_{\{n_i\}}$ is given by the formula
\[
	V_{\{n_i\}} = \sum_{i = 1}^{\infty} V_{n_i}.
\]
In other words, we regard the quasifinitely genuine $\Z$-structure as specifying that certain infinite sums of transfer maps converge in $X$. This additionally means that our notion of completeness is encoded as a structure rather than a property. This is the main insight of Kaledin in~\cite{Kal14}. 

Before we proceed, we recall that there is an adjunction of $\infty$-categories
\[
\begin{tikzcd}[column sep = large]
	\Sp^{\Z}_{\qfgen} \arrow[yshift=0.7ex]{r}{(-)^{\Phi n\Z}} & \Sp^{\Z/n\Z}_{\fgen} \arrow[yshift=-0.7ex]{l}{\mathrm{infl}_{n\Z}} 
\end{tikzcd}
\]
where the left adjoint is the geometric fixedpoints functor $(-)^{\Phi n\Z}$ (cf. Construction~\ref{construction:geometric_fixedpoints_functor}). The geometric fixedpoints functor can be calculated by means of a cofiber sequence which we recall. As in~\textsection\ref{subsection:qfgen_G}, we will employ the following terminology which we recall for convenience.

\begin{notation}
Let $(\QFin_{\Z})_{/(\Z/n\Z)}^{\mathrm{prop}}$ denote the category whose objects are quasifinite $\Z$-sets $S$ with a map $S \to \Z/n\Z$ such that the isotropy groups are proper subgroups of $n\Z$. For $n = 1$, we simply denote this category by $\QFin_{\Z}^{\mathrm{prop}}$. 
\end{notation}

Unwinding the definition, we see that every object $S$ of $(\QFin_{\Z})_{/(\Z/n\Z)}^{\mathrm{prop}}$ is of the form
\[
	S = \coprod_{i \geq 1} \Z/m_i\Z
\]
for some sequence $\{m_i\}_{i \geq 1}$ which satisfies that $n_i \to \infty$ for $i \to \infty$ and that $n$ is proper divisor of $m_i$ for $i \geq 1$. If $X$ is a quasifinitely genuine $\Z$-spectrum, then $X^{\Phi n\Z}$ fits into a cofiber sequence
\[
	\underset{S \in (\QFin_{\Z})_{/(\Z/n\Z)}^{\mathrm{prop}}}{\colim} X^S \to X^{n\Z} \to X^{\Phi n\Z}
\]
induced by the transfer maps $X^S \to X^{n\Z}$ (cf. Theorem~\ref{theorem:cofiber_sequence_geometricfixedpoints}). Recall that the functors $(-)^{\Phi n\Z}$ are jointly conservative on the full subcategory of $\Sp^{\Z}_{\qfgen}$ spanned by those which are uniformly bounded below (cf. Proposition~\ref{proposition:geometric_fixedpoints_conservative}). The analogue of this result is not true for the $\infty$-category of finitely genuine $\widehat{\Z}$-spectra, which is another advantage of working in $\Sp^{\Z}_{\qfgen}$. For the remainder of this section, we discuss some further structural features of the geometric fixedpoints functors. We can reduce questions about the geometric fixedpoints for the cofinite subgroup $n\Z$ to questions about the geometric fixedpoints for the full group $\Z$ by means of the following construction:

\begin{construction} \label{construction:shift_qfgen}
For $n \geq 1$, let $i_n : \QFin_{\Z} \to \QFin_{\Z}$ denote the functor defined by $i_n(S) = nS$, where $nS$ denotes the quasifinite $\Z$-set with the same underlying set as $S$ but whose $\Z$-action is given by restriction the $\Z$-action on $S$ along the endomorphism $n : \Z \to \Z$. As an example, if $S = \coprod_{i \geq 1} \Z/m_i\Z$ denotes an object of $\QFin_{\Z}$, then $i_n(S) \simeq \coprod_{i \geq 1} \Z/nm_i\Z$. 
\end{construction}

Let $i_n^{\star} : \Sp^{\Z}_{\qfgen} \to \Sp^{\Z}_{\qfgen}$ denote the functor obtained by restriction along $i_n$. Note that if $X$ is a quasifinitely genuine $\Z$-spectrum, then $i_n^{\star} X$ is the quasifinitely genuine $\Z$-spectrum with
\[
	(i_n^{\star} X)^{k\Z} \simeq X^{kn\Z}
\]
for every $k \geq 1$. We observe that this equivalence prolongs to geometric fixedpoints:

\begin{lemma} \label{lemma:ngeometricfxp_to_1geometricfxp}
Let $n \geq 1$. The canonical map of spectra
\[
	(i_n^{\star} X)^{\Phi\Z} \to X^{\Phi n\Z}
\]
is an equivalence for every quasifinitely genuine $\Z$-spectrum $X$.
\end{lemma}

\begin{proof}
This follows from the observation that the construction $S \mapsto nS$ determines a functor
\[
	\QFin_{\Z}^{\mathrm{prop}} \to (\QFin_{\Z})_{/(\Z/n\Z)}^{\mathrm{prop}}
\]
which is an equivalence since it admits an inverse determined by the construction $S' \mapsto \frac{1}{n}S'$. 
\end{proof}

Using Lemma~\ref{lemma:ngeometricfxp_to_1geometricfxp}, we will be able to reduce questions about the geometric fixedpoints spectrum for the cofinite group $n\Z \subseteq \Z$ into questions about the geometric fixedpoints for the full group $\Z$. We will use this observation repeatedly in~\textsection\ref{subsection:qfgenZ_TR}. In fact, we think of the functor $i_n^{\star}$ as an analog of the shift functor on the $\infty$-category of polygonic spectra. Consequently, to analyze geometric fixedpoints, we are reduced to analyzing the functor determined by the construction
\[
	X \mapsto \underset{S \in \QFin_{\Z}^{\mathrm{prop}}}{\colim} X^S.
\]
To that end, we establish the following result:

\begin{proposition} \label{proposition:formula_fiber_term_geometric_fxp}
For every quasifinitely genuine $\Z$-spectrum $X$, there is a natural equivalence
\[
	\underset{S \in \QFin_{\Z}^{\mathrm{prop}}}{\colim} X^S \simeq \Big\vert
	\begin{tikzcd}[column sep = scriptsize]
		\cdots \arrow{r} \arrow[yshift=0.8ex]{r} \arrow[yshift=-0.8ex]{r} & \displaystyle\prod_{p, q \in \mathbb{P}} (X^{\mathrm{gcd}(p, q)\Z})_{\h C_{\mathrm{gcd}(p, q)}} \arrow[yshift=0.4ex]{r} \arrow[yshift=-0.4ex]{r} & \displaystyle\prod_{p \in \mathbb{P}} (X^{p\Z})_{\h C_p}
	\end{tikzcd}
	\Big\vert,
\]
where the horizontal maps are induced by the transfer maps. In particular, the functor
\[
	X \mapsto \underset{S \in \QFin_{\Z}^{\mathrm{prop}}}{\colim} X^S
\]
preserves inverse limits of uniformly bounded below filtered quasifinitely genuine $\Z$-spectra.
\end{proposition}

The proof of Proposition~\ref{proposition:formula_fiber_term_geometric_fxp} is a consequence of a cofinality statement. To formulate this, first note that if $\cat{I}$ is a category with binary products and a weakly terminal object $Y$, then there is a simplicial object $\Delta^{\op} \to \cat{I}$ given by the construction $[n] \mapsto Y^{\times n+1}$. Formally, this is defined as the left Kan extension of the functor $\{[0]\} \to \cat{I}$ given by $[0] \mapsto Y$ along the inclusion $\{[0]\} \hookrightarrow \Delta^{\op}$, which exists by virtue of the assumption that $\cat{I}$ admits binary products.

\begin{lemma} \label{lemma:cofinality_cech_nerve}
Let $\cat{I}$ be a category with a weakly terminal object $Y$ and binary products as above. Then the simplicial object $\Delta^{\op} \to \cat{I}$ determined by the construction $[n] \mapsto Y^{\times n+1}$ is cofinal. In particular, the colimit over $\cat{I}$ preserves uniformly bounded below limits of spectra.
\end{lemma}

\begin{proof}
See~\cite[Proposition 6.28]{MNN17} for the first part. The second part then follows from a Postnikov tower argument since for levelwise connective $X_\bullet$ we have have an equivalence of spectra
\[
	\tau_{\leq n} (\colim_{\Delta^{\op}} X_{\bullet}) \simeq \tau_{\leq n} (\colim_{\Delta_{\leq n+1}^{\op}} X_{\bullet}),
\]
for every $n$, and the latter is a finite geometric realization (cf.~\cite[Lemma 1.2.4.17]{Lur17}). Thus it commutes with limits of spectra where all the terms are connective.
\end{proof}

\begin{proof}[Proof of Proposition~\ref{proposition:formula_fiber_term_geometric_fxp}]
We claim that $\QFin_{\Z}^{\mathrm{prop}}$ admits a weakly terminal object given by
\[
	S_{\mathrm{wt}} = \coprod_{p \in \mathbb{P}} \Z/p\Z
\]
where the product is indexed over the set of prime numbers $\mathbb{P}$. If $S = \coprod_{i \geq 1} \Z/n_i\Z$ is any object of $\QFin_{\Z}^{\mathrm{prop}}$, meaning that $n_i \geq 2$ and $n_i \to \infty$ for $i \to \infty$, then each $n_i$ is divisible by some prime, which in turn defines the desired map $S \to S_{\mathrm{wt}}$. Applying Lemma~\ref{lemma:cofinality_cech_nerve}, we conclude that
\begin{align*}
	\underset{S \in \QFin_{\Z}^{\mathrm{prop}}}{\colim} X^S &\simeq \Big\vert
	\begin{tikzcd}[column sep = scriptsize, ampersand replacement = \&]
		\cdots \arrow{r} \arrow[yshift=0.8ex]{r} \arrow[yshift=-0.8ex]{r} \& X(S_{\mathrm{wt}} \times S_{\mathrm{wt}}) \arrow[yshift=0.4ex]{r} \arrow[yshift=-0.4ex]{r} \& X(S_{\mathrm{wt}})
	\end{tikzcd}
	\Big\vert \\
	&\simeq \Big\vert
	\begin{tikzcd}[column sep = scriptsize, ampersand replacement = \&]
		\cdots \arrow{r} \arrow[yshift=0.8ex]{r} \arrow[yshift=-0.8ex]{r} \& X\big(\coprod_{p,q \in \mathbb{P}} \Z/\mathrm{gcd}(p,q) \Z\big) \arrow[yshift=0.4ex]{r} \arrow[yshift=-0.4ex]{r} \& X\big(\coprod_{p\in \mathbb{P}} \Z/p \Z\big)
	\end{tikzcd}
	\Big\vert \\
	&\simeq \Big\vert
	\begin{tikzcd}[column sep = scriptsize, ampersand replacement = \&]
		\cdots \arrow{r} \arrow[yshift=0.8ex]{r} \arrow[yshift=-0.8ex]{r} \& \prod_{p,q \in \mathbb{P}} X^{\mathrm{gcd}(p, q)\Z} \arrow[yshift=0.4ex]{r} \arrow[yshift=-0.4ex]{r} \& \prod_{p \in \mathbb{P}} X^{p \Z}
	\end{tikzcd}
	\Big\vert,
\end{align*}
where the maps are induced by the transfer maps and the final equivalence follows by virtue of our assumption that $X$ is a quasifinitely genuine $\Z$-spectrum. Finally, recall that the transfers are suitably equivariant as discussed in the beginning of this section, so we may write the geometric realization as claimed. The final assertion follows from the last assertion of Lemma~\ref{lemma:cofinality_cech_nerve}.
\end{proof}

As a consequence, we obtain the following result which will play an instrumental role in~\textsection\ref{subsection:qfgenZ_TR}. 

\begin{corollary} \label{corollary:geometric_preserves_limits}
For every $n \geq 1$, the geometric fixedpoints functor 
\[
	(-)^{\Phi n\Z} : \Sp^{\Z}_{\qfgen} \to \Sp 
\]
preserves inverse limits of uniformly bounded below filtered quasifinitely genuine $\Z$-spectra.
\end{corollary}

\begin{proof}
Assume that $n = 1$. The fixedpoint functors preserve arbitrary limits and cofibers formed in spectra are exact, so it suffices to prove that the construction
\[
	X \mapsto \underset{S \in \Fin_{\Z}^{\mathrm{prop}}}{\colim} X^S
\]
preserves inverse limits of uniformly bounded below filtered quasifinitely genuine $\Z$-spectra, which is the content of Proposition~\ref{proposition:formula_fiber_term_geometric_fxp}. This implies the general case by Lemma~\ref{lemma:ngeometricfxp_to_1geometricfxp} since the functor $i_n^{\star}$ preserves limits.
\end{proof}

\section{Quasifinitely genuine $\Z$-spectra and TR} \label{subsection:qfgenZ_TR}
In this section, we make use of the formalism developed in~\textsection\ref{section:polygonicTR_completeCartier}, to prove that TR of a polygonic spectrum naturally refines to a quasifinitely genuine $\Z$-spectrum answering a question of Kaledin raised in the introduction of~\cite{Kal14}. This exhibits compatible Verschiebung maps and Frobenius maps on TR and additionally encodes that TR is complete with respect to the filtration induced by all the Verschiebung maps as discussed in~\textsection\ref{section:introductionpolygonic}. As a main result, we prove that TR determines an equivalence from the $\infty$-category of uniformly bounded below polygonic spectra and the $\infty$-category of uniformly bounded below quasifinitely genuine $\Z$-spectra (cf. Theorem~\ref{theorem:TR_equiv_qfgen_pgcsp}). This is the analogue of~\cite[Theorem 3.21]{AN20} in the integral situation.

\subsection{TR as a quasifinitely genuine $\Z$-spectrum}
We begin by exhibiting TR as a quasifinitely genuine $\Z$-spectrum. The idea is that a quasifinitely genuine $\Z$-spectrum is determined by its collection of geometric fixedpoints $X^{\Phi n\Z}$ for all the cofinite subgroups of $\Z$, and this collection of geometric fixedpoints naturally carries the structure of a polygonic spectrum:

\begin{construction}
For every quasifinitely genuine $\Z$-spectrum $X$, define $\mathrm{L}X \in \PgcSp$ by
\[
	(\mathrm{L}X)_n = X^{\Phi n\Z}
\]
for $n \geq 1$, regarded as a spectrum with $C_n$-action. The polygonic structure maps are defined by
\[
	X^{\Phi n\Z} \simeq (X^{\Phi pn\Z})^{\Phi C_p} \to (X^{\Phi pn\Z})^{\tate C_p}
\]
for every prime $p$ (Constr.~\ref{construction:geometric_fixedpoints_functor}). The construction $X \mapsto \mathrm{L}X$ defines a functor $\mathrm{L} : \Sp^{\Z}_{\qfgen} \to \PgcSp$.
\end{construction}

We regard $\mathrm{L} : \Sp^{\Z}_{\qfgen} \to \PgcSp$ as an analogue of the mod $V$ functor in the $p$-typical situation. We remark that this functor admits a much more transparent construction than the analogue of the mod $V$ functor in~\cite{McC21}. We prove the following result:

\begin{proposition} \label{proposition:TR_refines_qfgenZ}
The functor $\mathrm{L} : \Sp^{\Z}_{\qfgen} \to \PgcSp$ admits a right adjoint
\[
	\underline{\TR} : \PgcSp \to \Sp^{\Z}_{\qfgen}
\]
such that the following assertions are satisfied:
\begin{enumerate}[leftmargin=2em, topsep=5pt, itemsep=5pt]
	\item If $X \in \PgcSp$, then there is an equivalence of spectra $\underline{\TR}(X)^{n\Z} \simeq \TR(\mathrm{sh}_n X)$ for every $n \geq 1$.
	\item The functor $\mathrm{L} : \Sp^{\Z}_{\qfgen} \to \PgcSp$ is right $t$-exact. 
	\item The functor $\underline{\TR} : \PgcSp \to \Sp^{\Z}_{\qfgen}$ is $t$-exact.
\end{enumerate}
\end{proposition}

\begin{proof}
The functor $\mathrm{L}$ preserves colimits since the geometric fixedpoints functor $(-)^{\Phi n\Z}$ preserves colimits for every $n \geq 1$. It follows that $\mathrm{L}$ admits a right adjoint $\underline{\TR} : \PgcSp \to \Sp^{\Z}_{\qfgen}$ since both $\PgcSp$ and $\Sp^{\Z}_{\qfgen}$ are presentable. We identify the fixedpoints for the cofinite subgroup $n\Z \subseteq \Z$ on $\underline{\TR}(X)$ for every polygonic spectrum $X$ and $n \geq 1$. Using Lemma~\ref{lemma:fixedpoints_corepresented}, we have that
\[
	\underline{\TR}(X)^{n\Z} \simeq \map_{\Sp^{\Z}_{\qfgen}}(\Sigma^\infty_+(\Z/n\Z), \underline{\TR}(X)) \simeq \map_{\PgcSp}(\mathrm{L}\Sigma^\infty_+(\Z/n\Z), X)
\]
since $\underline{\TR}$ is a right adjoint of $\mathrm{L}$. Unwinding the definition, we see that
\[
	(\mathrm{L}\Sigma^\infty_+(\Z/n\Z))_k = \Sigma^\infty_+(\Z/n\Z)^{\Phi k\Z} \simeq \Sigma^\infty_+((\Z/n\Z)^{k\Z}) \simeq
	\begin{cases}
		\mathrm{ind}^k_{\nicefrac{k}{n}}(\S) & \text{if $n$ divides $k$} \\
		0 & \text{otherwise}
	\end{cases}
\]
so $\mathrm{L}\Sigma^\infty_+(\Z/n\Z) \simeq \ell_n(\S)$, where $\ell_n$ is the left adjoint of $\mathrm{sh}_n$ by Example~\ref{example:polygonic_shift}. We conclude that
\[
	\underline{\TR}(X)^{n\Z} \simeq \map_{\PgcSp}(\ell_n(\S), X) \simeq \map_{\PgcSp}(\S, \mathrm{sh}_n X) \simeq \TR(\mathrm{sh}_n X)
\]
which proves $(1)$. Note that $\mathrm{L}$ is right $t$-exact by construction, so $\underline{\TR}$ is left $t$-exact by adjunction. It remains to prove that $\underline{\TR}$ is also right $t$-exact. Let $X$ denote a connective polygonic spectrum. We wish to prove that $\underline{\TR}(X)^{n\Z} \simeq \TR(\mathrm{sh}_n X)$ is connective for every $n \geq 1$. First note that we may assume that $n = 1$. Using Lemma~\ref{lemma:PgcSp_rightKanExtended}, there is an equivalence of spectra
\[
	\TR(X) \simeq \varprojlim_k \TR(X_{[k]}),
\]
where $X_{[k]}$ is the polygonic spectrum with $(X_{[k]})_i = X_i$ for $i \leq k$ and $(X_{[k]})_i = 0$ when $i > k$. Therefore, using the Milnor sequence, it suffices to prove that $\TR(X_{[k]})$ is connective for every $k \geq 1$ and that the maps in the projective system above are surjective. To this end, we note that there is a fiber sequence of spectra
\[
	\TR(X_k) \to \TR(X_{[k]}) \to \TR(X_{[k-1]}),
\]
where $X_k$ denotes the polygonic spectrum concentrated in degree $k$ with value $X_k$. The claim now follows inductively since $\TR(X_k) \simeq X_{\h C_k}$ by Lemma~\ref{lemma:TR_polygonic_in_degree_n} and the homotopy orbits functor preserves connectivity. This proves the desired statement.
\end{proof}

\begin{remark}
The fiber sequence appearing in the last part of the proof of Proposition~\ref{proposition:TR_refines_qfgenZ} is an analogue of the cofiber sequence for $\TR^n$ established by Hesselholt--Madsen~\cite[Theorem 1.2]{HM97}.
\end{remark}

In conclusion, we have constructed a quasifinitely genuine $\Z$-spectrum $\underline{\TR}(X)$ for every polygonic spectrum $X$. This encodes Verschiebung and Frobenius maps on TR in a coherent fashion. For instance, if $R$ is an $\E_1$-ring and $M$ is an $R$-bimodule, then $\underline{\TR}(R, M)^{n\Z} \simeq \TR(R, M^{\otimes_R n})$ and the latter admits a canonical $C_n$-action. Using Proposition~\ref{proposition:TR_refines_qfgenZ}, we obtain a pair of maps
\begin{align*}
	F_n &: \TR(R, M) \to \TR(R, M^{\otimes_R n})^{\h C_n} \\
	V_n &: \TR(R, M^{\otimes_R n})_{\h C_n} \to \TR(R, M)
\end{align*}
induced by the inclusion map and the transfer map, respectively. These maps satisfy coherences as explained in~\textsection\ref{subsection:qfgen_Z}. We refer to $F_n$ as the $n$th Frobenius map and to $V_n$ as the $n$th Verschiebung map. As already mentioned, the Verschiebung and Frobenius maps are already encoded by means of the underlying finitely genuine $\Z$-strucure on TR, so the additional datum encoded by the quasifinitely genuine $\Z$-structure is that certain infinite sums of Verschiebung maps exist. Indeed, for every sequence $\{n_i\}$ of integers $n_i \geq 2$ which satisfy that $n_i \to \infty$ for $i \to \infty$, we may consider the terminal map $\coprod_{i \geq 1} \Z/n_i\Z \to \Z/\Z$ in $\QFin_{\Z}^{\mathrm{prop}}$, which induces a map of spectra
\[
	\displaystyle\sum_{i = 1}^{\infty} V_{n_i} : \prod_{i \geq 1} \TR(\mathrm{sh}_{n_i} X) \to \TR(X)
\]
which forms the infinite sum of the collection of Verschiebung maps $\{V_{n_i}\}_{i \geq 1}$ on $\TR(X)$. Once we remember these infinite sums of Verschiebung maps, we will be able to recover the polygonic spectrum $X$ from $\underline{\TR}(X)$ under suitable boundedness assumptions on $X$, establishing the integral version of~\cite[Theorem 3.21]{AN20}. More precisely, we prove the following result:

\begin{theorem} \label{theorem:TR_equiv_qfgen_pgcsp}
The adjunction $\mathrm{L} \dashv \underline{\TR}$ restricts to an equivalence of $\infty$-categories
\[
\begin{tikzcd}
	\Sp^{\Z, -}_{\qfgen} \arrow[yshift=0.7ex]{r}{\mathrm{L}} & \PgcSp^{-} \arrow[yshift=-0.7ex]{l}{\underline{\TR}}
\end{tikzcd}
\]
between the full subcategories of uniformly bounded below objects on each side.
\end{theorem}

Note that the adjunction $\mathrm{L} \dashv \underline{\TR}$ indeed restricts as claimed in Theorem~\ref{theorem:TR_equiv_qfgen_pgcsp} since both $\mathrm{L}$ and $\underline{\TR}$ are right $t$-exact by Proposition~\ref{proposition:TR_refines_qfgenZ}. Furthermore, by Proposition~\ref{proposition:geometric_fixedpoints_conservative}, the left adjoint $\mathrm{L}$ is conservative on those quasifinitely genuine $\Z$-spectra which are uniformly bounded below, so to establish the assertion of Theorem~\ref{theorem:TR_equiv_qfgen_pgcsp} it suffices to prove that the counit of the adjunction
\[
	\mathrm{L}\underline{\TR}(X) \to X
\]
is an equivalence for every uniformly bounded below polygonic spectrum $X$. Equivalently, in this situation, we wish to prove that the counit $\underline{\TR}(X)^{\Phi n\Z} \to X_n$ is an equivalence for every $n \geq 1$. The main idea is to reduce to the case where $X$ is concentrated in finitely many degrees, which we understand by virtue of the following pair of results:

\begin{lemma} \label{lemma:Xk_geometricfixedpoints_in_degree_k_higher}
Let $X$ denote a polygonic spectrum which is concentrated in degrees $\leq k$ for some integer $k \geq 1$. If $n \geq k$, then the canonical map $\underline{\TR}(X)^{n\Z} \to \underline{\TR}(X)^{\Phi n\Z}$ is an equivalence, and
\[
	\underline{\TR}(X)^{\Phi n\Z} \simeq
	\begin{cases}
		X_k & \text{if $n = k$} \\
		0 & \text{if $n > k$}
	\end{cases}
\]
\end{lemma}

\begin{proof}
It suffices to prove that the colimit appearing on the left hand side of the cofiber sequence
\[
	\underset{S \in \QFin_{\Z}^{\mathrm{prop}}}{\colim} (i_n^{\star} \underline{\TR}(X))^{S} \to \underline{\TR}(X)^{n\Z} \to \underline{\TR}(X)^{\Phi n\Z}
\]
vanishes. Indeed, if $S = \coprod_{i \geq 1} \Z/m_i\Z$ is an object of $\QFin_{\Z}^{\mathrm{prop}}$, meaning that $\{m_i\}$ is a sequence of integers with $m_i \geq 2$ for every $i \geq 1$ and which satisfy that $m_i \to \infty$ for $i \to \infty$, then
\[
	(i_n^{\star}\underline{\TR}(X))^{S} \simeq \prod_{i \geq 1} (i_n^{\star}\underline{\TR}(X))^{m_i\Z} \simeq \prod_{i \geq 1} \underline{\TR}(X)^{nm_i\Z} \simeq \prod_{i \geq 1} \TR(\mathrm{sh}_{nm_i} X) \simeq 0,
\]
since $\mathrm{sh}_{nm_i} X = 0$ by virtue of our assumption that $n \geq k$ and $m_i \geq 2$ for every $i \geq 1$. This proves the first assertion of the lemma, which in turn yields that $\underline{\TR}(X)^{\Phi n\Z} \simeq \underline{\TR}(X)^{n\Z} \simeq \TR(\mathrm{sh}_n X)$ provided that $n \geq k$. This also proves the final assertion of the lemma since $\mathrm{sh}_n X = 0$ whenever $n > k$ and $\mathrm{sh}_k X$ is the polygonic spectrum concentrated in degree one with value $X_k$ which gives that $\TR(\mathrm{sh}_k X) \simeq X_k$ by Example~\ref{example:polygonic_concentrated_degree_1}\footnote{We refer to Example~\ref{example:polygonic_concentrated_degree_1} instead of Lemma~\ref{lemma:TR_polygonic_in_degree_n} since in the case when $X$ is concentrated in degree $1$, we do not need impose the bounded below assumption to identify $\TR(X) \simeq X$}.
\end{proof}

For a polygonic spectrum $X$ which is concentrated in degrees $\leq k$, Lemma~\ref{lemma:Xk_geometricfixedpoints_in_degree_k_higher} identifies the geometric fixedpoints $\underline{\TR}(X)^{\Phi n\Z}$ whenever $n \geq k$. The following result allows us to additionally identify the geometric fixedpoints for $n < k$, which requires a more elaborate argument.

\begin{proposition} \label{proposition:geometricfixedpoints_Xk}
Let $X$ denote a uniformly bounded below polygonic spectrum which is concentrated in degrees $\leq k$ for some integer $k \geq 1$. The following assertions are satisfied:
\begin{enumerate}[leftmargin=2em, topsep=5pt, itemsep=5pt]
	\item If $n \in \{1, \ldots, k\}$, then $\underline{\TR}(X)^{\Phi n\Z} \simeq X_n$. 
	\item If $n \geq k+1$, then $\underline{\TR}(X)^{\Phi n\Z} \simeq 0$.
\end{enumerate}
\end{proposition}

\begin{proof}
It remains to be proven that $\underline{\TR}(X)^{\Phi n\Z} \simeq X_n$ for $k \geq 2$ and $n \in \{1, \ldots, k-1\}$ by virtue of Lemma~\ref{lemma:Xk_geometricfixedpoints_in_degree_k_higher}. We first argue that this can be deduced from the following assertion:

\begin{enumerate}[leftmargin=2em, topsep=5pt, itemsep=5pt]
	\item[$(\star)$] For a bounded below polygonic spectrum $Y$ concentrated in degree $k$ for $k \geq 2$, $\underline{\TR}(Y)^{\Phi\Z} \simeq 0$.
\end{enumerate}
Indeed, there is a fiber sequence of spectra
\[
	\underline{\TR}(X_k)^{\Phi n\Z} \to \underline{\TR}(X)^{\Phi n\Z} \to \underline{\TR}(X_{[k-1]})^{\Phi n\Z},
\]
where $X_k$ denotes the polygonic spectrum concentrated in degree $k$ with value $X_k$, and $X_{[k-1]}$ is the polygonic spectrum which agrees with $X$ up to degree $k-1$. We claim that $\underline{\TR}(X_k)^{\Phi n\Z} \simeq 0$, which in turn proves the desired statement since then
\[
	\underline{\TR}(X)^{\Phi n\Z} \simeq \underline{\TR}(X_{[k-1]})^{\Phi n\Z} \simeq \cdots \simeq \underline{\TR}(X_{[n]})^{\Phi n\Z} \simeq X_n.
\]
We have that $\underline{\TR}(X_k)^{\Phi n\Z} \simeq \underline{\TR}(\mathrm{sh}_n X_k)^{\Phi \Z}$, hence if $n$ divides $k$, then $\mathrm{sh}_n X_k$ is the polygonic spectrum concentrated in degree $\nicefrac{k}{n} \geq 2$ with value $X_k$, which in turn means that
\[
	\underline{\TR}(X_k)^{\Phi n\Z} \simeq \underline{\TR}(\mathrm{sh}_n X_k)^{\Phi \Z} \simeq 0
\]
by virtue of $(\star)$. If $n$ does not divide $k$, then $\mathrm{sh}_n X_k = 0$, which also gives that $\underline{\TR}(X_k)^{\Phi n\Z} \simeq 0$. Finally, we prove $(\star)$ by proving that the canonical map $\colim_{S \in \QFin_{\Z}^{\mathrm{prop}}} \underline{\TR}(Y)^S \to \underline{\TR}(Y)^{\Z}$ is an equivalence. We first note that for every $n$ dividing $k$, the transfer map 
\[
	(\underline{\TR}(Y)^{n\Z})_{\h C_n} \to \underline{\TR}(Y)^{\Z}
\]
is an equivalence. Indeed, since $\mathrm{sh}_n X$ is the polygonic spectrum concentrated in degree $\nicefrac{k}{n}$ with value $Y$, there is a natural equivalence $\underline{\TR}(Y)^{n\Z} \simeq X_{\h C_{\nicefrac{k}{n}}}$ by Lemma~\ref{lemma:TR_polygonic_in_degree_n}, and the transfer map above is identified with the canonical equivalence $(Y_{\h C_{\nicefrac{k}{n}}})_{\h C_n} \simeq X_{\h C_k} \simeq \TR(Y)$, where the final equivalence similarly is a consequence of Lemma~\ref{lemma:TR_polygonic_in_degree_n}. Combining this with Proposition~\ref{proposition:formula_fiber_term_geometric_fxp}, we conclude that the transfer maps induce an equivalence of spectra
\begin{align*}
	\underset{S \in \QFin_{\Z}^{\mathrm{prop}}}{\colim} \underline{\TR}(Y)^S \simeq &\,\Big\vert
	\begin{tikzcd}[ampersand replacement = \&, column sep = scriptsize]
		\cdots \arrow{r} \arrow[yshift=0.8ex]{r} \arrow[yshift=-0.8ex]{r} \& \displaystyle\prod_{p, q \in \langle n \rangle} (\underline{\TR}(Y)^{\mathrm{gcd}(p, q)\Z})_{\h C_{\mathrm{gcd}(p, q)}} \arrow[yshift=0.4ex]{r} \arrow[yshift=-0.4ex]{r} \& \displaystyle\prod_{p \in \langle n \rangle} (\underline{\TR}(Y)^{p\Z})_{\h C_p}
	\end{tikzcd}
	\Big\vert \\
	\xrightarrow{\simeq} &\,\underline{\TR}(Y)^{\Z} \otimes \Sigma^\infty_+ \vert S \vert,
\end{align*}
where $S$ is the simplicial set with $i$-simplices given by $S_i = \{(p_1, \ldots, p_i) \in \mathbb{P}^i \mid p_1, \ldots, p_i \in \langle n \rangle\}$. Since $|S|$ is contractible, we conclude that the desired map is an equivalence, so $\underline{\TR}(Y)^{\Phi \Z} \simeq 0$ proving $(\star)$. This marks the end of the proof.
\end{proof}

Armed with Proposition~\ref{proposition:geometricfixedpoints_Xk}, we proceed to prove Theorem~\ref{theorem:TR_equiv_qfgen_pgcsp}. 

\begin{proof}[Proof of Theorem~\ref{theorem:TR_equiv_qfgen_pgcsp}]
We first prove that the counit $\mathrm{L}\underline{\TR}(X) \to X$ is an equivalence for every polygonic spectrum $X$ which is uniformly bounded below. Specifically, we prove that the counit $\underline{\TR}(X)^{\Phi n\Z} \to X_n$ is an equivalence for every $n \geq 1$. Using Lemma~\ref{lemma:PgcSp_rightKanExtended}, we may write
\[
	X \simeq \varprojlim_k X_{[k]},
\]
where $X_{[k]}$ is the polygonic spectrum with $(X_{[k]})_i = X_k$ for $i \leq k$ and $(X_{[k]})_i = 0$ for $i > k$ as in the proof of Proposition~\ref{proposition:TR_refines_qfgenZ}. Consequently, we have that
\[
	\underline{\TR}(X)^{\Phi n\Z} \simeq \big( \varprojlim_k \underline{\TR}(X_{[k]}) \big)^{\Phi n\Z} \simeq \varprojlim_k \underline{\TR}(X_{[k]})^{\Phi n\Z},
\]
where the first equivalence follows since $\underline{\TR}$ is a right adjoint, and the second equivalence follows from Corollary~\ref{corollary:geometric_preserves_limits} since $\underline{\TR}$ is $t$-exact by Proposition~\ref{proposition:TR_refines_qfgenZ}. Using Proposition~\ref{proposition:geometricfixedpoints_Xk}, we conclude that the tower $\{\underline{\TR}(X_{[k]})^{\Phi n\Z}\}_k$ is eventually constant with value $\underline{\TR}(X_{[n]})^{\Phi n\Z} \simeq X_n$. To finish the proof, we note that the unit of the adjunction is an equivalence under the uniformly bounded below assumption since $\mathrm{L}$ is conservative by Proposition~\ref{proposition:geometric_fixedpoints_conservative}. This proves the desired statement.
\end{proof}

Finally, we construct an action of the multiplicative monoid $\N$ on the $\infty$-category of quasifinitely genuine $\Z$-spectra, where $n \in \N$ acts via the functor $i_n^{\star}$ from Construction~\ref{construction:shift_qfgen}. This determines an action of $\N$ on the $\infty$-category of bounded below polygonic spectra via Theorem~\ref{theorem:TR_equiv_qfgen_pgcsp}.

\begin{construction} \label{construction:Naction_qfgen_pgcsp}
The multiplicative monoid $\N$ acts on the category $\QFin_{\Z}$, where $n \in \N$ acts on $\QFin_{\Z}$ via the assignment $S \mapsto nS$. This determines a functor of $\infty$-categories
\[
	\B\N \to \Cat_\infty^{\mathrm{lex}, \op},
\]
since it factors over $\Cat_1^{\mathrm{lex}}$. Using the functoriality of the construction $T \mapsto \Fun^{\mathrm{vadd}}(\mathrm{Span}(T), \Sp)$, we obtain our desired functor of $\infty$-categories $\B\N \to \Cat_\infty$. This determines an action of $\N$ on the $\infty$-category $\Sp^{\Z}_{\qfgen}$ through the functors $i_n^{\star}$ from Construction~\ref{construction:shift_qfgen}. We obtain an action of $\N$ on the $\infty$-category $\PgcSp^-$ of uniformly bounded below polygonic spectra by Theorem~\ref{theorem:TR_equiv_qfgen_pgcsp}, where $n \in \N$ acts via the functor $\mathrm{sh}_n : \PgcSp^- \to \PgcSp^-$ from Example~\ref{example:polygonic_shift}. 
\end{construction}


\section{Polygonic THH} \label{section:polygonicTHH}
In this section, we discuss the construction of topological Hochschild homology with coefficients as a polygonic spectrum in the sense of~\textsection\ref{section:pgcsp}. In~\textsection\ref{subsection:cyclicbimdouleoperad}, we introduce an $\infty$-operad $\CycBMod_n^{\otimes}$ which is designed such that its algebras precisely encode cyclic graphs labelled by rings and bimodules as discussed in~\textsection\ref{section:introductionpolygonic}. In~\textsection\ref{subsection:constructionTHHpolygonic}, we introduce the $\infty$-category $\Lambda^{\CycBMod}$ and prove that topological Hochschild homology refines to a trace theory on $\Lambda^{\CycBMod}$ (see Theorem~\ref{theorem:HH_trace}). This allows us to construct the desired polygonic structure on THH with coefficients (see Theorem~\ref{theorem:THH_polygonic}). 

\subsection{The cyclic bimodule operad} \label{subsection:cyclicbimdouleoperad}
We introduce a combinatorial formalism for encoding cyclic graphs labelled by rings and bimodules in a symmetric monoidal $\infty$-category. We begin by briefly reviewing the cyclic category, referring to~\cite[Appendix B]{NS18} for a more comprehensive treatment. The cyclic category was introduced by Connes~\cite{Con83}. 

\begin{definition}
A parasimplex is a nonempty linearly ordered set $I$ equipped with an action of the group $\Z$ denoted $+ : I \times \Z \to I$, such that the following pair of conditions are satisfied:
\begin{enumerate}[leftmargin=2em, topsep=5pt, itemsep=5pt]
	\item For every pair of elements $x$ and $x'$ of $I$, the set $\{y \in I \mid x \leq y \leq x'\}$ is finite.
	\item If $x$ is an element of $I$, then $x < x + 1$. 
\end{enumerate}
A paracyclic morphism is a morphism of set $f : I \to J$ which is nondecreasing and $\Z$-equivariant. The paracyclic category $\Lambda_{\infty}$ is the category whose objects are parasimplices and whose morphisms are paracyclic morphisms. 
\end{definition}

Alternatively, we can describe the paracyclic category as follows:

\begin{example}
Let $\mathrm{Lin}_{\Z}$ denote the category of nonempty linearly ordered sets with a $\Z$-action. The paracyclic category $\Lambda_\infty$ is isomorphic to the full subcategory of $\mathrm{Lin}_{\Z}$ spanned by the objects of the form $\frac{1}{n}\Z$ for $n \geq 1$. We regard $\frac{1}{n}\Z$ as an object of $\mathrm{Lin}_{\Z}$ with its usual ordering and $\Z$-action given by addition. Let $[n]_{\Lambda}$ denote the set $\frac{1}{n}\Z$ regarded as an object of $\Lambda_\infty$. 
\end{example}

\begin{example}
If $Q$ is a nonempty linearly ordered set, then the product $Q \times \Z$ can be regarded as an object of $\Lambda_{\infty}$ with the reverse lexicographic ordering and $\Z$-action by $(q, m) + n = (q, m + n)$. The construction $Q \mapsto Q \times \Z$ defines a faithful functor $\Delta \to \Lambda_\infty$, which is also essentially surjective since there is an isomorphism of parasimplices $[n] \times \Z \simeq [n+1]_{\Lambda}$ for $n \geq 0$.  
\end{example}

The cyclic category is defined as follows:

\begin{definition}
The cyclic category $\Lambda$ is the category with the same objects as $\Lambda_\infty$, and
\[
	\Hom_{\Lambda}(I, J) = \Hom_{\Lambda_\infty}(I, J)/\Z,
\]
where $\Z$ acts on the set of morphisms $\Hom_{\Lambda_\infty}(I, J)$ in $\Lambda_\infty$ by the formula $(f+n)(x) = f(x) + n$. 
\end{definition}

\begin{example}
Regarded as an object of $\Lambda_{\infty}$, we picture $[n]_{\Lambda}$ as an infinite directed graph whose vertices are labelled by the elements of $[n]_{\Lambda}$. There is a directed edge from $\nicefrac{a}{n}$ to the consecutive element $\nicefrac{a+1}{n}$ for every $a \in \Z$. Regarded as an object of $\Lambda$, we picture $[n]_{\Lambda}$ as a cyclic graph with $n$ vertices labelled by the elements $\nicefrac{a}{n}$ for every $0 \leq a \leq n-1$. There is a directed edge from $\nicefrac{a}{n}$ to the consecutive element $\nicefrac{a+1}{n}$ for every $a \in \Z/n$. In other words, the object $[n]_{\Lambda} \in \Lambda$ is obtained by wrapping the object $[n]_{\Lambda} \in \Lambda_\infty$ around itself identifying each of the elements $\nicefrac{a}{n}$ for $a \in n\Z$. 
\end{example}

Let us now describe our goal for this section. Informally, for every object of the cyclic category, we want to label the associated cyclic graph with rings and bimodules. More precisely, for every object $I$ of $\Lambda$, we construct a coloured operad $\CycBMod_I$ which is designed in such a way that the algebras for this operad precisely encode cyclic graphs labelled by algebra objects and bimodule objects between them. For instance, if $I = [3]_{\Lambda}$, then we depict an algebra object for the operad $\CycBMod_{[3]_{\Lambda}}$ in a symmetric monoidal $\infty$-category $\cat{C}$ by
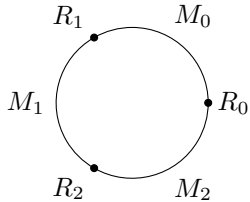
\[
\begin{tikzpicture}
	\filldraw (1,0) circle (1.3pt);
	\filldraw (-0.5, 0.86602) circle (1.3pt);
	\filldraw (-0.5, -0.86602) circle (1.3pt);

	\node at (1, 0) [right] {$R_0$};
	\node at (-0.5, 0.86602) [above left] {$R_1$};
	\node at (-0.5, -0.86602) [below left] {$R_2$};

	\node at (0.8, 1.15) {$M_0$};
	\node at (-1.4, 0) {$M_1$};
	\node at (0.8, -1.15) {$M_2$};

	\draw (1,0) arc[start angle = 0, end angle = 120, x radius = 1, y radius = 1];
	\draw (-0.5, 0.86602) arc[start angle = 120, end angle = 240, x radius = 1, y radius = 1];
	\draw (-0.5, -0.86602) arc[start angle = 240, end angle = 360, x radius = 1, y radius = 1];
\end{tikzpicture}
\]
where $R_i \in \Alg(\cat{C})$ and $M_i$ is an $R_i$-$R_{i+1}$-bimodule object of $\cat{C}$. The set of colours of $\CycBMod_I$ will be the set of all vertices and all possible paths between vertices on the cyclic graph determined by $I$. We will begin by making this precise. 

\begin{construction}
Let $I$ denote an object of $\Lambda_\infty$. The set of paths in $I$ is defined by
\[
	\Path(I) = \{(x, y) \in I \times I \mid x \leq y \leq x+1\}/\Z,
\]
where the $\Z$-action is given by $(x, y) + n = (x + n, y + n)$. The assignment $I \mapsto \Path(I)$ determines a functor $\Lambda_\infty \to \mathrm{Set}$ since every morphism $f : I \to J$ in $\Lambda_\infty$ induces a function $\Path(I) \to \Path(J)$ defined by $(x, y) \mapsto (f(x), f(y))$. The functor $\Path : \Lambda_\infty \to \mathrm{Set}$ factors over $\Lambda$ since $\Lambda = \Lambda_\infty/\B\Z$. 
\end{construction}

\begin{notation}
Let $I = [n]_{\Lambda}$ denote an object of $\Lambda$. We will use the following notation:
\begin{enumerate}[leftmargin=2em, topsep=5pt, itemsep=5pt]
	\item Let $v_a$ denote the path $(\nicefrac{a}{n}, \nicefrac{a}{n}) \in \Path(I)$ for every $0 \leq a \leq n-1$.
	\item For $0 \leq a, b < n$, let $e_{[a,b]}$ denote the path $(\nicefrac{a}{n}, \nicefrac{b}{n})$ if $a < b$ and $(\nicefrac{a}{n}, \nicefrac{b}{n+1})$ otherwise.
\end{enumerate}
The elements $v_a \in \Path(I)$ correspond to the vertices of the cyclic graph given by $I$. The element $e_{[a, b]} \in \Path(I)$ corresponds to the path from $v_a$ to $v_b$ counting the indices modulo $n$.
\end{notation}

\begin{example}
In the following examples, the sets of paths are given by
\begin{align*}
	\Path([1]_{\Lambda}) &= \{v_0, e_{[0, 0]}\}. \\
	\Path([2]_{\Lambda}) &= \{v_0, v_1, e_{[0,1]}, e_{[1,0]}, e_{[0,0]}, e_{[1,1]}\}. \\
	\Path([3]_{\Lambda}) &= \{v_0, v_1, v_2, e_{[0,1]}, e_{[1,2]}, e_{[2,0]}, e_{[0,2]}, e_{[1,0]}, e_{[2,1]}, e_{[0,0]}, e_{[1,1]}, e_{[2,2]}\}.
\end{align*}
\end{example}

To facilitate the definition of the operad $\CycBMod_I$, we introduce the auxiliary notion of an admissible sequence in $\Path(I)$ relative to a specified path. Let $I$ denote an object of $\Lambda$.  

\begin{definition} \label{definition:admissible}
Let $(\gamma_1, \ldots, \gamma_n)$ be a sequence in $\Path(I)$ and let $\gamma$ be an element of $\Path(I)$. The sequence $(\gamma_1, \ldots, \gamma_n)$ is $\gamma$-admissible if there exists a tuple $\{x_0 \leq \ldots \leq x_n \leq x_0 + 1\}$ in $I$ such that $\gamma = [(x_0, x_n)]$ and $\gamma_i = [(x_{i-1}, x_i)]$ for every $1 \leq i \leq n$. 
\end{definition}

For instance, both $(v_0, e_{[0,1]}, v_1, e_{[1,2]}, v_2)$ and $(v_0, e_{[0,2]}, v_2)$ are examples of $e_{[0,2]}$-admissible sequences in $\Path([3]_{\Lambda})$. With this in mind, we define the operad $\CycBMod_I$ as follows:

\begin{definition}
Let $I$ denote an object of $\Lambda$. We define an operad $\mathbf{CycBMod}_I$ as follows:
\begin{enumerate}[leftmargin=2em, topsep=5pt, itemsep=5pt]
	\item The set of objects of $\mathbf{CycBMod}_I$ is given by $\Path(I)$.

	\item For every finite sequence of paths $(\gamma_1, \ldots, \gamma_n)$ in $\mathbf{CycBMod}_I$, the set of operations
	\[
		\mathrm{Mul}_{\mathbf{CycBMod}_I}((\gamma_1, \ldots, \gamma_n); \gamma)	
	\]
	is defined as follows: If $(\gamma_1, \ldots, \gamma_n)$ is $\gamma$-admissible, then
	\[
		\mathrm{Mul}_{\mathbf{CycBMod}_I}((\gamma_1, \ldots, \gamma_n); \gamma) = \{\sigma \in \Sigma_n \mid \text{$(\gamma_{\sigma(1)}, \ldots, \gamma_{\sigma(n)})$ is $\gamma$-admissible}\},
	\]
	and otherwise the set of operations is empty. 

	\item The composition law on $\mathbf{CycBMod}_I$ is determined by the composition of linear orderings. 
\end{enumerate}
We refer to $\mathbf{CycBMod}_I$ as the cyclic bimodule operad for $I$. Let $\CycBMod_I^{\otimes}$ denote the underlying $\infty$-operad of $\mathbf{CycBMod}_I$, which we refer to as the cyclic bimodule $\infty$-operad for $I$. 
\end{definition}

\begin{remark}
Assume that $I = [1]_{\Lambda}$. The operad $\mathbf{CycBMod}_1$ is a coloured suboperad of $\mathbf{BM}$, where $\mathbf{BM}$ denotes the coloured bimodule operad defined in~\cite[Definition 4.3.1.1]{Lur17}.
\end{remark}

It will be convenient to work with the symmetric monoidal envelope of the $\infty$-operad $\CycBMod_I^{\otimes}$. We obtain the following explicit description of $(\CycBMod_I^{\otimes})_{\act}$ for every object $I$ of $\Lambda$. 

\begin{remark} \label{remark:active_part}
The symmetric monoidal envelope $(\CycBMod_I^{\otimes})_{\act}$ is described as follows:
\begin{enumerate}[leftmargin=2em, topsep=5pt, itemsep=5pt]
	\item The objects of $(\CycBMod_I^{\otimes})_{\act}$ are finite sets coloured by $\Path(I)$.

	\item A morphism from $X$ to $Y$ in $(\CycBMod_I^{\otimes})_{\act}$ is a map of finite sets $f : X \to Y$ equipped with a linear ordering on $f^{-1}(y)$ for every element $y$ of $Y$, such that the preimage of a $\gamma$-coloured element is coloured by a $\gamma$-admissible sequence.

	\item The composition is given by composition in $\mathrm{Fin}$ with lexicographic ordering on the preimages.
\end{enumerate}
The symmetric monoidal structure on $(\CycBMod_I^{\otimes})_{\act}$ is determined by disjoint union of finite sets. For every $e \in \Path(I)$, let $\ast_e$ denote the set with one element coloured by $e$.
\end{remark}

Note that $(\CycBMod_I^{\otimes})_{\act}$ is an ordinary symmetric monoidal category. In fact, the reader is welcome to regard Remark~\ref{remark:active_part} as a definition of $(\CycBMod_I^{\otimes})_{\act}$.  

\begin{remark}
Let $\cat{C}$ denote a symmetric monoidal $\infty$-category and assume that
\[
	F : (\CycBMod_n^{\otimes})_{\act} \to \cat{C}
\]
is a symmetric monoidal functor of $\infty$-categories. Then $R_a = F(\ast_{v_a})$ admits the structure of an algebra object of $\cat{C}$ for every $0 \leq a \leq n-1$, whose multiplication is induced by $\ast_{v_a} \amalg \ast_{v_a} \to \ast_{v_a}$. We have that $M_a = F(\ast_{e_{[a,a+1]}})$ admits the structure of an $R_a$-$R_{a+1}$-bimodule object of $\cat{C}$ for every $0 \leq a \leq n-1$, whose structure map is induced by $\ast_{v_a} \amalg \ast_{e_{[a, a+1]}} \amalg \ast_{v_{a+1}} \to \ast_{e_{[a, a+1]}}$. 
\end{remark}

In conclusion, for every integer $n \geq 1$, we have defined an $\infty$-operad $\CycBMod_n^{\otimes}$ such that the datum of a symmetric monoidal functor of $\infty$-categories $(\CycBMod_n^{\otimes})_{\act} \to \cat{C}$ precisely encodes a cyclic graph labelled by algebras and bimodules between them as informally described above. We additionally need to encode that the value of our functor on $\ast_{e_{[a, a+2]}}$ is equivalent to the $R_a$-$R_{a+2}$-bimodule given by $M_a \otimes_{R_{a+1}} M_{a+1}$. Informally, the relative tensor product $M_a \otimes_{R_{a+1}} M_{a+1}$ is obtained as the geometric realization of the simplicial object $\Delta^{\op} \to \cat{C}$ given by the assignment
\[
	Q \mapsto M_a \otimes (R_{a+1})^{\otimes Q} M_{a+1}.
\] 
We formalize the description of the simplicial object above as follows. For every nonempty finite linearly ordered set $Q$, the set of cuts in $Q$ is defined by
\[
	\mathrm{Cut}(Q) = \{Q = Q_0 \amalg Q_1 \mid Q_0 < Q_1\}.
\] 
If $Q \to Q'$ is a morphism of nonempty finite linearly ordered sets, then there is an induced map $\mathrm{Cut}(Q') \to \mathrm{Cut}(Q)$ determined by sending a cut to its preimage. The construction $Q \mapsto \mathrm{Cut}(Q)$ determines a functor $\mathrm{Cut} : \Delta^{\op} \to \Fin$.

\begin{construction} \label{construction:cut}
Let $I$ denote an object of $\Lambda$ and let $e_{[a, b]} \in \Path(I)$ be a path of length $2$. The construction $Q \mapsto \mathrm{Cut}(Q)$ described above refines to a functor 
\[
	c_{[a, b]} : \mathrm{Cut} : \Delta^{\op} \to (\CycBMod_I^{\otimes})_{\act},
\]
where the nontrivial cuts are coloured by $v_{a+1}$, while $(Q, \varnothing)$ is coloured by $e_{[a, a+1]}$ and $(\varnothing, Q)$ is coloured by $e_{[a+1, b]}$. If $f : Q \to Q'$ is a morphism of nonempty linearly ordered sets and $(Q_0, Q_1)$ cut of $Q$, then $\mathrm{Cut}(f)^{-1}(Q_0, Q_1)$ is totally ordered by declaring that $(R_0, R_1) \leq (R'_0, R'_1)$ precisely if $R_0 \subseteq R'_0$ and $R'_1 \subseteq R_1$. 
\end{construction}

Finally, we have arrived at the main definition of this section.

\begin{definition}
Let $\cat{C}$ denote a symmetric monoidal $\infty$-category which admits sifted colimits. For every object $I$ of $\Lambda$, let $\CycBMod_I(\cat{C})$ denote the full subcategory of 
\[
	\Fun^{\otimes}((\CycBMod_I^{\otimes})_{\act}, \cat{C})
\]
spanned by those symmetric monoidal functors $F$ which satisfy that the canonical map
\[
	\mathrm{colim}(\Delta^{\op} \xrightarrow{c_{[a, b]}} (\CycBMod_I^{\otimes})_{\act} \xrightarrow{F} \cat{C}) \to F(\ast_{e_{[a, b]}})
\]
is an equivalence for every path $e_{[a, b]} \in \Path(I)$ of length $2$. 
\end{definition}

Unwinding the definition, we see that a symmetric monoidal functor $F : (\CycBMod_I^{\otimes})_{\act} \to \cat{C}$ refines to an object of the $\infty$-category $\CycBMod_I(\cat{C})$ precisely if the canonical map
\[
	M_a \otimes_{R_{a+1}} M_{a+1} \to F(\ast_{e_{[a, b]}})
\]
is an equivalence for every path of length $2$. We remark here that this assertion holds for a path of arbitrary length (see Lemma~\ref{lemma:relative_tensor_long_paths}). We end this section, by discussion the functoriality of the construction $I \mapsto \CycBMod_I(\cat{C})$.

\begin{construction} \label{construction:functor_CycBMod}
Every morphism $f : I \to J$ in $\Lambda$, induces a symmetric monoidal functor
\[
	f_\ast : (\CycBMod_I^{\otimes})_{\act} \to (\CycBMod_J^{\otimes})_{\act}
\]
determined by regarding a finite set coloured by $\Path(I)$ as a finite set coloured by $\Path(J)$ by postcomposition with $\Path(f) : \Path(I) \to \Path(J)$. Consequently, we obtain a functor
\[
	\Fun^{\otimes}((\CycBMod_J^{\otimes})_{\act}, \cat{C}) \to \Fun^{\otimes}((\CycBMod_I^{\otimes})_{\act}, \cat{C})
\]
which restricts to a functor of $\infty$-categories $\CycBMod_f(\cat{C}) : \CycBMod_J(\cat{C}) \to \CycBMod_I(\cat{C})$. Let $\CycBMod : \Lambda^{\op} \to \Cat_{\infty}$ denote the functor determined by the construction $I \mapsto \CycBMod_I(\cat{C})$. 
\end{construction} 

\begin{example} \label{example:tensorproducts}
Let $\delta_1 : [1]_{\Lambda} \to [2]_{\Lambda}$ denote the morphism in $\Lambda$ induced by $\delta_1 : [0] \to [1]$ in $\Delta$. To distinguish them from $\Path([1]_{\Lambda})$, we denote the elements of $\Path([2]_{\Lambda})$ as follows:
\[
	\Path([2]_{\Lambda}) = \{w_0, w_1, f_{[0,1]}, f_{[1, 0]}, f_{[0,0]}, f_{[1,1]}\}.
\]
The map $\Path([1]_{\Lambda}) \to \Path([2]_{\Lambda})$ is given by $v_0 \mapsto w_0$ and $e_{[0, 1]} \mapsto f_{[0,0]}$, and the functor
\[
	\CycBMod_{\delta_1}(\cat{C}) : \CycBMod_2(\cat{C}) \to \CycBMod_1(\cat{C})
\]
is given by the construction $(R_0, M_0; R_1, M_1) \mapsto (R_0, M_0 \otimes_{R_1} M_1)$. Similarly, the functor
\[
	\CycBMod_{\delta_0}(\cat{C}) : \CycBMod_2(\cat{C}) \to \CycBMod_1(\cat{C})	
\]
is given by the construction $(R_0, M_0; R_1, M_1) \mapsto (R_1, M_1 \otimes_{R_0} M_0)$.
\end{example}

\begin{example}
Let $\tau : [2]_{\Lambda} \to [2]_{\Lambda}$ denote the morphism in $\Lambda$ which corresponds to the self-map of $S^1$ given by rotation with $\pi$. In this case, the map $\Path([2]_{\Lambda}) \to \Path([2]_{\Lambda})$ is given by
\[
	v_0 \mapsto w_1 \quad v_1 \mapsto w_0 \quad e_{[0,1]} \mapsto f_{[1, 0]} \quad e_{[1,0]} \mapsto f_{[0, 1]} \quad e_{[0,0]} \mapsto f_{[1, 1]} \quad e_{[1,1]} \mapsto f_{[0, 0]}
\]
and the induced functor of $\infty$-categories
\[
	\CycBMod_{\tau}(\cat{C}) : \CycBMod_2(\cat{C}) \to \CycBMod_2(\cat{C})
\]
is given by the construction $(R_0, M_0; R_1, M_1) \mapsto (R_1, M_1; R_0, M_0)$. 
\end{example}

As mentioned above, we have the following result:

\begin{lemma} \label{lemma:relative_tensor_long_paths}
Let $I$ be an object of $\Lambda$ and let $\cat{C}$ denote a symmetric monoidal $\infty$-category which admits geometric realizations. If $F$ is an object of $\CycBMod_I(\cat{C})$, then the canonical map
\[
	M_a \otimes_{R_{a+1}} M_{a+1} \otimes_{R_{a+2}} \cdots \otimes_{R_{b-1}} M_{b-1} \to F(\ast_{e_{[a, b]}})
\]
is an equivalence for every path $e_{[a,b]} \in \Path(I)$. 
\end{lemma}

\begin{proof}
Assume that $F \in \CycBMod_n(\cat{C})$ for $n \geq 1$. If $e_{[a, b]}$ has length $\ell$, then there is a morphism $f : [n-\ell+1]_{\Lambda} \to [n]_{\Lambda}$ in $\Lambda$ obtained by successively invoking Example~\ref{example:tensorproducts}, such that
\[
	F(\ast_{e_{[a, b]}}) \simeq \CycBMod_f(F)(\ast_{e_{[a, a+1]}}) \simeq M_a \otimes_{R_{a+1}} M_{a+1} \otimes_{R_{a+2}} \cdots \otimes_{R_{b-1}} M_{b-1},
\]
where that $\CycBMod_f(F) \in \CycBMod_{n-\ell+1}(\cat{C})$.
\end{proof}

\subsection{Trace theories} \label{subsection:constructionTHHpolygonic}
Using the formalism developed in~\textsection\ref{subsection:cyclicbimdouleoperad}, we introduce an $\infty$-category $\Lambda^{\CycBMod}$ which neatly packages the collection of $\infty$-categories $\CycBMod_I$ as we let $I$ vary over $\Lambda$. This is a variant of the $\infty$-category $\Lambda^{\mathrm{st}}$ introduced by the third author in~\cite{HS19} which classifies cyclic graphs labelled by stable $\infty$-categories and profunctors. We introduce the notion of a trace theory on $\Lambda^{\CycBMod}$ and prove that Hochschild homology refines to a trace theory on $\Lambda^{\CycBMod}$. We use this to obtain the desired polygonic structure on $\THH(R, M)$. Finally, we note that the idea of regarding Hochschild homology as a trace theory is due to Kaledin~\cite{Kal15}. 

\begin{definition}
Let $\cat{C}$ denote a symmetric monoidal $\infty$-category which admits sifted colimits.
\begin{enumerate}[leftmargin=2em, topsep=5pt, itemsep=5pt]
	\item Let $p_{\cat{C}} : \Lambda^{\CycBMod}_{\cat{C}} \to \Lambda^{\op}$ denote the cocartesian fibration classified by the functor
	\[
		\CycBMod : \Lambda^{\op} \to \Cat_{\infty}
	\]
	determined by the construction $I \mapsto \CycBMod_I(\cat{C})$ as discussed in Construction~\ref{construction:functor_CycBMod}.

	\item A trace theory valued in $\cat{C}$ is a functor of $\infty$-categories
	\[
		T : \Lambda^{\CycBMod}_{\cat{C}} \to \cat{C}
	\]
	which carries $p_{\cat{C}}$-cocartesian morphisms in $\Lambda^{\CycBMod}_{\cat{C}}$ to equivalences of $\cat{C}$. 
\end{enumerate}
\end{definition}

Unwinding the construction, we obtain the following explicit description of $\Lambda_{\cat{C}}^{\CycBMod}$. 

\begin{remark}
An object of the total space $\Lambda^{\CycBMod}_{\cat{C}}$ is given by a pair $(I, F)$, where $I \in \Lambda$ and $F$ is an object of the $\infty$-category $\CycBMod_I(\cat{C})$. A morphism from $(I, F)$ to $(J, G)$ in $\Lambda_{\cat{C}}^{\CycBMod}$ is a pair $(f, \alpha)$ consisting of a morphism $f : J \to I$ in $\Lambda$, and a natural transformation
\[
	\alpha : \CycBMod_f(F) \to G
\]
in the $\infty$-category $\CycBMod_J(\cat{C})$. Every $p_{\cat{C}}$-cocartesian morphism of $\Lambda_{\cat{C}}^{\CycBMod}$ can be written as a composition of rotations, contractions, and insertions since this is true in $\Lambda^{\op}$ regarded as the total space of the cocartesian fibration $\Lambda^{\op} \to \B\T$. 
\end{remark}

Let $T : \Lambda^{\CycBMod} \to \Sp$ denote a trace theory. The main point of introducing $\Lambda^{\CycBMod}$ is that it organizes the combinatorics for neatly expressing the cyclic invariance of the trace theory $T$. For instance, if $M$ is an $R$-$S$-bimodule and $N$ is an $S$-$R$-bimodule, then the maps
\[
\begin{tikzcd}[column sep = scriptsize]
	T(R, M \otimes_S N) & T([2]_{\Lambda}, (R, M; S, N)) \arrow[swap]{l}{\simeq} \arrow{r}{\simeq} & T([2]_{\Lambda}, (S, N; R, M)) \arrow{r}{\simeq} & T(S, N \otimes_R M)
\end{tikzcd}
\]
are equivalences since the first and last maps are induced by contractions, and the middle map is induced by a rotation all of which are cocartesian edges of $\Lambda^{\CycBMod}$. It is precisely the cyclic invariance that will provide us with an action of $C_n$ on $\THH(R, M^{\otimes_R n})$ which is a crucial part of the polygonic structure. Presently, our goal is to construct a trace theory $\HH : \Lambda^{\CycBMod}_{\cat{C}} \to \cat{C}$. For every pair $(R, M)$ consisting of an algebra $R \in \Alg(\cat{C})$ and an $R$-$R$-bimodule $M$, the Hochschild homology $\HH(R, M)$ is the geometric realization of the simplicial object $\Delta^{\op} \to \cat{C}$ given by
\[
	Q \mapsto M \otimes R^{\otimes Q}.
\]
More generally, we wish to form the Hochschild homology of an object $(I, F)$ of $\Lambda_{\cat{C}}^{\CycBMod}$. We achieve this by explicitly constructing a simplicial object in $(\CycBMod_I^{\otimes})_{\act}$ for every $I \in \Lambda$. This is a cyclic version of the functor $\mathrm{Cut} : \Delta^{\op} \to \Fin$ considered in Construction~\ref{construction:cut} above. First, recall that the paracyclic category $\Lambda_\infty$ is equivalent to its opposite. There is a functor
\[
	(-)^{\vee} : \Lambda_{\infty}^{\op} \to \Lambda_{\infty}
\]
which sends a morphism $f : I \to J$ in $\Lambda_{\infty}$ to $f^{\vee} : J \to I$ defined by $f^{\vee}(j) = \mathrm{min}\{i \in I \mid f(i) \geq j\}$. This functor is a $\B\Z$-equivariant equivalence, so it descends to an equivalence $(-)^{\vee} : \Lambda^{\op} \to \Lambda$. We construct the cyclic version of the $\mathrm{Cut} : \Delta^{\op} \to \Fin$ functor as follows:

\begin{construction}
The assignment $I \mapsto I/\Z$ defines a functor $\Lambda_{\infty} \to \Fin$ which factors over $\Lambda$. Let $\mathrm{Cut}_{\Lambda} : \Lambda^{\op} \to \Fin$ denote the functor defined by the following composite
\[
	\Lambda^{\op} \xrightarrow{(-)^{\vee}} \Lambda \xrightarrow{\nicefrac{(-)}{\Z}} \Fin.
\]
If $Q$ is a nonempty finite linearly ordered set, $\mathrm{Cut}_{\Lambda}(Q \times \Z) \simeq \mathrm{Cut}(Q)/{\sim}$, where $\sim$ is the equivalence relation which identifies the trivial cuts $(Q, \varnothing) \sim (\varnothing, Q)$. For every object $I$ of $\Lambda$, there is a functor
\[
	(-) \times I : \Delta \to \Lambda
\]
determined by the assignment $Q \mapsto Q \times I$, where the product is an object of $\Lambda$ equipped with the reverse lexicographic ordering and $\Z$-action given by $(q, m) + n = (q, m+n)$. In conclusion, the construction $Q \mapsto \mathrm{Cut}_{\Lambda}(Q \times I)$ determines a functor $\mathrm{Cut}_{\Lambda}(- \times I) : \Delta^{\op} \to \Fin$.  
\end{construction}

\begin{lemma} \label{lemma:cut_refines_cycbmod}
For every object $I$ of $\Lambda$, the functor $\mathrm{Cut}_{\Lambda}(- \times I) : \Delta^{\op} \to \Fin$ refines to a functor
\[
	\mathrm{Cut}_{\Lambda}(- \times I) : \Delta^{\op} \to (\CycBMod_I^{\otimes})_{\act}.
\]
\end{lemma}

\begin{proof}
For every $Q \in \Delta$, we define a map $c_Q : Q \times I \to \{(x, y) \in I \times I \mid x \leq y \leq x+1\}$ by
\[
	c_Q(q, i) = 
	\begin{cases}
		(i, \mathrm{cons}(i)) & \text{if } q = \mathrm{min}(Q) \\
		(i, i) & \text{otherwise}.
	\end{cases}
\]
Consequently, we obtain a map $c_Q : \mathrm{Cut}_{\Lambda}(Q \times I) \to \Path(I)$ since the map above is $\Z$-equivariant. Finally, note that the preimage of an $v_i$-coloured element is coloured by an $v_i$-admissible sequence and that the preimage of an $e_{[i, j]}$-coloured element is coloured by an $e_{[i, j]}$-admissible sequence. Finally, if $f : Q \to Q'$ is a morphism of nonempty linearly ordered sets, then we have to endow the preimages of the map $f \times \id : Q \times I/\Z \to Q' \times I/\Z$ with a linear ordering. Indeed, if $(q', [x])$ is an element of $Q' \times I/\Z$ with lift $(q', x)$ to $Q' \times I$, then the preimage $(f \times \id)^{-1}(q', x)$ carries a linear ordering as a subset of $Q \times I$, and this proves the desired statement. 
\end{proof}

We unwind the construction of $\mathrm{Cut}_{\Lambda}(- \times I)$ in the basic cases where $I = [1]_{\Lambda}$ and $I = [2]_{\Lambda}$. 

\begin{example}
The functor $\mathrm{Cut}_{\Lambda}(- \times [1]_{\Lambda})$ is the simplicial object in $(\CycBMod_1^{\otimes})_{\act}$ given by
\[
	[n] \mapsto \ast_{e_{[0,0]}} \amalg \underbrace{\ast_{v_0} \amalg \cdots \amalg \ast_{v_0}}_{\text{$n$ times}}.
\]
Similarly, the functor $\mathrm{Cut}_{\Lambda}(- \times [2]_{\Lambda})$ is the simplicial object in $(\CycBMod_2^{\otimes})_{\act}$ given by
\[
	[n] \mapsto \ast_{e_{[0,1]}} \amalg \ast_{e_{[1,0]}} \amalg \underbrace{(\ast_{v_0} \amalg \ast_{v_1}) \amalg \cdots \amalg (\ast_{v_0} \amalg \ast_{v_1})}_{\text{$n$ times}}.
\]
\end{example}

For every object $I$ of $\Lambda$, we define Hochschild homology as a functor on $\CycBMod_I$ as follows:

\begin{definition}
Let $\cat{C}$ denote a symmetric monoidal $\infty$-category which admits sifted colimits. For every object $I$ of $\Lambda$, let $\HH_I : \CycBMod_I(\cat{C}) \to \cat{C}$ denote the functor defined by
\[
	\CycBMod_I(\cat{C}) \hookrightarrow \Fun^{\otimes}((\CycBMod_I^{\otimes})_{\act}, \cat{C}) \xrightarrow{\mathrm{Cut}_{\Lambda}(- \times I)^{\ast}} \Fun(\Delta^{\op}, \cat{C}) \xrightarrow{|-|} \cat{C},
\]
where $\mathrm{Cut}_{\Lambda}(- \times I)$ denotes the functor from Lemma~\ref{lemma:cut_refines_cycbmod}. 
\end{definition}

Finally, we prove the main result of this section. 

\begin{theorem} \label{theorem:HH_trace}
Let $\cat{C}$ denote a symmetric monoidal $\infty$-category which admits sifted colimits. The construction $(I, F) \mapsto \HH_I(F)$ determines a functor of $\infty$-categories
\[
	\HH : \Lambda_{\cat{C}}^{\CycBMod} \to \cat{C}
\] 
which satisfies the trace property.  
\end{theorem}

\begin{proof}
We begin by addressing the functoriality of the construction $(I, F) \mapsto \HH_I(F)$. Let $(f, \alpha)$ be a morphism from $(I, F)$ to $(J, G)$ in $\Lambda_{\cat{C}}^{\CycBMod}$, which means that $f : J \to I$ is a morphism in $\Lambda$, and $\alpha : \CycBMod_f(F) \to G$ is a natural transformation in $\CycBMod_J(\cat{C})$. For every $Q \in \Delta$, the map $f$ induces a map $\mathrm{Cut}_{\Lambda}(Q \times I) \to \mathrm{Cut}_{\Lambda}(Q \times J)$, which refines to a map
\[
	\mathrm{Cut}_{\Lambda}(Q \times I) \to f_\ast \mathrm{Cut}_{\Lambda}(Q \times J)	
\]
in $(\CycBMod_I^{\otimes})_{\act}$, where $f_\ast$ denotes the functor introduced in Construction~\ref{construction:functor_CycBMod}. Furthermore, if $Q \to Q'$ is a morphism in $\Delta$, then we obtain a commutative diagram
\[
\begin{tikzcd}
	\mathrm{Cut}_{\Lambda}(Q' \times I) \arrow{r} \arrow{d} & f_\ast \mathrm{Cut}_{\Lambda}(Q' \times J) \arrow{d} \\
	\mathrm{Cut}_{\Lambda}(Q \times I) \arrow{r} & f_\ast \mathrm{Cut}_{\Lambda}(Q \times J)
\end{tikzcd}
\]
in $(\CycBMod_I^{\otimes})_{\act}$. Consequently, the morphism $(f, \alpha)$ induces a canonical map
\[
	\HH_I(F) \to \HH_J(\CycBMod_f(F)) \xrightarrow{\HH_J(\alpha)} \HH_J(G),
\]
which is compatible with composition by naturality, so the construction $(I, F) \mapsto \HH_I(F)$ refines to a functor on $\Lambda_{\cat{C}}^{\CycBMod}$. We verify that $\HH$ satisfies the trace property. Since the rotations are equivalences in $\Lambda^{\CycBMod}$ and the insertions are one sided inverses to contractions in $\Lambda^{\CycBMod}$, it suffices to prove that $([2]_{\Lambda}, (R, M; S, N)) \to ([1]_{\Lambda}, (R, M \otimes_S N))$ induces an equivalence
\[
	\HH_{[2]_{\Lambda}}(R, M; S, N) \to \HH_{[1]_{\Lambda}}(R, M \otimes_S N).
\]
The object $\HH(R, M \otimes_S N)$ is the colimit of the bisimplicial object $\Delta^{\op} \times \Delta^{\op} \to \cat{C}$ given by
\[
	(Q, Q') \mapsto (M \otimes S^{\otimes Q} N) \otimes R^{\otimes Q'},
\]
where we have used that $M \otimes_S N$ is computed by the two-sided bar construction, and the colimit of the diagonal simplicial object $\Delta^{\op} \to \Delta^{\op} \times \Delta^{\op} \to \cat{C}$ is equivalent to $\HH_{[2]_{\Lambda}}(R, M; S, N)$. This proves the desired statement as the diagonal of $\Delta^{\op}$ is cofinal since $\Delta^{\op}$ is a sifted $\infty$-category.
\end{proof}

As a consequence of Theorem~\ref{theorem:HH_trace}, we obtain the following result:

\begin{corollary} \label{corollary:c_n_action}
Let $(R, M)$ denote a pair consisting of an $\E_1$-ring $R$ and an $R$-bimodule $M$. The spectrum $\THH(R, M^{\otimes_R n})$ admits a canonical $C_n$-action for every $n \in \N$. 
\end{corollary}

\begin{proof}
By virtue of the trace property, there is an equivalence of spectra
\[
	\THH(R, M^{\otimes_R n}) \simeq \THH(R, M; \ldots; R, M),
\]
and $([n]_{\Lambda}, (R, M; \ldots; R, M)$ carries a $C_n$-action by rotation regarded as an object of $\Lambda^{\CycBMod}$.
\end{proof}

Consequently, to finish the construction of $\THH(R, M)$ as an object of the $\infty$-category $\PgcSp$, it remains to construct a $C_n$-equivariant map of spectra
\[
	\varphi_{p, n} : \THH(R, M^{\otimes_R n}) \to \THH(R, M^{\otimes_R pn})^{\tate C_p}
\]
for every prime $p$ and $n \in \N$, where the Tate construction carries the residual $C_{pn}/C_p \simeq C_n$-action. The main ingredient is the extended functoriality of the $p$-fold Tate diagonal on the category of finite free $C_p$-sets as discussed in~\cite[\textsection III.3]{NS18}. 

\begin{notation}
Let $\Free(C_p)$ denote the category of finite free $C_p$-sets and $C_p$-equivariant maps. The assignment $T \mapsto \overline{T} = T/C_p$ defines a functor $\Free(C_p) \to \Fin$. 
\end{notation}

We will begin by recalling a result of Nikolaus--Scholze adressing the functoriality of the Tate diagonal. This result appears as~\cite[Proposition III.3.6]{NS18} combined with~\cite[Lemma III.3.7 and Corollary III.3.8]{NS18}. 

\begin{theorem} \label{theorem:NS_result_tate_diagonal}
Let $p$ denote a prime number. 
\begin{enumerate}[leftmargin=2em, topsep=5pt, itemsep=5pt]
	\item There is a natural $\B C_p$-equivariant symmetric monoidal functor
	\[
		I_p : \Free(C_p) \times_{\Fin} \Sp_{\act}^{\otimes} \to \Sp
	\]
	determined by the construction $(T, \{X_{\overline{t}}\}_{\overline{t} \in \overline{T}}) \mapsto \bigotimes_{\overline{t} \in \overline{T}} X_{\overline{t}}$.

	\item There is a natural $\B C_p$-equivariant lax symmetric monoidal functor
	\[
		\tilde{T}_p : \Free(C_p) \times_{\Fin} \Sp_{\act}^{\otimes} \to (\Sp_{\act}^{\otimes})^{\B C_p} \xrightarrow{\otimes} \Sp^{\B C_p} \xrightarrow{(-)^{\tate C_p}} \Sp,
	\]
	where the first functor is determined by the construction $(T, \{X_{\overline{t}}\}_{\overline{t} \in \overline{T}}) \mapsto (T, \{X_{\overline{t}}\}_{t \in T})$\footnote{In fact, the first functor is symmetric monoidal while the functors $\otimes$ and $(-)^{\tate C_p}$ are lax symmetric monoidal.}.

	\item There is an essentially unique $\B C_p$-equivariant lax symmetric monoidal transformation 
	\[
		\tilde{\Delta}_p : I_p \to \tilde{T}_p,
	\]
	where both $I_p$ and $\tilde{T}_p$ are regarded as functors $\Free(C_p) \times_{\Fin} \Sp_{\act}^{\otimes} \to \Sp$ as above. 
  \end{enumerate}
\end{theorem}

Using the notation from Theorem~\ref{theorem:NS_result_tate_diagonal}, the $p$-fold Tate diagonal is given by a map of spectra
\[
	\bigotimes_{\overline{t} \in \overline{T}} X_{\overline{t}} \to \Big( \bigotimes_{t \in T} X_{\overline{t}} \Big)^{\tate C_p}
\]
for every $\overline{T}$-indexed collection of spectra $\{X_{\overline{t}}\}_{\overline{t} \in \overline{T}}$. The idea is to refine our functor $\mathrm{Cut}_{\Lambda}(- \times I)$ to a functor with values in $\Free(C_p) \times_{\Fin} (\CycBMod_I^{\otimes})_{\act}$: 

\begin{construction} \label{construction:refined_cut}
Let $I$ denote an object of $\Lambda$. Let $pI$ denote the parasimplex with the same underlying linearly ordered set as $I$ but whose $\Z$-action is given by restricting the $\Z$-action on $I$ along the homomorphism $p : \Z \to \Z$ given by multiplication by $p$. For instance, if $I = [n]_{\Lambda}$, then $pI = [pn]_{\Lambda}$. Note that the construction $Q \mapsto \mathrm{Cut}_{\Lambda}(Q \times pI)$ determines a functor 
\[
	\Delta^{\op} \to \Free(C_p),
\]
where the $C_p$-action on $\mathrm{Cut}_{\Lambda}(Q \times pI) \simeq Q \times \mathrm{Cut}_{\Lambda}(pI)$ is given by the trivial $C_p$-action on $Q$ and on $\mathrm{Cut}_{\Lambda}(pI)$ by sending an element to its consecutive element. The following diagram commutes
\[
\begin{tikzcd}
	\Delta^{\op} \arrow{r}{\mathrm{Cut}_{\Lambda}(- \times I)} \arrow[swap]{d}{\mathrm{Cut}_{\Lambda}(- \times pI)} & (\CycBMod_I^{\otimes})_{\act} \arrow{d} \\
	\Free(C_p) \arrow{r}{-/C_p} & \Fin 
\end{tikzcd}
\]
so we obtain a canonical functor of $\infty$-categories $\Delta^{\op} \to \Free(C_p) \times_{\Fin} (\CycBMod_I^{\otimes})_{\act}$.
\end{construction}

Finally, we prove the main result of this section:

\begin{theorem} \label{theorem:THH_polygonic}
Let $R$ denote an $\E_1$-ring. The topological Hochschild homology functor
\[
	\THH(R, -) : \BMod_R \to \Sp
\]
determined by the construction $M \mapsto \THH(R, M)$ canonically refines to a functor of $\infty$-categories
\[
	\THH(R, -) : \BMod_R \to \PgcSp.
\]
\end{theorem}

\begin{proof}
Recall that $\THH(R, M^{\otimes_R n})$ admits a $C_n$-action for every $n \in \N$ by virtue of Corollary~\ref{corollary:c_n_action}. For simplicity, we explain the construction of a map of spectra
\[
	\THH(R, M) \to \THH(R, M^{\otimes_R p})^{\tate C_p}.
\]
Regarding the pair $(R, M)$ as an object of $\CycBMod_1(\Sp)$, we obtain a functor of $\infty$-categories
\[
	\Delta^{\op} \to \Free(C_p) \times_{\Fin} (\CycBMod_1^{\otimes})_{\act} \to \Free(C_p) \times_{\Fin} \Sp_{\act}^{\otimes}
\]
by virtue of Construction~\ref{construction:refined_cut}. Theorem~\ref{theorem:NS_result_tate_diagonal} supplies a map of spectra from $\THH(R, M)$ to the geometric realization of the simplicial spectrum which we informally depict by
\[
\begin{tikzcd}
	\cdots \arrow[yshift=0.5ex]{r} \arrow[yshift=-0.5ex]{r} \arrow[yshift=1.5ex]{r} \arrow[yshift=-1.5ex]{r} & (M \otimes R \otimes R \otimes \cdots \otimes M \otimes R \otimes R)^{\tate C_p} \arrow{r} \arrow[yshift=1ex]{r} \arrow[yshift=-1ex]{r} & (M \otimes R \otimes \cdots \otimes M \otimes R)^{\tate C_p} \arrow[yshift=0.5ex]{r} \arrow[yshift=-0.5ex]{r} & (M^{\otimes p})^{\tate C_p}.
\end{tikzcd}
\]
Rearranging terms and using the multiplication on $R$, we obtain the desired map of spectra
\[
	\THH(R, M) \to \THH(R, M^{\otimes_R p})^{\tate C_p}
\]
which proves the desired statement.
\end{proof}

\begin{remark}
The polygonic structure on $\THH(R, M)$ can also be obtained using~\cite{Law21}.
\end{remark}

We construct a Teichm{\"u}ller map on TR using Proposition~\ref{proposition:TR_refines_qfgenZ} and the preceeding discussion.

\begin{construction} \label{construction:teichmuller}
For every $\E_1$-ring $R$ and $R$-$R$-bimodule $M$, we construct a map
\[
	\tau : \Sigma^\infty_+\underline{M} \to \underline{\TR}(R, M)
\]
where $\underline{M}$ is the quasifinitely genuine $\Z$-space with $\underline{M}^{n\Z} \simeq \Omega^\infty(M^{\times n})$. By Proposition~\ref{proposition:TR_refines_qfgenZ}, such a map of quasifinitely genuine $\Z$-spectra is adjoint to a map of polygonic spectra
\[
	\mathrm{L}\Sigma^\infty_+ \underline{M} \to \THH(R, M).
\]
By a similar analysis as above, there is a canonical $C_n$-equivariant map of spectra
\[
	M^{\times n} \to \THH(R, M^{\otimes_R n}),
\]
which then induces the desired $C_n$-equivariant map
\[
	(\mathrm{L}\Sigma^\infty_+\underline{M})_n = (\Sigma^\infty_+ \underline{M})^{\Phi n\Z} \simeq \Sigma^\infty_+ \Omega^\infty(M^{\times n}) \to M^{\times n} \to \THH(R, M^{\otimes_R n}),
\]
where the latter map is the counit of the adjunction $\Sigma^\infty_+ \dashv \Omega^\infty$. 
\end{construction}


\bibliographystyle{abbrv}
\bibliography{polygonic} 

\begin{thebibliography}{10}

\bibitem{AN20}
B.~Antieau and T.~Nikolaus.
\newblock Cartier modules and cyclotomic spectra.
\newblock {\em Journal of the American Mathematical Society}, 34(1):1--78,
  2021.

\bibitem{Bar17}
C.~Barwick.
\newblock Spectral {M}ackey functors and equivariant algebraic {K}-theory
  ({I}).
\newblock {\em Advances in Mathematics}, 304:646--727, 2017.

\bibitem{BGS19}
C.~Barwick, S.~Glasman, and J.~Shah.
\newblock Spectral {M}ackey functors and equivariant algebraic {K}-theory,
  {II}.
\newblock {\em Tunisian Journal of Mathematics}, 2(1):97--146, 2019.

\bibitem{Ber22}
J.~Berman.
\newblock {THH} and traces of enriched categories.
\newblock {\em International Mathematics Research Notices}, 2022(4):3074--3105,
  2022.

\bibitem{BS05}
S.~Betley and C.~Schlichtkrull.
\newblock The cyclotomic trace and curves on {K}-theory.
\newblock {\em Topology}, 44(4):845--874, 2005.

\bibitem{BM16}
A.~Blumberg and M.~Mandell.
\newblock The homotopy theory of cyclotomic spectra.
\newblock {\em Geometry \& Topology}, 19(6):3105--3147, 2016.

\bibitem{Con83}
A.~Connes.
\newblock Cohomologie cyclique et foncteurs ext$^n$.
\newblock {\em CR Acad. Sci. Paris S{\'e}r. I Math}, 296(23):953--958, 1983.

\bibitem{DKNP22b}
E.~Dotto, A.~Krause, T.~Nikolaus, and I.~Patchkoria.
\newblock Witt vectors with coefficients and the components of the norm.
\newblock In preparation.

\bibitem{DKNP20}
E.~Dotto, A.~Krause, T.~Nikolaus, and I.~Patchkoria.
\newblock Witt vectors with coefficients and characteristic polynomials over
  non-commutative rings.
\newblock {\em arXiv:2002.01538}, 2020.

\bibitem{DS88}
A.~Dress and C.~Siebeneicher.
\newblock The {B}urnside ring of profinite groups and the {W}itt vector
  construction.
\newblock {\em Advances in Mathematics}, 70(1):87--132, 1988.

\bibitem{GGN16}
D.~Gepner, M.~Groth, and T.~Nikolaus.
\newblock Universality of multiplicative infinite loop space machines.
\newblock {\em Algebraic \& Geometric Topology}, 15(6):3107--3153, 2016.

\bibitem{GM11}
B.~Guillou and J.~May.
\newblock Models of {$G$}-spectra as presheaves of spectra.
\newblock {\em arXiv:1110.3571}, 2011.

\bibitem{Hes96}
L.~Hesselholt.
\newblock On the $p$-typical curves in {Q}uillen's {K}-theory.
\newblock {\em Acta Mathematica}, 177(1):1--53, 1996.

\bibitem{HM97}
L.~Hesselholt and I.~Madsen.
\newblock On the {K}-theory of finite algebras over {W}itt vectors of perfect
  fields.
\newblock {\em Topology}, 36(1):29--101, 1997.

\bibitem{HS19}
L.~Hesselholt and P.~Scholze.
\newblock Arbeitsgemeinschaft: Topological {C}yclic {H}omology.
\newblock {\em Oberwolfach Reports}, 15(2):805--940, 2019.

\bibitem{Kal14}
D.~Kaledin.
\newblock Mackey profunctors.
\newblock {\em arXiv:1412.3248}, 2014.

\bibitem{Kal15}
D.~Kaledin.
\newblock Trace theories and localization.
\newblock {\em Stacks and categories in geometry, topology, and algebra},
  643:227--262, 2015.

\bibitem{Law21}
T.~Lawson.
\newblock Unwinding the relative {T}ate diagonal.
\newblock {\em Journal of Topology}, 14(2):674--699, 2021.

\bibitem{LM12}
A.~Lindenstrauss and R.~McCarthy.
\newblock On the {T}aylor tower of relative {K}-theory.
\newblock {\em Geometry \& Topology}, 16(2):685--750, 2012.

\bibitem{Lur09}
J.~Lurie.
\newblock {\em Higher topos theory}.
\newblock Princeton University Press, 2009.

\bibitem{Lur17}
J.~Lurie.
\newblock Higher algebra.
\newblock Available on the author's website, 2017.

\bibitem{MM02}
M.~Mandell and J.~May.
\newblock {\em Equivariant orthogonal spectra and {$S$}-modules}.
\newblock American Mathematical Soc., 2002.

\bibitem{MNN17}
A.~Mathew, N.~Naumann, and J.~Noel.
\newblock Nilpotence and descent in equivariant stable homotopy theory.
\newblock {\em Advances in Mathematics}, 305:994--1084, 2017.

\bibitem{McC21}
J.~McCandless.
\newblock On curves in $\mathrm{K}$-theory and $\mathrm{TR}$.
\newblock {\em arXiv:2102.08281}, 2021.

\bibitem{Nar16}
D.~Nardin.
\newblock Parametrized higher category theory and higher algebra: Expos{\'e} iv
  -- {S}tability with respect to an orbital $\infty$-category.
\newblock {\em arXiv:1608.07704}, 2016.

\bibitem{NS18}
T.~Nikolaus and P.~Scholze.
\newblock On topological cyclic homology.
\newblock {\em Acta Mathematica}, 221(2):203--409, 2018.

\bibitem{Sch20}
S.~Schwede.
\newblock Lectures on equivariant stable homotopy theory.
\newblock Available on the author's website.

\end{thebibliography}

\end{document}